\documentclass[11pt]{amsart}

\usepackage{subfiles}
\usepackage{comment}
\usepackage{float}
\usepackage{marginnote}
\usepackage{tabu}
\usepackage{euscript}
\usepackage[dvipsnames]{xcolor}   		             		
\usepackage{graphicx}			
\usepackage{amssymb}
\usepackage{mathrsfs}
\usepackage{amsthm}
\usepackage{amsmath}
\usepackage{stmaryrd}
\usepackage{tikz}
\usepackage{tikz-cd}
\usepackage{accents}
\usepackage{upgreek}
\usepackage{enumerate}
\usepackage{bm}
\usepackage{mathtools}
\usepackage[all]{xy}
\usepackage{caption}
\usepackage{url}
\usepackage{float}
\usepackage{todonotes} 
\usepackage{colonequals}
\usepackage{bbm}
\usepackage{longtable}
\usepackage[full]{textcomp}
\usepackage[cal=cm]{mathalfa}
\usepackage{xparse}
\usepackage{comment}
\usepackage[noadjust]{cite} 

\usepackage{faktor} 
\usepackage{xfrac} 

\usetikzlibrary{calc}
\usetikzlibrary{fadings}
\usetikzlibrary{decorations.pathmorphing}
\usetikzlibrary{decorations.pathreplacing}
\usepackage{tikz,tikz-cd,tikz-3dplot}
\usepackage{pgfplots}
\usetikzlibrary{arrows,shadows,positioning, calc, decorations.markings, 
hobby,quotes,angles,decorations.pathreplacing,intersections,shapes}
\usepgflibrary{shapes.geometric}
\usetikzlibrary{fillbetween,backgrounds}
\usetikzlibrary{arrows.meta} 

\usepackage[margin=1.05in]{geometry}

\tikzset{
  commutative diagrams/.cd, 
  arrow style=tikz, 
  diagrams={>=stealth}
}
\tikzset{
  arrow/.pic={\path[tips,every arrow/.try,->,>=#1] (0,0) -- +(0,4pt);},
  pics/arrow/.default={triangle 90}
}
\tikzset{->-/.style={decoration={
  markings,
  mark=at position .6 with {\arrow{latex}}},postaction={decorate}}
  }
\tikzset{
  c/.style={every coordinate/.try}
}

\theoremstyle{definition}
\newtheorem{innercustomthm}{Theorem}
\newenvironment{customthm}[1]
  {\renewcommand\theinnercustomthm{#1}\innercustomthm}
  {\endinnercustomthm}

\theoremstyle{definition}
\newtheorem{innercustomcor}{Corollary}
\newenvironment{customcor}[1]
  {\renewcommand\theinnercustomcor{#1}\innercustomcor}
  {\endinnercustomcor}

  \theoremstyle{definition}
\newtheorem{innercustomprop}{Proposition}
\newenvironment{customprop}[1]
  {\renewcommand\theinnercustomprop{#1}\innercustomprop}
  {\endinnercustomprop}
  
\makeatletter
\def\@tocline#1#2#3#4#5#6#7{\relax
  \ifnum #1>\c@tocdepth 
  \else
    \par \addpenalty\@secpenalty\addvspace{#2}%
    \begingroup \hyphenpenalty\@M
    \@ifempty{#4}{%
      \@tempdima\csname r@tocindent\number#1\endcsname\relax
    }{%
      \@tempdima#4\relax
    }%
    \parindent\z@ \leftskip#3\relax \advance\leftskip\@tempdima\relax
    \rightskip\@pnumwidth plus4em \parfillskip-\@pnumwidth
    #5\leavevmode\hskip-\@tempdima
      \ifcase #1
       \or\or \hskip 1em \or \hskip 2em \else \hskip 3em \fi%
      #6\nobreak\relax
    \dotfill\hbox to\@pnumwidth{\@tocpagenum{#7}}\par
    \nobreak
    \endgroup
  \fi}
\makeatother

\newcounter{marginnote}
\DeclareMathAlphabet{\mathpzc}{OT1}{pzc}{m}{it}

\usepackage[backref=page]{hyperref}
\hypersetup{
  colorlinks   = true,          
  urlcolor     = blue!65!black,          
  linkcolor    = blue!65!black,          
  citecolor   = blue             
}

\pgfplotsset{compat=1.18}

\theoremstyle{definition}
\newtheorem{theorem}{Theorem}[section]

\newtheorem*{claim*}{Claim}

\newtheorem{corollary}[theorem]{Corollary}

\newtheorem{lemma}[theorem]{Lemma}
\newtheorem{proposition}[theorem]{Proposition}
\newtheorem{remark}[theorem]{Remark}

\newtheorem*{runningexample*}{Running example}

\newtheorem*{aside*}{Aside}

\newtheorem{definition}[theorem]{Definition}
\newtheorem{example}[theorem]{Example}

\newtheorem{notation}[theorem]{Notation}
\newtheorem*{notation*}{Notation} 
\newtheorem{proposition-definition}[theorem]{Proposition-Definition}
\newtheorem{theorem-definition}[theorem]{Theorem-Definition} 

\newtheorem*{question*}{Question}

\newcommand{\bcd}{\begin{center}\begin{tikzcd}}
\newcommand{\ecd}{\end{tikzcd}\end{center}}

\newcommand{\calC}{\mathcal{C}}

\newcommand{\calQ}{\mathcal{Q}}

\newcommand{\calP}{\mathcal{P}} 
\newcommand{\calR}{\mathcal{R}} 

\newcommand{\raag}{A_\Gamma} 
\newcommand{\sal}{\mathbb{S}_{\Gamma}} 
\newcommand{\uag}{U(A_{\Gamma})} 
\newcommand{\oag}{\operatorname{Out}(A_{\Gamma})} 
\newcommand{\vcduag}{\textsc{vcd}(U(A_{\Gamma}))} 
\newcommand{\link}[1]{\operatorname{lk}({#1})} 
\newcommand{\str}[1]{\operatorname{st}({#1})} 
\newcommand{\splt}[1]{\operatorname{split}({#1})} 
\newcommand{\mx}[1]{\operatorname{max}({#1})}
\newcommand{\salb}[1]{\mathbb{S}_\Gamma^{#1}} 
\newcommand{\spine}{K_{\Gamma}} 
\newcommand{\os}{\mathcal{O}_\Gamma} 
\newcommand{\res}{\widehat{K_\Gamma}} 

\newcommand{\ssslash}{\smash{\sslash}} 

\begin{document} 

\title{Realising VCD for untwisted automorphism groups of RAAGs}

\author{Gabriel Corrigan}
\address{School of Mathematics and Statistics, University of Glasgow, Glasgow, G12 8QQ, UK}
\email{g.corrigan.1@research.gla.ac.uk}
\urladdr{https://www.gabrielcorrigan.com}

\begin{abstract}
    The virtual cohomological dimension of~$\operatorname{Out}(F_n)$ is given precisely by the dimension of the spine of Culler--Vogtmann Outer space. However, the dimension of the spine of untwisted Outer space for a general right-angled Artin group~$A_\Gamma$ does not necessarily match the virtual cohomological dimension~$\textsc{vcd}(U(A_{\Gamma}))$ of the untwisted subgroup~$U(A_\Gamma) \leq \operatorname{Out}(A_\Gamma)$. Under certain graph-theoretic conditions, we perform an equivariant deformation retraction of this spine to produce a new contractible cube complex upon which~$U(A_\Gamma)$ acts properly and cocompactly. Furthermore, we give conditions for when the dimension of this complex realises the virtual cohomological dimension of~$U(A_\Gamma)$. We finish with two applications of our construction; in particular we show that the difference between the dimension of the untwisted spine and~$\textsc{vcd}(U(A_{\Gamma}))$ can be arbitrarily large.
\end{abstract}

\maketitle

\tableofcontents

\section{Introduction}\label{sec:intro}

\subsection{Overview of results}\label{subsec:overview of results}

Given a finite simplicial graph~$\Gamma$, define the associated \emph{right-angled Artin group}~$\raag$ by the presentation
\[\raag \coloneqq \left\langle V(\Gamma) \;\;\vert\;\; [a, b] = 1 \; \text{if} \; \{a, b\} \in E(\Gamma) \right\rangle.\]
Hence, when~$\Gamma$ has no edges, the corresponding right-angled Artin group is simply the free group of rank~$|V(\Gamma)|$, and when~$\Gamma$ is a complete graph,~$\raag \cong \smash{\mathbb{Z}^{|V(\Gamma)|}}$.

The automorphism groups of right-angled Artin groups are well-studied. We will be interested in the so-called \emph{untwisted subgroup}~$U(A_\Gamma) \leq \oag$ of outer automorphisms, which is obtained by excluding certain automorphisms called \emph{twists}. 

For any right-angled Artin group~$\raag$, Charney--Stambaugh--Vogtmann \cite{CharneyStambaughVogtmannUntwistedOuterSpace17} constructed an \emph{untwisted Outer space}~$\mathcal{O}_\Gamma^U$, which is a contractible complex with a proper~$\uag$-action. This is a generalisation of the influential \emph{Culler--Vogtmann Outer space}~$CV_n$, introduced in 1986 \cite{CullerVogtmannOuterSpace86}. Untwisted Outer space has a \emph{spine}~$\spine$ -- a deformation retract  of~$\mathcal{O}_\Gamma^U$ which naturally has the structure of a cube complex and which has a proper and cocompact~$\uag$-action. This implies that~$\dim\left(\spine\right)$ is an upper bound for the virtual cohomological dimension (\textsc{vcd}) of~$\uag$.

In many cases, equality holds. Millard and Vogtmann \cite{MillardVogtmannCubeComplexes21} show that the \emph{principal rank} (see \textit{Definition \ref{def:principal rank}}) of~$\Gamma$ is a lower bound for~$\vcduag$. However, there exist simplicial graphs~$\Gamma$ for which the principal rank is strictly less than~$\dim(\spine)$. For example, the \emph{2-rake} graph~$T_2$ shown in \textit{Figure \ref{fig:2-rake (intro)}} has principal rank equal to 5 but~$\dim(K_{T_2}) = 6$.

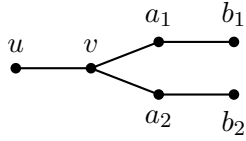
\begin{figure}[H]
    \begin{center}
        \begin{tikzpicture}
            \fill (-1, 0) circle (2pt);
            \fill (0, 0) circle (2pt);
            \fill (0.9, 0.35) circle (2pt);
            \fill (0.9, -0.35) circle (2pt);
            \fill (1.9, 0.35) circle (2pt);
            \fill (1.9, -0.35) circle (2pt);
            \draw[black, thick] (-1, 0) -- (0, 0);
            \draw[black, thick] (0, 0) -- (0.9, 0.35);
            \draw[black, thick] (0, 0) -- (0.9, -0.35);
            \draw[black, thick] (0.9, 0.35) -- (1.9, 0.35);
            \draw[black, thick] (0.9, -0.35) -- (1.9, -0.35);
            \node at (-1, 0) [above = 2pt, black, thick]{$u$};
            \node at (0, 0) [above = 2pt, black, thick]{$v$};
            \node at (0.9, 0.35) [above = 2pt, black, thick]{$a_1$};
            \node at (0.9, -0.35) [below = 2pt, black, thick]{$a_2$};
            \node at (1.9, 0.35) [above = 2pt, black, thick]{$b_1$};
            \node at (1.9, -0.35) [below = 2pt, black, thick]{$b_2$};
        \end{tikzpicture}
        \caption{The 2-rake~$T_2$.}
        \label{fig:2-rake (intro)}
    \end{center}
\end{figure}

Under certain graph-theoretic conditions we call \emph{spiky} and \emph{barbed} (defined in \textit{Definitions \ref{def:spiky}} \& \textit{\ref{def:barbed}}), we determine~$\vcduag$ geometrically.

\begin{customthm}{A}[\ref{cor:conditions 1 & 2 + barbed implies vcd = M(L)}]\label{customthm:main thm}
    Suppose that~$\Gamma$ is spiky and barbed. Then~$\vcduag$ is equal to the principal rank of~$\Gamma$. Moreover, there is a~$\uag$-complex realising this virtual cohomological dimension.
\end{customthm}

We also present a collection of infinite families of graphs generalising the 2-rake graph in \emph{Figure \ref{fig:2-rake (intro)}}. We show that these graphs satisfy our conditions, and determine the values of the principal rank and the dimension of the untwisted spine~$\spine$ for these graphs. These calculations, combined with \emph{Theorem \ref{customthm:main thm}}, yield the following.

\begin{customcor}{B}[\ref{cor:arbitrarily large gaps}]\label{customcor:arbitrary gaps}
    The difference between~$\dim(\spine)$ and~$\vcduag$ can be arbitrarily large.
\end{customcor}

We also give another application of \emph{Theorem \ref{customthm:main thm}}, by applying its proof to the following graph~$\Delta$:

\begin{figure}[H]
    \begin{center}
        \begin{tikzpicture}
            \fill (-1, 0) circle (2pt); 
            \fill (-1, -1) circle (2pt); 
            \fill (0, 0) circle (2pt); 
            \fill (0, -1) circle (2pt); 
            \fill (0.9, 0.5) circle (2pt); 
            \fill (0.9, -0.5) circle (2pt); 
            \fill (0.9, -1.5) circle (2pt); 
            \fill (1.9, 0.5) circle (2pt); 
            \fill (1.9, -0.5) circle (2pt); 
            \fill (1.9, -1.5) circle (2pt); 
            \draw[black, thick] (-1, 0) -- (0, 0); 
            \draw[black, thick] (-1, -1) -- (0, -1); 
            \draw[black, thick] (0, 0) -- (0.9, 0.5); 
            \draw[black, thick] (0, 0) -- (0.9, -0.5); 
            \draw[black, thick] (0, -1) -- (0.9, -0.5); 
            \draw[black, thick] (0, -1) -- (0.9, -1.5); 
            \draw[black, thick] (0.9, 0.5) -- (1.9, 0.5); 
            \draw[black, thick] (0.9, -0.5) -- (1.9, -0.5); 
            \draw[black, thick] (0.9, -1.5) -- (1.9, -1.5); 
            \node at (-1, 0) [above = 2pt, black, thick]{$u_1$};
            \node at (0, 0) [above = 2pt, black, thick]{$v_1$};
            \node at (-1, -1) [above = 2pt, black, thick]{$u_2$};
            \node at (0, -1) [above = 2pt, black, thick]{$v_2$};
            \node at (0.9, 0.5) [above = 2pt, black, thick]{$a_1$};
            \node at (0.9, -0.5) [below = 2pt, black, thick]{$a_2$};
            \node at (0.9, -1.5) [below = 2pt, black, thick]{$a_3$};
            \node at (1.9, 0.5) [above = 2pt, black, thick]{$b_1$};
            \node at (1.9, -0.5) [below = 0pt, black, thick]{$b_2$};
            \node at (1.9, -1.5) [below = 0pt, black, thick]{$b_3$};
        \end{tikzpicture}
    \end{center}
    \caption{The graph~$\Delta$.}
    \label{fig:delta (intro)}
\end{figure}
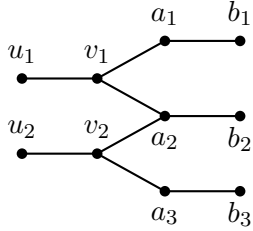

In \cite{MillardVogtmannCubeComplexes21}, Millard and Vogtmann give a sufficient condition for the principal rank to be equal to~$\dim(\spine)$ (and hence for~$\vcduag = \dim(\spine)$). This graph-theoretic condition roughly says that all the vertices of~$\Gamma$ which `dominate' another vertex~$u$ lie in the same component of~$\Gamma \setminus \str{u}$ (here the \emph{star} of~$u$,~$\str{u}$, is the full subgraph on~$u$ and all its neighbours; by `dominate' we refer to the relation~$<_\circ$ -- see \S\ref{subsec:graph-theoretic preliminaries}). One may wonder whether this hypothesis can be weakened to give a sufficient condition guaranteeing that the gap between~$\vcduag$ and~$\dim\left(\spine\right)$ is at most~$k \geq 1$. Arguably the most natural guess for such a condition is the following one, which we denote~$P(k)$: the vertices `dominating'~$u$ must be divided between at most~$(k+1)$ components of~$\Gamma \setminus \str{u}$. Indeed, our prototypical examples of graphs which satisfy~$P(k)$ (see \S\ref{subsec:rakes}) indeed have~$\dim\left(\spine\right) - \vcduag$ equal to exactly~$k$. However, we dispel this speculation by applying (the proof of) \emph{Theorem \ref{customthm:main thm}} to~$\Delta$, to obtain:

\begin{customcor}{C}[\ref{cor:counterexample to suff condition hoped-theorem}]\label{customcor:counterexample}
    There is a graph~$\Delta$ which satisfies~$P(1)$, and has~$\dim(K_\Delta) - \textsc{vcd}(U(A_\Delta)) > 1$.
\end{customcor}

\subsection{Strategy of proofs}\label{subsec:strategy of proof}

The spine~$\spine$ may be viewed as a geometric realisation of a poset of \emph{compatible} sets of so-called \emph{$\Gamma$-Whitehead partitions} (see \textit{Definitions \ref{def:Gamma-Whitehead partition}} \& \textit{\ref{def:compatibility}}) along with some marking data. In particular,~$\dim(\spine)$ is the size of the largest possible such set. The principal rank is also computed using~$\Gamma$-Whitehead partitions.

We introduce the idea of \emph{1-} and \emph{2-hugged}~$\Gamma$-partitions (see \textit{Definition \ref{def:2-hugged in Pi}}), which roughly captures a notion of `redundancy' among the cubes of~$\spine$.  Most of the work of this paper goes into the following pair of statements.

\begin{customthm}{D}[\S\ref{sec:retraction process}, \ref{thm:barbed implies cubes of dimension higher than principal rank are redundant}]\label{customthm:main work}
    \begin{enumerate}[(i)]
        \item Consider the subcomplex~$\res$ of~$\spine$ consisting only of cubes which are not redundant. If~$\Gamma$ is spiky, then~$\smash{\res}$ can be obtained as a~$\uag$-equivariant deformation retraction of~$\spine$. 

        \item Suppose that the principal rank of~$\Gamma$ is strictly less than~$\dim(\spine)$. If~$\Gamma$ is barbed, then every cube in~$\spine$ of dimension greater than the principal rank is redundant.
    \end{enumerate}
\end{customthm}

Combining these statements, we obtain that if~$\Gamma$ is both spiky and barbed, then~$\res$ is a contractible complex with a proper and cocompact~$\uag$-action and is of dimension at most the principal rank of~$\Gamma$. This yields \textit{Theorem \ref{customthm:main thm}}.

Additionally, we use a good understanding of the~$\Gamma$-Whitehead partitions of \emph{rake} graphs to compute their principal rank and the dimension of their untwisted spine:
\begin{customprop}{E}[\ref{prop:M(L), M(V) for rake graphs}]\label{customprop:dim(spine) and M(L) for rake graphs}
    The~$d$-rake graph~$T_d$ (see Figure \ref{fig:the d-rake}) has principal rank~$3d-1$. The dimension of the spine~$K_{T_d}$ is~$4d-2$.
\end{customprop}

The rake graphs are both spiky and barbed, so \textit{Theorem \ref{customthm:main thm}} applies. Hence, we find that the virtual cohomological dimension~$\textsc{vcd}(U(A_{T_d}))$ is equal to~$3d-1$, which gives \textit{Corollary \ref{customcor:arbitrary gaps}}. Certain derivatives of `rake' graphs also exhibit this behaviour. Similarly, we use a good understanding of the possible~$\Delta$-Whitehead partitions to obtain \textit{Corollary \ref{customcor:counterexample}}.

\subsection{Context}\label{subsec:overview of context}

The untwisted subgroup~$\uag \leq \oag$ is obtained by excluding \emph{twists}. Twists are automorphisms which multiply a generator by another generator with which it commutes, so the untwisted subgroup is most closely related to the `free group end' of the spectrum of right-angled Artin groups -- where there are no commuting generators. In particular,~$U(F_n)$ is the entire group of outer automorphisms~$\operatorname{Out}(F_n)$. The abelianisation~$\raag \to \mathbb{Z}^n$ induces a map~$\oag \to \operatorname{GL}_n(\mathbb{Z})$ whose kernel is always contained in~$\uag$ (this kernel was studied in more detail by Toinet,  \cite{ToinetConjugacySeparabilityRAAGs13}).

The spine for the free group~$F_n$ has dimension~$2n-3$, and it is easy to find a free abelian subgroup of~$\operatorname{Out}(F_n)$ of rank~$2n-3$. Hence, the virtual cohomological dimension of~$\operatorname{Out}(F_n)$ is precisely the dimension of the spine (this was observed in \cite{CullerVogtmannOuterSpace86}). We refer to Vogtmann's surveys \cite{VogtmannOuterSpaceSurvey02}, \cite{VogtmannOuterSpaceSurvey18}, and \cite{CroweHelmsVogtmannOuterSpaceSurvey25} for detailed descriptions of Culler--Vogtmann Outer space and its spine, as well as their applications. 

\begin{remark}\label{rem:dichotomy about failure of bounds to match}
    Building on the untwisted Outer space constructed in \cite{CharneyStambaughVogtmannUntwistedOuterSpace17}, Bregman--Charney--Vogtmann \cite{BregmanCharneyVogtmannOuterSpaceRAAGs23} constructed a full \emph{Outer space}~$\os$ -- a contractible space with a proper~$\oag$-action.
\end{remark}

For those~$\Gamma$ for which the principal rank is strictly less than~$\dim(\spine)$, we must have at least one of the following departures from the free group case:
\begin{itemize}
    \item~$\vcduag$ is not realised by the rank of any free abelian subgroup (or there is such a free abelian subgroup, but it is in some sense unexpected or pathological);
    \item~$\vcduag$ is not realised geometrically as~$\dim(\spine)$.
\end{itemize}

Regarding the first of these possibilities, it is proved in \cite{CharneyVogtmannSubgpsQuotientsRAAGs11} that for many graphs,~$\uag$ cannot contain a torsion-free non-abelian solvable subgroup. However, there still exist graphs for which we cannot disregard such a possibility (compare with~$\textsc{vcd}(\operatorname{GL}_n(\mathbb{Z}))$, which is realised as the Hirsch rank of a (non-abelian) polycyclic subgroup (see \cite[Theorem 7.10 \& \S7.4]{BieriHomologicalDimensionDiscreteGroupsBook81})). As yet, we know of no example where~$\vcduag$ is not realised as the rank of a (natural, easily found) free abelian subgroup (Millard--Vogtmann's proof \cite{MillardVogtmannCubeComplexes21} that~$M(L) \leq \vcduag$ proceeds by finding a free abelian subgroup of rank~$M(L)$, generated in a reasonably anticapable way.) 

Millard--Vogtmann also consider the second problem. They proved that if there is a gap between the principal rank and~$\dim(\spine)$, and if~$\Gamma$ is barbed, then one can equivariantly retract all the top-dimensional cubes of~$\spine$, tightening the upper bound on~$\vcduag$ by one. Our work is inspired by this strategy and both replaces and extends it.

\subsection{Further directions}\label{subsec:further work}

There are certain obvious questions regarding the approach taken here. It is not known to what extent the conditions spiky and barbed are simply artefacts of the proof. In \S\ref{subsec:key lemmas} and \S\ref{sec:sufficient condition for M(L) = M(V)}, we discuss examples of graphs that violate spikiness in different ways. For one of these examples (\textit{Figure \ref{fig:delta}}), our results can nevertheless be applied with no modification. However, for another (\textit{Eg.~\ref{eg:graph that breaks Condition 1 and the ensuing lemma}}), the retraction process cannot be executed in its current form. This requires further investigation -- for example, the notion of redundancy used in this paper can be tweaked to be applied to the graph in \emph{Eg. \ref{eg:graph that breaks Condition 1 and the ensuing lemma}}, but this comes at the cost of a less streamlined definition and a less intuitive (if slightly weaker) version of our `spikiness' condition.

There are also natural questions to be asked about our resultant complex~$\res$. For example, is~$\dim(\res)$ always equal to~$\vcduag$? At present, we know of no example to the contrary. In another direction, Bridson--Vogtmann \cite{BridsonVogtmannSymsOuterSpace01} proved that~$\operatorname{Out}(F_n)$ is the group of simplicial automorphisms of the corresponding spine (for~$n \geq 3$). We may ask the same question here: when is the group of simplicial automorphisms of~$\spine$ isomorphic to~$\uag$? What about the group of automorphisms of~$\res$? One could phrase this as: when, if ever, is~$\spine$ or~$\res$ an accurate model for~$\uag$?

Finally, we remark that there already exists an algorithm, due to Day--Sale--Wade \cite{DaySaleWadeVCDAlgorithm21}, which computes the virtual cohomological dimension of~$\oag$ and~$\uag$ for any graph~$\Gamma$. This algorithm is difficult to apply in practice and does not provide a complex whose dimension realises the virtual cohomological dimension. One interesting further research direction would be to compare and contrast these two approaches for computing~$\vcduag$. For example, is there any transparent way of translating the value of~$\vcduag$ calculated in \cite{DaySaleWadeVCDAlgorithm21} into the size of a compatible set of~$\Gamma$-Whitehead partitions? If so, does this set resemble the top-dimensional cubes in~$\res$ in any way?

More concretely, we propose the following question as a good departure point. 
\begin{question*}
    Is~$\vcduag$ always equal to the principal rank of~$\Gamma$?
\end{question*}
The results in this paper give a positive answer to this question for spiky, barbed~$\Gamma$. Strengthening the approach taken here would therefore be one way of looking for a positive answer in general. On the other hand, one may search for counterexamples by applying Day--Sale--Wade's algorithm to graphs which fail to be simultaneously spiky and barbed.

\subsection*{Organisation}

In \S\ref{sec:auts of raags and CV Outer space}, we establish notation and gather the necessary preliminaries. \S\ref{sec:untwisted spine} introduces the spine of untwisted Outer space and discusses the work of \cite{MillardVogtmannCubeComplexes21} regarding the principal rank. In \S\ref{sec:retraction process}, we define `spikiness' and detail the retraction process which obtains~$\res$ from~$\spine$, proving \textit{Theorem \ref{customthm:main work} (i)}. \S\ref{sec:realising vcd} is dedicated to proving \textit{Theorem \ref{customthm:main work} (ii)} and deducing \textit{Theorem \ref{customthm:main thm}}. Finally, in \S\ref{sec:arbitrarily large gaps} and \S\ref{sec:sufficient condition for M(L) = M(V)}, we present detailed examples and computations, deducing \textit{Corollaries \ref{customcor:arbitrary gaps}} \& \textit{\ref{customcor:counterexample}}.

\subsection*{Acknowledgements}

This work grew out of my Master's thesis, completed under the kind and patient supervision of Karen Vogtmann. I would especially like to thank Prof. Vogtmann for the time she has dedicated to me and this project, as well as for introducing me to this special topic. In particular, I am grateful for her clarification regarding \emph{Definition \ref{def:spiky}}.

I would also like to thank my doctoral supervisors, Rachael Boyd and Tara Brendle, for their guidance, patience, and trust, as well as helpful discussions. Thanks are also due to Corey Bregman for taking the time to read through an earlier draft of this paper and for his helpful comments. Finally, I would like to express my gratitude to the referee for their careful reading and many helpful suggestions for improvement.

\section{Automorphisms of RAAGs and Culler--Vogtmann Outer space}\label{sec:auts of raags and CV Outer space}

In this section we provide the necessary preliminaries on right-angled Artin groups and their automorphisms, and on virtual cohomological dimension. There is also a brief discussion of Culler--Vogtmann Outer space.

\subsection{Right-angled Artin groups}\label{subsec:raags}

\begin{definition}
    Let~$\Gamma$ be a finite simplicial graph (that is, it has no single-edge loops and no multi-edges) whose vertices are labelled. Write~$V(\Gamma)$ and~$E(\Gamma)$ for its vertex and edge sets, respectively. The associated \emph{right-angled Artin group} (RAAG), denoted~$\raag$, is defined by the presentation \[\raag = \langle V(\Gamma) \;\;\vert\;\; [a, b] = 1 \;\text{if}\; \{a, b\} \in E(\Gamma)\rangle.\]

    Say that~$\Gamma$ is the \emph{defining graph} for~$\raag$. We will blur the distinction between a vertex of~$\Gamma$ and the corresponding generator of~$\raag$, and will communicate the data of a RAAG simply by its defining graph.
\end{definition}

Every Artin group has an associated cell complex called the \emph{Salvetti complex} \cite{SalvettiTopologyHyperplaneComplements87}. For a RAAG~$\raag$, the Salvetti complex~$\sal$ is a cube complex with a particularly simple description, as follows. There is one 0-cell, to which is attached a directed 1-cell for each generator. Then for each edge~$\{a, b\}$ in~$\Gamma$, we glue in a 2-torus along the corresponding commutator~$aba^{-1}b^{-1}$. Continuing inductively, for each~$k$-clique in~$\Gamma$ (which corresponds to a set of~$k$ mutually commuting generators), we glue in a~$k$-torus whose constituent~$(k-1)$-tori correspond to the subcliques of our~$k$-clique which have~$k-1$ vertices. 

For example, the Salvetti complex of the free group of rank~$n$ is simply the~$n$-petalled rose graph, while the Salvetti complex of~$\mathbb{Z}^n$ is an~$n$-torus. The Salvetti complex of a RAAG associated to a triangle-free graph is the presentation complex of the RAAG. Note that~$\pi_1(\sal) \cong \raag$. Charney and Davis proved \cite{CharneyDavisKpi195} that~$\sal$ is a~$K(\raag, 1)$ space, resolving the~$K(\pi, 1)$ conjecture for RAAGs; it follows that RAAGs are torsion-free and biautomatic (as shown by Niblo and Reeves \cite{NibloReevesCCsComplexity98}), for instance. 

We refer the reader to Charney's introduction to RAAGs for further details \cite{CharneyRAAGsSurvey07}. 

\subsection{Graph-theoretic preliminaries}\label{subsec:graph-theoretic preliminaries}

In this section, we collect some notation and terminology regarding defining graphs which we shall use throughout. For more details and proofs, we refer to \cite{CharneyVogtmannFinitenessProperties09}.

\begin{definition}\label{def:link and star}
    For~$\Gamma$ a simplicial graph, the \emph{link}~$\link{v}$ of a vertex~$v \in V(\Gamma)$ is the full subgraph spanned by the vertices adjacent to~$v$. The \emph{star} of~$v$ is the full subgraph spanned by~$v$ and all vertices adjacent to it; denote it~$\str{v}$.
\end{definition}


Define a relation~$\leq$ on the vertices of~$\Gamma$ by saying that~$v \leq w$ if~$\link{v}$ is contained in~$\str{w}$. We define an equivalence relation~$\sim$ on~$V(\Gamma)$ by saying that~$v \sim w$ if both~$v \leq w$ and~$w \leq v$. Then~$\leq$ is a partial order on the set of equivalence classes. We write~$[v]$ for the equivalence class of~$v$. 

A vertex~$v$ is said to be \emph{maximal} if its equivalence class is maximal with respect to this partial order, i.e., for all~$v' \in [v]$, there is no vertex~$w \notin [v]$ with~$v' \leq w$.

One way to have~$v \leq w$ is if~$\link{v} \subseteq \link{w}$; in this case we write~$v \leq_\circ w$. We write~$v <_\circ w$ if there is a strict containment~$\link{v} \subsetneq \link{w}$; in this case we say that~$w$ \emph{dominates}~$v$. 

\begin{definition}\label{def:principal vertex}
    We say that a vertex~$v \in V(\Gamma)$ is \emph{principal} if there is no~$w \in V(\Gamma)$ with~$v <_\circ w$. Otherwise, we say that~$v$ is \emph{non-principal}.
\end{definition}

All maximal vertices are principal. However, as observed in \cite{MillardVogtmannCubeComplexes21}, not all principal vertices are maximal. For example, in the triangle with leaves at two of its vertices, the third vertex of the triangle is principal but not maximal (see \emph{Figure \ref{fig:example of principal but not maximal}}).

\begin{figure}[H]
    \begin{center}
        \begin{tikzpicture}
            \fill (-1.5, 0) circle (2pt);
            \fill (-0.5, 0) circle (2pt);
            \fill (0.5, 0) circle (2pt);
            \fill (1.5, 0) circle (2pt);
            \fill (0, 0.8) circle (2pt);
            \draw[black, thick] (-1.5, 0) -- (-0.5, 0);
            \draw[black, thick] (-0.5, 0) -- (0.5, 0);
            \draw[black, thick] (0.5, 0) -- (1.5, 0);
            \draw[black, thick] (-0.5, 0) -- (0, 0.8);
            \draw[black, thick] (0.5, 0) -- (0, 0.8);
            \node at (0, 0.8) [above = 2pt, black, thick]{$v$};
        \end{tikzpicture}
        \caption{$v$ is a principal vertex but is not maximal.}
        \label{fig:example of principal but not maximal}
    \end{center}
\end{figure}
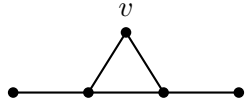

\subsection{Automorphisms of RAAGs}\label{subsec:auts of raags}

Confirming a conjecture of H.~Servatius \cite{ServatiusAutRAAGs89}, Laurence proved that~$\operatorname{Aut}(\raag)$ has the following set of generators (and hence~$\oag$ is generated by their images).

\begin{theorem-definition}[\cite{LaurenceAutRAAGs95}, \emph{Theorem 6.9}]\label{thmdef:generators for Aut(RAAG)}
    For any~$\Gamma$,~$\operatorname{Aut}(\raag)$ is generated by the following set of \emph{elementary transformations}.
    \begin{enumerate}
        \item \emph{Graph automorphisms}: an automorphism of~$\Gamma$ permutes the generators of~$\raag$ and induces an element of~$\operatorname{Aut}(\raag)$.
        \item \emph{Inversions}: a map which sends a generator~$a \mapsto a^{-1}$ and fixes all other generators.
        \item \emph{Transvections}: for two vertices~$v, w$ of~$\Gamma$ with~$v \leq w$, the map which sends~$v \mapsto vw$ or~$wv$ and fixes all other generators is an automorphism. These are split into two types:
        \begin{itemize}
            \item \emph{twists}, if~$v$ and~$w$ are adjacent (in which case~$vw = wv$);
            \item \emph{folds}, if~$v, w$ are not adjacent. In particular,~$v \mapsto vw$ is a \emph{right fold}, while~$v \mapsto wv$ is a \emph{left fold}.
        \end{itemize}
        \item \emph{Partial conjugations}: for a vertex~$v$, let~$C$ be a connected component of~$\Gamma \setminus \str{v}$. The map which sends every generator~$c \in C$ to the conjugate~$vcv^{-1}$ and fixes all other generators is called a \emph{partial conjugation} (Laurence and Servatius refer to these as \emph{locally inner} automorphisms).  
    \end{enumerate}
\end{theorem-definition}

When talking about generators of~$\oag$ and its subgroups, we will use the same terminology: e.g., by a `partial conjugation' in~$\oag$ we mean the image of a partial conjugation under the quotient~$\operatorname{Aut}(\raag) \twoheadrightarrow \oag$.

In this paper we restrict attention to the so-called \emph{untwisted subgroup}~$\uag \leq \oag$, which is generated by all graph automorphisms, inversions, partial conjugations, and folds, i.e.~everything in the above list other than twists. 

\subsection{Culler--Vogtmann Outer space~$CV_n$ and its spine~$K_n$}\label{subsec:Culler-Vogtmann Outer space}

Let~$F_n$ denote the free group of rank~$n$. In \cite{CullerVogtmannOuterSpace86}, Culler and Vogtmann introduced what is now called \emph{Culler--Vogtmann Outer space}, a contractible complex~$CV_n$ with a proper action of~$\operatorname{Out}(F_n)$. Relevant for us,~$CV_n$ has a natural deformation retract~$K_n$ called its \emph{spine}, which is a contractible cube complex with a proper and cocompact action of~$\operatorname{Out}(F_n)$. Bridson and Vogtmann \cite{BridsonVogtmannSymsOuterSpace01} prove that the group of simplicial automorphisms of~$K_n$ is precisely~$\operatorname{Out}(F_n)$.

For surveys of some of the applications of~$CV_n$, we refer the reader to \cite{VogtmannOuterSpaceSurvey02}, \cite{VogtmannOuterSpaceSurvey18}, \cite{CroweHelmsVogtmannOuterSpaceSurvey25}.

Generalising~$CV_n$, Charney and Vogtmann (initially with Crisp \cite{CharneyCrispVogtmann2dimRAAGs07}, then with Stambaugh \cite{CharneyStambaughVogtmannUntwistedOuterSpace17}, and finally the full construction with Bregman \cite{BregmanCharneyVogtmannOuterSpaceRAAGs23}) have produced an `Outer space'~$\mathcal{O}_\Gamma$ for any RAAG~$\raag$. In particular, there is an \emph{untwisted Outer space}~$\mathcal{O}_\Gamma^U$, first constructed in \cite{CharneyStambaughVogtmannUntwistedOuterSpace17} (where it is denoted~$\Sigma_\Gamma$), which is a contractible complex with a proper action of~$\uag$. In analogy with~$CV_n$ retracting onto its spine~$K_n$, untwisted Outer space also has a \emph{spine}, denoted~$\spine$. We will describe this in more detail in \S\ref{subsec:untwisted spine}.

\subsection{Virtual cohomological dimension}\label{subsec:vcd}

We briefly recall the notion of \emph{virtual cohomological dimension}; the reader is referred to \cite{BrownCohomologyofGroups1stEd82} for further details.

Cohomological dimension may be defined for any unital nonzero commutative ring~$R$, but we shall only be interested in the case~$R = \mathbb{Z}$.

\begin{definition}[cf. ~\cite{BrownCohomologyofGroups1stEd82}, \emph{Lemma VIII.2.1}]
    Let~$G$ be a discrete group, and consider~$\mathbb{Z}$ as the trivial~$\mathbb{Z}G$-module, where~$\mathbb{Z}G$ is the group ring. Define~$\textsc{cd}(G) = \textsc{cd}_{\mathbb{Z}}(G)$, the \emph{cohomological dimension} of~$G$, to be the least integer~$n$ such that~$\mathbb{Z}$ admits a projective resolution of~$\mathbb{Z}G$-modules of length~$n$, if such~$n$ exists, or to be infinite otherwise.
\end{definition}

Equivalently (and hence the nomenclature),~$\textsc{cd}(G)$ is the greatest integer such that the cohomology group~$H^n(G; M)$ does not vanish, for any~$\mathbb{Z}G$-module~$M$.

Observe that the cohomological dimension of a group with torsion is always infinite, as all the even cohomology groups with coefficients in~$\mathbb{Z}G$ are nonzero. However, we have the following result of Serre \cite{SerreCohomologieGroupesDiscrets71}, which can be found in \textit{Ch.~VIII.3} of \cite{BrownCohomologyofGroups1stEd82}.

\begin{theorem}[Serre]
    Let~$G$ be virtually torsion-free. Then all finite-index subgroups of~$G$ have the same cohomological dimension.
\end{theorem}

Hence for a group~$G$ which contains a finite-index torsion-free subgroup~$H$, the \emph{virtual cohomological dimension} of~$G$, defined to be~$\textsc{vcd}(G) \coloneqq \textsc{cd}(H)$, is well-defined. 

Charney and Vogtmann have shown \cite{CharneyVogtmannFinitenessProperties09} that for any finite simplicial graph~$\Gamma$,~$\oag$ is virtually torsion-free and has finite virtual cohomological dimension. Hence~$\vcduag$ is also finite.

The following geometric control over the virtual cohomological dimension will be crucial.

\begin{theorem}[\cite{BrownCohomologyofGroups1stEd82}, \emph{Theorem VIII.11.1}]\label{thm:vcd upper bound realisation theorem}
    Let~$G$ act properly and cocompactly on a proper contractible CW-complex~$X$. Then~$\textsc{vcd}(G) \leq \dim(X)$.
\end{theorem}

Since the spine~$K_n$ of~$CV_n$ is proper (since it is a locally finite complex) and contractible and has a proper and cocompact action of~$\operatorname{Out}(F_n)$, this theorem implies that
\[\textsc{vcd}(\operatorname{Out}(F_n)) \leq \dim(K_n) = 2n-3.\]
On the other hand, as noted in §0 of \cite{CullerVogtmannOuterSpace86}, one can find a free abelian subgroup of~$\operatorname{Out}(F_n)$ of rank~$2n-3$; one such example is the image in~$\operatorname{Out}(F_n)$ of \[\langle \alpha_i: x_i \mapsto x_1x_i, \; \beta_i: x_i \mapsto x_ix_1 \;:\; 2 \leq i \leq n\rangle \leq \operatorname{Aut}(F_n).\] Thus~$\textsc{vcd}(\operatorname{Out}(F_n)) = 2n-3$.

There are no twists in~$\operatorname{Out}(F_n)$, so in this case the untwisted subgroup~$U(F_n)$ is the entire outer automorphism group. It is therefore natural to ask to what extent the above strategy generalises to determining~$\vcduag$, for any RAAG~$\raag$.

\section{The untwisted spine and known results on \texorpdfstring{$\vcduag$}{vcd of untwisted subgroup}}\label{sec:untwisted spine}

Culler--Vogtmann Outer space~$CV_n$ is constructed using `marked metric graphs': graphs with a homotopy equivalence to the~$n$-petalled rose, and where each of the edges has a length. Roughly, one can obtain the spine~$K_n$ by ignoring the edge-lengths and considering only the combinatorial type of each marked graph. 

For any right-angled Artin group~$\raag$, untwisted Outer space~$\mathcal{O}_\Gamma^U$ is built analogously. The Salvetti complex~$\mathbb{S}_\Gamma$ generalises the~$n$-petalled rose, and then the basic objects of~$\mathcal{O}_\Gamma^U$ are \emph{marked metric~$\Gamma$-complexes}: cube complexes which are `marked' with a certain homotopy equivalence to~$\mathbb{S}_\Gamma$ and which are endowed with a metric that turns the constituent cubes into rectilinear parallelepipeds. Once again, the untwisted spine~$\spine$ is roughly obtained by ignoring this metric. We will not need the construction of the untwisted Outer space~$\mathcal{O}_\Gamma^U$, but we now outline in some detail the structure of~$\spine$.

\subsection{$\Gamma$-Whitehead partitions}\label{subsec:Gamma-Whitehead partitions}

Fix a finite simplicial graph~$\Gamma$, and write~$V = V(\Gamma)$. Write~$V^\pm$ for~$V \cup V^{-1}$, i.e.~the collection of the generators of~$\raag$ and their inverses. Write~$\Gamma^\pm$ for the simplicial graph on~$V^\pm$ where vertices corresponding to commuting elements of~$\raag$ are joined by an edge, unless they are inverse to each other. We will write~$\link{v}^\pm$ (resp.~$\str{v}^\pm$) to mean the link (resp. star) of~$v$ in~$\Gamma^\pm$.

We now define~$\Gamma$-Whitehead partitions. Roughly, given a vertex~$m \in V$, we distribute the connected components of~$\Gamma^\pm \setminus \link{m}^\pm$ between two `sides', with~$m$ in one of the sides and~$m^{-1}$ in the other. 

\begin{definition}\label{def:Gamma-Whitehead partition}
    Let~$m \in V$. A \emph{$\Gamma$-Whitehead partition~$(\mathcal{P}, m)$ based at m} is a partition of the set~$V^\pm$ satisfying the following conditions:
    \begin{itemize}
        \item the partition consists of three subsets: two \emph{sides}~$P$ and~$\overline{P}$, and the \emph{link}~$\link{\calP}$;
        \item \noindent$\link{\calP} = V\left(\link{v}^\pm\right)$ -- that is,~$\link{\calP}$ consists of all generators that commute with~$m$, and their inverses;
        \item \noindent$P \sqcup \overline{P}$ is a \emph{thick} partition of~$V^\pm \setminus \link{\calP}$ (i.e., there are at least two elements in each side);
        \item \noindent$m$ and~$m^{-1}$ are in different sides of~$\calP$;
        \item if two distinct vertices~$v, w$ of~$\Gamma$ are in the same connected component of~$\Gamma \setminus \str{m}$, then~$v, w$ and their inverses are all on the same side of~$\calP$.
    \end{itemize}

    We call the side of~$\calP$ which contains~$m$ the \emph{$m$-side} (and the other side the \emph{$m^{-1}$-side}). 

    \noindent We will abbreviate `$\Gamma$-Whitehead partition' to `$\Gamma$-partition'; when~$\Gamma$ is clear, we will often simply say `partition' instead.
\end{definition}

When considering a~$\Gamma$-partition, we may sometimes omit reference to the base (this is discussed along with~$\Gamma$-Whitehead automorphisms below). In this case, we may give the~$\Gamma$-partition as~$\calP = \left(P\vert \overline{P} \vert \link{\calP}\right)$.

\emph{Figure \ref{fig:example of Gamma-partition}} presents an example illustrating this definition.

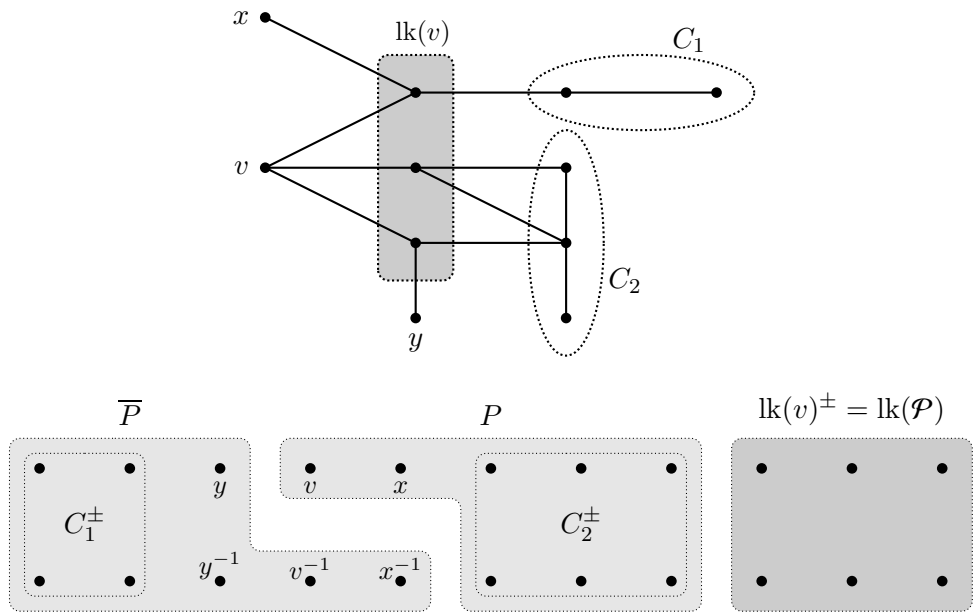
\begin{figure}[H]
    \begin{center}
        \begin{tikzpicture}
            %
            \draw[black, densely dotted, thick] (5, -1) ellipse (1.5cm and 0.5cm);
            \draw[black, densely dotted, thick] (4, -3) ellipse (0.5cm and 1.5cm);
            \draw[black, rounded corners, thick, densely dotted, fill=black!20!white] (1.5, -3.5) rectangle ++(1, 3);
            \fill (0, 0) circle (2pt);
            \fill (2, -1) circle (2pt);
            \fill (4, -1) circle (2pt);
            \fill (6, -1) circle (2pt);
            \fill (0, -2) circle (2pt);
            \fill (2, -2) circle (2pt);
            \fill (4, -2) circle (2pt);
            \fill (2, -3) circle (2pt);
            \fill (4, -3) circle (2pt);
            \fill (2, -4) circle (2pt);
            \fill (4, -4) circle (2pt);
            \draw[black, thick] (0, 0) -- (2, -1);
            \draw[black, thick] (2, -1) -- (4, -1);
            \draw[black, thick] (4, -1) -- (6, -1);
            \draw[black, thick] (0, -2) -- (2, -1);
            \draw[black, thick] (0, -2) -- (2, -2);
            \draw[black, thick] (0, -2) -- (2, -3);
            \draw[black, thick] (2, -2) -- (4, -2);
            \draw[black, thick] (2, -2) -- (4, -3);
            \draw[black, thick] (4, -2) -- (4, -3);
            \draw[black, thick] (2, -3) -- (4, -3);
            \draw[black, thick] (2, -3) -- (2, -4);
            \draw[black, thick] (4, -3) -- (4, -4);
            \node at (0, -2) [left = 2pt, black, thick]{$v$};
            \node at (0, 0) [left = 2pt, black, thick]{$x$};
            \node at (2, -4) [below = 2pt, black, thick]{$y$};
            \node at (5.2, -0.3) [right = 2pt, black, thick]{$C_1$};
            \node at (4.35, -3.5) [right = 2pt, black, thick]{$C_2$};
            \node at (2.1, -0.6) [above = 2pt, black, thick]{\small$\operatorname{lk}(v)$};
            %
            %
            \draw[black, thin, densely dotted, rounded corners, fill = black!10!white] (-2, -5.6) to (-0.2, -5.6) to (-0.2, -7.1) to (2.2, -7.1) to (2.2, -7.9) to (-3.4, -7.9) to (-3.4, -5.6) to (-2, -5.6);
            \draw[black, thin, densely dotted, rounded corners, fill = black!10!white] (1, -5.6) to (5.8, -5.6) to (5.8, -7.9) to (2.6, -7.9) to (2.6, -6.4) to (0.2, -6.4) to (0.2, -5.6) to (1, -5.6);
            \draw[black, thin, densely dotted, rounded corners] (-2.4, -5.8) to (-1.6, -5.8) to (-1.6, -7.7) to (-3.2, -7.7) to (-3.2, -5.8) to (-2.4, -5.8);
            \draw[black, thin, densely dotted, rounded corners] (4, -5.8) to (5.6, -5.8) to (5.6, -7.7) to (2.8, -7.7) to (2.8, -5.8) to (4, -5.8);
            \draw[black, thin, densely dotted, rounded corners, fill=black!20!white] (7, -5.6) to (9.4, -5.6) to (9.4, -7.9) to (6.2, -7.9) to (6.2, -5.6) to (7, -5.6);
            \fill (-3, -6) circle (2pt);
            \fill (-3, -7.5) circle (2pt);
            \fill (-1.8, -6) circle (2pt);
            \fill (-1.8, -7.5) circle (2pt);
            \fill (-0.6, -6) circle (2pt);
            \fill (-0.6, -7.5) circle (2pt);
            \fill (0.6, -6) circle (2pt);
            \fill (0.6, -7.5) circle (2pt);
            \fill (1.8, -6) circle (2pt);
            \fill (1.8, -7.5) circle (2pt);
            \fill (3, -6) circle (2pt);
            \fill (3, -7.5) circle (2pt);
            \fill (4.2, -6) circle (2pt);
            \fill (4.2, -7.5) circle (2pt);
            \fill (5.4, -6) circle (2pt);
            \fill (5.4, -7.5) circle (2pt);
            \fill (6.6, -6) circle (2pt);
            \fill (6.6, -7.5) circle (2pt);
            \fill (7.8, -6) circle (2pt);
            \fill (7.8, -7.5) circle (2pt);
            \fill (9, -6) circle (2pt);
            \fill (9, -7.5) circle (2pt);
            \node at (-0.6, -6) [below = 1pt, thick, black]{\small$y$};
            \node at (-0.6, -7.65) [above = 1pt, black, thick]{\small$y^{-1}$};
            \node at (0.6, -6) [below = 1pt, thick, black]{\small$v$};
            \node at (0.6, -7.65) [above = 1pt, black, thick]{\small$v^{-1}$};
            \node at (1.8, -6) [below = 1pt, thick, black]{\small$x$};
            \node at (1.8, -7.65) [above = 1pt, black, thick]{\small$x^{-1}$};
            \node at (-2.4, -6.75) [black, thick]{$C_1^\pm$};
            \node at (4.2, -6.75) [black, thick]{$C_2^\pm$};
            \node at (-1.8, -5.6) [above = 1pt, black, thick]{$\overline{P}$};
            \node at (3, -5.6) [above = 1pt, black, thick]{$P$};
            \node at (7.8, -5.6) [above = 1pt, black, thick]{$\link{v}^\pm = \link{\calP}$};
        \end{tikzpicture}
        \caption{This is a schematic showing an example of a graph~$\Gamma$ and a~$\Gamma$-partition $\calP = \left(\{v, x, C_2^\pm\} \; | \; \{v^{-1}, x^{-1}, y, y^{-1}, C_1^\pm\} \; | \; \link{\calP} \right)$, which is based at~$v$. In the illustration of~$\calP$, the nodes are arranged in pairs, with a vertex of~$\Gamma$ on the top row and its inverse below it on the bottom row. Here,~$v$ is the only possible base for~$\calP$. The connected components of~$\Gamma^\pm \setminus \str{v}^\pm$ are the singletons~$\{x\}$,~$\{x^{-1}\}$,~$\{y\}$,~$\{y^{-1}\}$, along with the components~$C_1^\pm$,~$C_2^\pm$ as indicated in the diagram.}
        \label{fig:example of Gamma-partition}
    \end{center}
\end{figure}

\begin{definition}\label{def:split and max of a partition}
    Given a~$\Gamma$-partition~$\calP$, we say that~$\calP$ \emph{splits} a vertex~$v$ if~$v$ and~$v^{-1}$ are on different sides of~$\calP$. Recall that we have a partial order~$\leq$ on the (equivalence classes of) vertices of~$\Gamma$. Define
    \begin{center}
       ~$\splt{\calP} \coloneqq \{v \in V \colon \calP \; \text{splits} \; v\} \; ;$
        \[\mx{\calP} \coloneqq \{v \in V \colon v \; \text{maximal w.r.t.~$\leq$ in} \; \splt{\calP}\}.\]
    \end{center}

\end{definition}

The definition of~$\Gamma$-partition is motivated by the following: each~$\Gamma$-partition~$(\calP, m)$ corresponds to a \emph{$\Gamma$-Whitehead automorphism}~$\varphi(\calP, m)$ of~$\raag$, as follows. If~$y \in \splt{\calP}$ then~$\varphi(\calP, m)$ sends~$y \mapsto ym^{-1}$ for~$y \in P$, and~$y \mapsto my$ if~$y \in \overline{P}$. If~$y, y^{-1} \in P$ then~$\varphi(\calP, m)$ sends~$y \mapsto mym^{-1}$, and for all other~$y$, we have~$\varphi(\calP, m)(y) = y$. Observe that if we did not require \emph{thickness} (that is, that both sides contain at least two elements), then the corresponding~$\Gamma$-Whitehead automorphism might be inner. 

Note that if~$m, m' \in \mx{\calP}$ then while~$(\calP, m)$ and~$(\calP, m')$ (and~$(\calP, m^{-1})$ and~$(\calP, m'^{-1})$) are the same~$\Gamma$-partition, their corresponding~$\Gamma$-Whitehead automorphisms are not the same. All elements of~$\mx{\calP}$ are equivalent to each other with respect to~$\leq$, and none commutes with any other element of~$\mx{\calP}$, so they each have the same link in~$\Gamma^\pm$. Hence, any element of~$\mx{\calP}^\pm$ could serve as a base for the partition~$\calP$. A~$\Gamma$-partition is therefore entirely determined by one of its sides, and we may omit reference to its base if we are simply considering the partition and not its associated~$\Gamma$-Whitehead automorphism. 

Every fold or partial conjugation is an instance of a~$\Gamma$-Whitehead automorphism, and each~$\Gamma$-Whitehead automorphism is a product of folds and partial conjugations. The untwisted subgroup~$\uag$ is generated by inversions, graph automorphisms, and~$\Gamma$-Whitehead automorphisms, while the work of Laurence and Servatius (\cite{ServatiusAutRAAGs89}, \cite{LaurenceAutRAAGs95}; see \textit{Theorem-Definition \ref{thmdef:generators for Aut(RAAG)}}) proves that~$\oag$ is generated by these along with twists.

\begin{definition}\label{def:adjacency}
    Say that~$\Gamma$-partitions~$\calP$ and~$\calQ$ are \emph{adjacent} if~$\mx{\calP} \subseteq \link{\calQ}$, or equivalently, if~$\mx{\calQ} \subseteq \link{\calP}$ (equivalence follows from the pairwise equivalence of elements of~$\mx{\calP}$.) If~$v$ is a vertex of~$\Gamma$, say that \emph{$v$ is adjacent to~$\calP$} if~$v \in \link{\calP}$. Finally, say that two vertices of~$\Gamma$ are \emph{adjacent} if they are adjacent in~$\Gamma$ in the usual graph-theoretic sense.
\end{definition}

If two~$\Gamma$-partitions~$\calP, \calQ$ are adjacent, then this means that some (hence any) base of~$\calP$ commutes with some (hence any) base of~$\calQ$, but that these two bases are not equal. The converse also holds. We refer the reader to the ``Warning'' recorded after \emph{Definition 2.7} in \cite{BregmanCharneyVogtmannOuterSpaceRAAGs23}.

\begin{definition} \label{def:compatibility}
    We will say that distinct~$\Gamma$-partitions~$\calP$ and~$\calQ$ are \emph{compatible} if either
    \begin{itemize}
        \item~$\calP$ and~$\calQ$ are adjacent; or
        \item there exists a side of~$\calP$ which has empty intersection with at least one of the sides of~$\calQ$.
    \end{itemize}

    If~$\Pi$ is a set of~$\Gamma$-partitions which are pairwise compatible with each other, then call~$\Pi$ a \emph{compatible set}.
\end{definition}

Note that compatibility is not transitive, as demonstrated in \emph{Figure \ref{fig:compatibility is not transitive}}.

\begin{figure}
    \begin{center}
        \begin{tikzpicture}
            %
            %
            \fill (-2, 0.5) circle (2pt);
            \fill (-1, 0.5) circle (2pt);
            \fill (0, 0.5) circle (2pt);
            \fill (1, 0.5) circle (2pt);
            \fill (2, 0.5) circle (2pt);
            \draw[black, thick] (-2, 0.5) -- (-1, 0.5);
            \draw[black, thick] (-1, 0.5) -- (0, 0.5);
            \draw[black, thick] (0, 0.5) -- (1, 0.5);
            \node at (-2, 0.5) [above = 1pt, black, thick]{$a$};
            \node at (-1, 0.5) [above = 1pt, black, thick]{$b$};
            \node at (0, 0.5) [above = 1pt, black, thick]{$c$};
            \node at (1, 0.5) [above = 1pt, black, thick]{$d$};
            \node at (2, 0.5) [above = 1pt, black, thick]{$e$};
            \node at (-3, 0.5) [black, thick]{$\Gamma \;\; =$};
            %
            %
            \node at (0, -0.5) [black, thick]{$\calP_1 = \left( \{a, c, c^{-1}, d, d^{-1}\} \;|\; \{a^{-1}, e, e^{-1}\} \;|\; \{b, b^{-1}\} \right)$};
            \node at (0, -1.2) [black, thick]{$\calP_2 = \left( \{b, e\} \;|\; \{b^{-1}, d, d^{-1}, e^{-1}\} \;|\; \{a, a^{-1}, c, c^{-1}\} \right)$};
            \node at (0, -1.9) [black, thick]{$\calP_3 = \left( \{d, a, a^{-1}, b, b^{-1}, e\} \;|\; \{d^{-1}, e^{-1}\} \;|\; \{c, c^{-1}\} \right)$};
            %
            %
            \draw [black, thick, rounded corners, densely dotted, fill = black!20!white] (-7, -3.1) to (-4.6, -3.1) to (-4.6, -5.15) to (-6.4, -5.15) to (-6.4, -3.2) to (-7.6, -3.2) to (-7.6, -4) to (-8.4, -4) to (-8.4, -3.1) to (-7, -3.1);
            \draw [black, thick, rounded corners, densely dotted] (-7, -5.4) to (-3.6, -5.4) to (-3.6, -3.1) to (-4.4, -3.1) to (-4.4, -5.3) to (-7.6, -5.3) to (-7.6, -4.5) to (-8.4, -4.5) to (-8.4, -5.4) to (-7, -5.4);
            \draw [black, thick, rounded corners, densely dotted] (-7.4, -4.25) to (-7.4, -5.15) to (-6.6, -5.15) to (-6.6, -3.3) to (-7.4, -3.3) to (-7.4, -4.25);
            \fill (-8, -3.5) circle (2pt);
            \fill (-7, -3.5) circle (2pt);
            \fill (-6, -3.5) circle (2pt);
            \fill (-5, -3.5) circle (2pt);
            \fill (-4, -3.5) circle (2pt);
            \fill (-8, -5) circle (2pt);
            \fill (-7, -5) circle (2pt);
            \fill (-6, -5) circle (2pt);
            \fill (-5, -5) circle (2pt);
            \fill (-4, -5) circle (2pt);
            \node at (-8, -4.05) [above = 1pt, black, thick]{$a$};
            \node at (-7, -4.05) [above = 1pt, black, thick]{$b$};
            \node at (-6, -4.05) [above = 1pt, black, thick]{$c$};
            \node at (-5, -4.05) [above = 1pt, black, thick]{$d$};
            \node at (-4, -4.05) [above = 1pt, black, thick]{$e$};
            \node at (-7, -4.25) [black, thick]{\tiny$\link{\calP_1}$};
            \node at (-5.5, -2.85) [black, thick]{$P_1$};
            \node at (-3.8, -5.7) [black, thick]{$\overline{P_1}$};
            %
            %
            \draw [black, thick, rounded corners, densely dotted, fill = black!20!white] (0, -3.1) to (2.4, -3.1) to (2.4, -4) to (1.6, -4) to (1.6, -3.2) to (-0.6, -3.2) to (-0.6, -4) to (-1.4, -4) to (-1.4, -3.1) to (0, -3.1);
            \draw [black, thick, rounded corners, densely dotted] (0, -5.4) to (2.4, -5.4) to (2.4, -4.5) to (1.4, -4.5) to (1.4, -3.3) to (0.6, -3.3) to (0.6, -5.3) to (-0.6, -5.3) to (-0.6, -4.6) to (-1.4, -4.6) to (-1.4, -5.4) to (0, -5.4);
            \draw [black, thick, rounded corners, densely dotted] (-2.4, -4.25) to (-2.4, -5.4) to (-1.6, -5.4) to (-1.6, -4.4) to (-0.4, -4.4) to (-0.4, -5.15) to (0.4, -5.15) to (0.4, -3.3) to (-0.4, -3.3) to (-0.4, -4.1) to (-1.6, -4.1) to (-1.6, -3.1) to (-2.4, -3.1) to (-2.4, -4.25);
            \fill (-2, -3.5) circle (2pt);
            \fill (-1, -3.5) circle (2pt);
            \fill (0, -3.5) circle (2pt);
            \fill (1, -3.5) circle (2pt);
            \fill (2, -3.5) circle (2pt);
            \fill (-2, -5) circle (2pt);
            \fill (-1, -5) circle (2pt);
            \fill (0, -5) circle (2pt);
            \fill (1, -5) circle (2pt);
            \fill (2, -5) circle (2pt);
            \node at (-2, -4.05) [above = 1pt, black, thick]{$a$};
            \node at (-1, -4.05) [above = 1pt, black, thick]{$b$};
            \node at (0, -4.05) [above = 1pt, black, thick]{$c$};
            \node at (1, -4.05) [above = 1pt, black, thick]{$d$};
            \node at (2, -4.05) [above = 1pt, black, thick]{$e$};
            \node at (0.5, -2.85) [black, thick]{$P_2$};
            \node at (1.5, -5.7) [black, thick]{$\overline{P_2}$};
            \node at (-1.9, -4.25) [black, thick]{\tiny$\link{\calP_2}$};
            %
            %
            \draw [black, thick, rounded corners, densely dotted] (6, -3.1) to (8.4, -3.1) to (8.4, -4) to (6.6, -4) to (6.6, -3.2) to (5.4, -3.2) to (5.4, -5.4) to (3.6, -5.4) to (3.6, -3.1) to (6, -3.1);
            \draw [black, thick, rounded corners, densely dotted, fill = black!20!white] (8, -5.4) to (8.4, -5.4) to (8.4, -4.5) to (6.6, -4.5) to (6.6, -5.4) to (8, -5.4);
            \draw [black, thick, rounded corners, densely dotted] (6.4, -4.25) to (6.4, -5.4) to (5.6, -5.4) to (5.6, -3.3) to (6.4, -3.3) to (6.4, -4.25);
            \fill (4, -3.5) circle (2pt);
            \fill (5, -3.5) circle (2pt);
            \fill (6, -3.5) circle (2pt);
            \fill (7, -3.5) circle (2pt);
            \fill (8, -3.5) circle (2pt);
            \fill (4, -5) circle (2pt);
            \fill (5, -5) circle (2pt);
            \fill (6, -5) circle (2pt);
            \fill (7, -5) circle (2pt);
            \fill (8, -5) circle (2pt);
            \node at (4, -4.05) [above = 1pt, black, thick]{$a$};
            \node at (5, -4.05) [above = 1pt, black, thick]{$b$};
            \node at (6, -4.05) [above = 1pt, black, thick]{$c$};
            \node at (7, -4.05) [above = 1pt, black, thick]{$d$};
            \node at (8, -4.05) [above = 1pt, black, thick]{$e$};
            \node at (4.5, -2.85) [black, thick]{$P_3$};
            \node at (7.5, -5.7) [black, thick]{$\overline{P_3}$};
            \node at (6, -4.25) [black, thick]{\tiny$\link{\calP_3}$};
        \end{tikzpicture}
        \caption{An example of a graph~$\Gamma$ and three~$\Gamma$-partitions,~$\calP_1$ (based at~$a$),~$\calP_2$ (based at~$b$), and~$\calP_3$ (based at~$d$). Observe that although~$e$ is split by~$\calP_2$ and~$\calP_3$, it cannot serve as a base for either as~$\link{e} = \emptyset$, so~$e <_\circ b$ and~$e <_\circ d$.~$\calP_1$ is compatible with~$\calP_2$ since there is a choice of side of each (shaded) which have empty intersection (in fact,~$\calP_1$ and~$\calP_2$ are also adjacent, since~$a$ and~$b$ commute). Since~$d$ does not commute with~$a$ or~$b$, we cannot have~$d$ in the link of either~$\calP_1$ or~$\calP_2$, so~$\calP_3$ is not adjacent to~$\calP_1$ or~$\calP_2$. However,~$P_2 \cap \overline{P_3} = \emptyset$, so~$\calP_2$ and~$\calP_3$ are compatible. Each side of~$\calP_3$ has non-empty intersection with each side of~$\calP_1$; hence~$\calP_1$ is not compatible with~$\calP_3$. This example shows that compatibility of partitions is not transitive.}
        \label{fig:compatibility is not transitive}
    \end{center}
\end{figure}
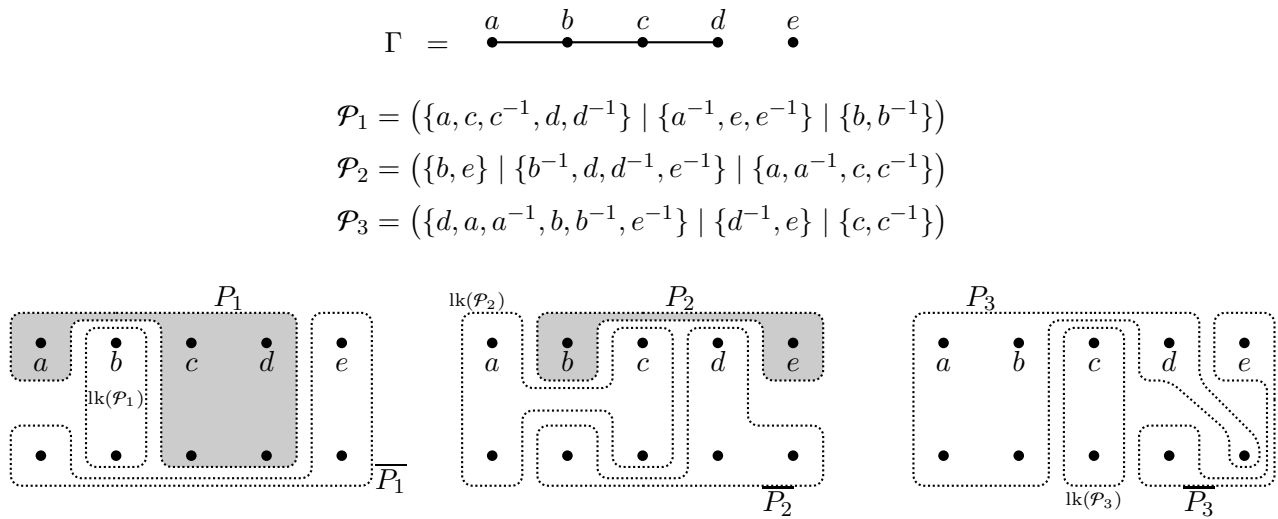

\subsection{The untwisted spine~$\spine$}\label{subsec:untwisted spine}

In this subsection, we recall from \cite{CharneyStambaughVogtmannUntwistedOuterSpace17} the construction of the untwisted spine~$\spine$, and collate some facts about it.

\begin{definition}\label{def:hyperplane collapse}
Let~$H$ be a hyperplane in a cube complex~$X$. Call the closure of the set of cubes that intersect~$H$ the \emph{hyperplane carrier}~$\kappa(H)$. We say that~$H$ is a \emph{carrier retract} if~$\kappa(H)$ is isomorphic to~$H \times [0, 1]$ as a cube complex. If~$H$ is a carrier retract, then the \emph{hyperplane collapse} associated with~$H$ is the deformation retract~$c_H\colon X \to X\ssslash H$ which collapses~$\kappa(H)$ orthogonally to~$H$. For a collection~$\mathcal{H}$ of carrier retracts, we write~$X\ssslash \mathcal{H}$ for the space obtained from~$X$ by collapsing all the hyperplanes in~$\mathcal{H}$ in any order (by \cite[Lemma 4.4]{CharneyStambaughVogtmannUntwistedOuterSpace17}, this is well-defined). We will say that such a set~$\mathcal{H}$ is \emph{acyclic} if~$X \to X\ssslash \mathcal{H}$ is a homotopy equivalence. [The double slash notation~$X\ssslash \mathcal{H}$ is not intended to evoke other notions using the double slash, such as homotopy quotients or GIT quotients; we take our lead from e.g.~\cite{BregmanCharneyVogtmannFiniteSubgroups25}.]

    \noindent If an edge and a hyperplane have non-empty intersection, then we say that they are \emph{dual} to each other. An \emph{orientation} of a hyperplane~$H$ is a consistent choice of orientation on each of the edges dual to~$H$.
\end{definition}

Let~$\Pi$ be a set of pairwise compatible~$\Gamma$-partitions. Associated with~$\Pi$ is a \emph{blowup}~$\salb{\Pi}$ of the Salvetti complex~$\sal$. The blowup is a cube complex with one (unoriented) hyperplane~$H_\calP$ for each~$\Gamma$-partition~$\calP \in \Pi$, and one (oriented) hyperplane~$H_v$ for each vertex~$v \in V(\Gamma)$. The set of hyperplanes associated to partitions~$\mathcal{H} \coloneqq \{H_\calP\}_{\calP \in \Pi}$ is acyclic, and performing the associated collapse~$\salb{\Pi} \to \salb{\Pi} \ssslash \mathcal{H}$ yields a cube complex isomorphic to the Salvetti complex~$\sal$.

For a full description of the blowing-up process, which allows one to build the blowup~$\salb{\Pi}$ from the Salvetti complex~$\sal$, we refer the reader to \cite{CharneyStambaughVogtmannUntwistedOuterSpace17}. To give some sense of how this construction goes, we give the following examples.

\begin{example}
    \begin{enumerate}[(i)]
        \item If~$\Pi = \emptyset$, then~$\salb{\Pi}$ is simply the Salvetti complex~$\sal$.
        \item If~$\Gamma$ has~$n$ vertices and no edges, then~$\salb{\Pi}$ will be a finite connected graph with no separating edges or bivalent vertices and with fundamental group~$\raag \cong F_n$. The edges dual to the set of hyperplanes~$\{H_\calP\}_{\calP \in \Pi}$ form a maximal tree~$T$, and the edges in~$\salb{\Pi} \setminus T$ are oriented and labelled by~$V(\Gamma)$. Collapsing the maximal tree~$T$ yields a rose with~$n$ petals, each labelled by a generator of~$\raag \cong F_n$ -- this is a copy of the Salvetti complex.
        \item The blowup~$\salb{\calP}$ of the Salvetti complex by a single partition~$\calP$ has the following structure. It has two vertices, say~$x_1$ and~$x_2$. The carrier~$\kappa(H_\calP)$ is isomorphic to the product of an interval with the Salvetti complex for the full subgraph on~$\link{\calP}$. Let~$v \in V(\Gamma)$. If there are two edges dual to~$H_v$, then~$v \in \link{\calP}$. If there is only one edge dual to the (oriented) hyperplane~$H_v$ and it terminates at~$x_i$, then~$v \in P_i$; if the edge originates at~$x_i$ then~$v^{-1} \in P_i$. In this way one can recover~$\calP = \left(P_1 \;|\; P_2 \;|\; \link{\calP} \right)$ from the blowup~$\salb{\calP}$.
    \end{enumerate}
\end{example}

A \emph{$\Gamma$-complex} is a cube complex which is isomorphic to the underlying cube complex of a blowup~$\salb{\Pi}$, for~$\Pi$ a set of pairwise compatible~$\Gamma$-partitions. A \emph{blowup structure} on a~$\Gamma$-complex~$X$ is a choice of labelling (and possibly orientation) of the hyperplanes of~$X$ which identifies it with a blowup~$\salb{\Pi}$. This results in one oriented hyperplane labelled~$H_v$ for each~$v \in V(\Gamma)$, while the remaining hyperplanes are unoriented and labelled by~$\Gamma$-partitions (and these~$\Gamma$-partitions form a pairwise compatible set~$\Pi$). A collapse map~$c_\Pi\colon X \to X \ssslash \{H_\calP\}_{\calP \in \Pi} \cong \sal$ is determined by a blowup structure.

A \emph{marking} on a~$\Gamma$-complex~$X$ is a homotopy equivalence~$\alpha\colon X \to \sal$. Choosing a blowup structure~$\salb{\Pi}$ on~$X$, we obtain a composition~$\sal \xrightarrow{\smash{\alpha^{-1}}} X \xrightarrow{c_\Pi} \sal$. We say that the marking~$\alpha$ is \emph{untwisted} if this composition induces an element of~$\uag$ on the level of fundamental groups.

\newpage 
Two marked~$\Gamma$-complexes~$(X, \alpha)$,~$(X', \alpha')$ are \emph{equivalent} if there is a cube complex isomorphism~$i\colon X' \to X$ with~$\alpha \circ i \simeq \alpha'$:

\[\begin{tikzcd}
	X && {X'} \\
	\\
	& {\mathbb{S}_\Gamma}
	\arrow["{\alpha}"', from=1-1, to=3-2]
	\arrow["{\alpha'}", from=1-3, to=3-2]
	\arrow["i"', from=1-3, to=1-1]
\end{tikzcd}\]

\noindent If~$\alpha$ is an untwisted marking, we call the equivalence class of the pair~$(X, \alpha)$ an \emph{untwisted marked~$\Gamma$-complex}. If~$X \cong \sal$, we call the equivalence class of~$(X, \alpha)$ a \emph{marked Salvetti}. In this paper, `marked~$\Gamma$-complex' will always mean `untwisted marked~$\Gamma$-complex'.

\begin{example}[\cite{CharneyStambaughVogtmannUntwistedOuterSpace17}, \emph{Lemma 3.2}]
    Suppose that~$\Pi = \{\calP\}$. For each~$m \in \mx{\calP}$, there is a cube complex isomorphism~$h_m$ of~$\salb{\calP}$ with the following property. Letting~$c_\calP^{-1}$ be a homotopy inverse for~$c_\calP$, the composition
    \[\smash{c_\calP \circ h_m \circ c_\calP^{-1} \colon \sal \to \salb{\calP} \xrightarrow{\cong} \salb{\calP} \to \sal}\]
    induces the~$\Gamma$-Whitehead automorphism~$\varphi(\calP, m)$ on~$\raag = ~\pi_1\left(\sal\right)$.
\end{example}

An acyclic set~$\mathcal{T}$ of hyperplanes in a~$\Gamma$-complex is called \emph{treelike} if the collapse~$c_\mathcal{T}\colon X \to X \ssslash \mathcal{T}$ gives a cube complex isomorphic to~$\sal$.

\begin{proposition}[\cite{CharneyStambaughVogtmannUntwistedOuterSpace17}, \S4]
    The~$X$ be a~$\Gamma$-complex and let~$\mathcal{T}$ be a treelike set of hyperplanes in~$X$. Then there is a set~$\Pi$ of pairwise compatible~$\Gamma$-partitions and an isomorphism~$X \cong \salb{\Pi}$ such that~$\mathcal{T}$ is the set of hyperplanes labelled by~$\Pi$ in~$\salb{\Pi}$.
\end{proposition}

Hence, any treelike set of hyperplanes is the set labelled by the partitions in at least one blowup structure. The labels and orientations of the other hyperplanes (those labelled by vertices of~$\Gamma$) may be permuted by an automorphism of~$\Gamma$, and changing these assignments changes the partitions labelling the treelike set by the same (signed) permutation of vertices.

To build the \emph{spine}~$\spine$, define a partial order on the set of marked~$\Gamma$-complexes as follows. Let~$(X, \alpha)$ be a marked~$\Gamma$-complex, let~$\mathcal{T}$ be a treelike set of hyperplanes in~$X$, and let~$\mathcal{H} \subseteq \mathcal{T}$. Then set \[(X \ssslash \mathcal{H}, \alpha \circ c_\mathcal{H}^{-1}) \;<\; (X, \alpha).\]

\begin{definition}\label{def:untwisted spine}
    The \emph{spine}~$\spine$ is the geometric realisation of the poset of marked~$\Gamma$-complexes.
\end{definition}

In other words,~$\spine$ is the simplicial complex where each vertex is a marked~$\Gamma$-complex, and each edge represents a hyperplane collapse, which is a homotopy equivalence, from one~$\Gamma$-complex to another. A~$k$-simplex corresponds to a chain of~$k$ such hyperplane collapses.

We will view the spine~$\spine$ as having a cube complex structure, as follows. Fix a set of pairwise compatible~$\Gamma$-partitions~$\Pi = \{\calP_1, \dots, \calP_k\}$. Any ordering of the~$\calP_i$ gives a~$k$-simplex as above, the bottom (with respect to the partial order) vertex of which is a marked Salvetti. The union of all such~$k$-simplices is a~$k$-dimensional cube (see \textit{Figure \ref{fig:cube complex structure of the spine}}), with bottom vertex the marked Salvetti, and top vertex a marked~$\Gamma$-complex with blowup structure~$\salb{\Pi}$. 

Hence the star of a marked Salvetti~$(\sal, \alpha)$ is a cube complex, with one~$k$-cube for each compatible set~$\Pi$ such that~$\vert \Pi \vert = k$. We will denote this cube by~$c\left(\emptyset, \Pi; \alpha\right)$, which specifies its two extreme vertices and the marking with which we view it. Faces of this cube correspond to subsets~$\Pi_1$,~$\Pi_2$ with~$\emptyset \subseteq \Pi_1 \subseteq \Pi_2 \subseteq \Pi$, and will be denoted~$c\left(\Pi_1, \Pi_2; \alpha\right)$. We suppress the marking from the notation when it is not needed. Note that a cube~$c\left(\Pi_1, \Pi_2\right)$ may lie in the star of two different marked Salvettis~$(\sal, \alpha)$ and~$(\sal, \beta)$. In this case, writing~$c\left(\Pi_1, \Pi_2; \alpha\right)$ simply conveys that we are presently thinking of this cube as a subface of a cube in the star of~$(\sal, \alpha)$.

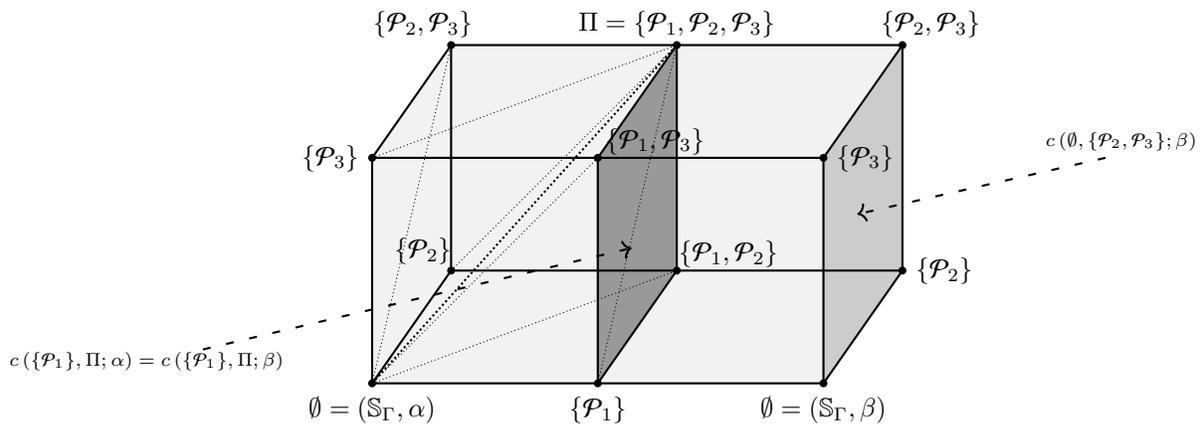
\begin{figure}
    \begin{center}
        \begin{tikzpicture}[scale = 1.5]
            \fill[black!5!white] (0, 0) to (4, 0) to (4.7, 1) to (4.7, 3) to (0.7, 3) to (0, 2) to (0, 0);
            \fill[black!20!white] (4, 0) to (4, 2) to (4.7, 3) to (4.7, 1) to (4, 0);
            \fill[black!40!white] (2, 0) to (2, 2) to (2.7, 3) to (2.7, 1) to (2, 0);
            \fill (0, 0) circle (1pt);
            \fill (2, 0) circle (1pt);
            \fill (4, 0) circle (1pt);
            \fill (0.7, 1) circle (1pt);
            \fill (2.7, 1) circle (1pt);
            \fill (4.7, 1) circle (1pt);
            \fill (0, 2) circle (1pt);
            \fill (2, 2) circle (1pt);
            \fill (4, 2) circle (1pt);
            \fill (0.7, 3) circle (1pt);
            \fill (2.7, 3) circle (1pt);
            \fill (4.7, 3) circle (1pt);
            \draw[black, thick] (0, 0) to (2, 0);
            \draw[black, thick] (2, 0) to (2.7, 1);
            \draw[black, thick] (0, 0) to (0.7, 1);
            \draw[black, thick] (0.7, 1) to (2.7, 1);
            \draw[black, thick] (0, 0) to (0, 2);
            \draw[black, thick] (2, 0) to (2, 2);
            \draw[black, thick] (0.7, 1) to (0.7, 3);
            \draw[black, thick] (2.7, 1) to (2.7, 3);
            \draw[black, thick] (0, 2) to (2, 2);
            \draw[black, thick] (2, 2) to (2.7, 3);
            \draw[black, thick] (0, 2) to (0.7, 3);
            \draw[black, thick] (0.7, 3) to (2.7, 3);
            \draw[black, thick] (2, 0) to (4, 0);
            \draw[black, thick] (2, 2) to (4, 2);
            \draw[black, thick] (2.7, 1) to (4.7, 1);
            \draw[black, thick] (2.7, 3) to (4.7, 3);
            \draw[black, thick] (4, 0) to (4, 2);
            \draw[black, thick] (4, 0) to (4.7, 1);
            \draw[black, thick] (4, 2) to (4.7, 3);
            \draw[black, thick] (4.7, 1) to (4.7, 3);
            \draw[black, densely dotted] (0, 0) to (2.7, 1);
            \draw[black, densely dotted] (0, 0) to (0.7, 3);
            \draw[black, densely dotted] (2, 0) to (2.7, 3);
            \draw[black, densely dotted] (0.7, 1) to (2.7, 3);
            \draw[black, densely dotted] (0, 0) to (2, 2);
            \draw[black, densely dotted] (0, 2) to (2.7, 3);
            \draw[black, thick, densely dotted] (0, 0) to (2.7, 3);
            \node at (0, 0) [below = 1pt, black, thick]{\small$\emptyset = (\sal, \alpha)$};
            \node at (2, 0) [below = 1pt, black, thick]{\small$\{\calP_1\}$};
            \node at (0.45, 0.95) [above = 0.5pt, black, thick]{\small$\{\calP_2\}$};
            \node at (0, 2) [left = 1pt, black, thick]{\small$\{\calP_3\}$};
            \node at (1.95, 2.15) [right = 1pt, black, thick]{\small$\{\calP_1, \calP_3\}$};
            \node at (0.45, 2.95) [above = 1pt, black, thick]{\small$\{\calP_2, \calP_3\}$};
            \node at (2.6, 1.15) [right = 1pt, black, thick]{\small$\{\calP_1, \calP_2\}$};
            \node at (2.7, 2.95) [above = 1pt, black, thick]{\small$\Pi = \{\calP_1, \calP_2, \calP_3\}$};
            \node at (4.95, 2.95) [above = 1pt, black, thick]{\small$\{\calP_2, \calP_3\}$};
            \node at (4.7, 1) [right = 1pt, black, thick]{\small$\{\calP_2\}$};
            \node at (4, 2) [right = 1pt, black, thick]{\small$\{\calP_3\}$};
            \node at (4, 0) [below = 1pt, black, thick]{\small$\emptyset = (\sal, \beta)$};
            \draw[black, thick, loosely dashed, ->] (6.5, 2) to (4.3, 1.5);
            \draw[black, thick, loosely dashed, ->] (-1.5, 0.3) to (2.3, 1.2);
            \node at (5.85, 2.15) [right = 1pt, black, thick]{\tiny$c\left(\emptyset, \{\calP_2, \calP_3\}; \beta\right)$};
            \node at (-2, 0.2) [black, thick]{\tiny$c\left(\{\calP_1\}, \Pi; \alpha\right) = c\left(\{\calP_1\}, \Pi; \beta\right)$};
        \end{tikzpicture}
    \caption{An illustration of the cube complex structure of a portion of~$\spine$, and our notation for the `address' of a cube in~$\spine$. The dotted lines in the cube with bottom vertex~$(\sal, \alpha)$ illustrate the simplicial complex structure of~$\spine$, which we henceforth ignore. The middle shaded face, along which the two shown 3-cubes are identified, may simply be referred to as~$c\left(\{\calP_1\}, \Pi\right)$.}
    \label{fig:cube complex structure of the spine}
    \end{center}
\end{figure}

An automorphism~$\varphi \in \oag$ can be realised by a map~$\overline{\varphi}\colon \sal \to \sal$ (or more precisely, by the homotopy class thereof). Hence, the untwisted subgroup~$\uag$ acts on~$\spine$ on the left by changing the marking:~$\varphi \cdot (X, \alpha) = (X, \overline{\varphi} \circ \alpha)$.

Charney--Stambaugh--Vogtmann \cite{CharneyStambaughVogtmannUntwistedOuterSpace17} also build a full \emph{untwisted Outer space}~$\mathcal{O}_\Gamma^U$ for~$\uag$, in analogy with Culler--Vogtmann Outer space~$CV_n$ for the free group~$F_n$. The spine~$\spine$ naturally sits inside~$\mathcal{O}_\Gamma^U$ as a deformation retract, generalising the way~$K_n$, the spine of~$CV_n$, sits inside reduced Culler--Vogtmann Outer space~$CV_n^{red}$ (the deformation retract of~$CV_n$ obtained by restricting only to graphs with no separating edges). However, we will only be interested in the intrinsic structure of~$\spine$ as the geometric realisation of a poset.

\begin{theorem}[\cite{CharneyStambaughVogtmannUntwistedOuterSpace17}]\label{thm:Charney-Stambaugh-Vogtmann main theorem}
    The spine~$\spine$ is a contractible cube complex, and the action of~$\uag$ on~$\spine$ is proper and cocompact.
\end{theorem}

The geometric control afforded us by \emph{Theorem \ref{thm:vcd upper bound realisation theorem}} yields the immediate corollary
\[\vcduag \leq \dim(\spine).\]

\noindent Observe that the dimension of~$\spine$ is precisely the size of the largest set of pairwise compatible~$\Gamma$-partitions. 

\subsection{Bounds on~$\vcduag$}\label{subsec:bounds on vcd of untwisted subgp}

A thorough analysis of large abelian subgroups of~$\uag$ is made by Millard and Vogtmann in \cite{MillardVogtmannCubeComplexes21}, which provides lower bounds on the virtual cohomological dimension of~$\uag$.

We say that a~$\Gamma$-partition is \emph{principal} if some (hence any) base of it is a principal vertex; define \emph{non-principal partitions} analogously.

Recall (\textit{Definition \ref{def:principal vertex}}) that a vertex~$v$ is principal if there is no other vertex~$w$ with~$v <_\circ w$; that is, with~$\link{v} \subsetneq \link{w}$.

\begin{notation}
    Let~$W \subseteq V(\Gamma)$ be a subset of the vertices of~$\Gamma$. Write~$M(W)$ for the size of the largest set of pairwise compatible~$\Gamma$-Whitehead partitions, such that every partition in the set may be based at a vertex in~$W$. Write~$V = V(\Gamma)$ and let~$L$ denote the set of principal vertices of~$\Gamma$.
\end{notation}

\begin{definition}\label{def:principal rank}
    We call~$M(L)$ the \emph{principal rank} of~$\Gamma$.
\end{definition}

In particular,~$M(V)$ is the maximal size of a set of pairwise compatible~$\Gamma$-partitions, so we have~$\dim(\spine) = M(V)$. Meanwhile,~$M(L)$ is the largest possible dimension of a cube in~$\spine$ whose top vertex corresponds to blowups of the Salvetti complex exclusively by principal partitions.

\begin{theorem}[\cite{MillardVogtmannCubeComplexes21}, Theorem 4.25]\label{thm:Millard-Vogtmann M(L) lower bound}
    The untwisted subgroup~$\uag$ contains a free abelian subgroup of rank~$M(L)$.
\end{theorem}

This provides a lower bound on~$\vcduag$. Combined with the upper bound obtained above, we have the following.

\begin{corollary}\label{corollary a priori bounds on vcd}
    We have~$M(L) \leq \vcduag \leq \dim(\spine) = M(V).$
\end{corollary}

In fact, Millard and Vogtmann go further; they are able to prove that any abelian subgroup generated by~$\Gamma$-Whitehead automorphisms has rank at most~$M(L)$. Hence, if there is some larger-rank free abelian subgroup of~$\uag$, then it is not generated by elements of this natural class. We will see later (\S\ref{subsec:rakes}) that there are graphs~$\Gamma$ for which~$M(V)$ is strictly larger than~$M(L)$. Thus in these cases, as discussed after \emph{Remark \ref{rem:dichotomy about failure of bounds to match}}, either these natural large-rank free abelian subgroups generated by~$\Gamma$-Whitehead automorphisms fail to realise~$\vcduag$, or else the dimension of the untwisted spine~$\spine$ fails to realise~$\vcduag$. 

\section{The retraction process}\label{sec:retraction process}

Fix a finite simplicial graph~$\Gamma$. We now outline a process for performing a~$\uag$-equivariant deformation retraction of~$\spine$, obtaining a new cube complex which has a proper and cocompact action of~$\uag$. Since the only complexes we are considering from here on are the untwisted spine~$\spine$ and retractions of it, when we say `equivariant' we will always mean with respect to the action of the untwisted subgroup~$\uag$.

Henceforth,~$\calP$ will always refer to a principal partition, and~$\calQ$ will always refer to a non-principal partition. We will also frequently use~$m$ and~$n$ to denote a principal vertex, while~$u$ will usually refer to a non-principal vertex. We indicate this rule of thumb to alleviate the burden of some of the notation later; nonetheless we will endeavour to make notation as clear as possible at every point.

In \S\ref{subsec:hugging partitions}, we make our key definition, that of \emph{hugging partitions}. The motivation here is that if a non-principal partition~$\calQ$ is hugged by principal partitions with which it is compatible, then it is very `close' to these principal partitions. Then compatibility of other partitions with~$\calQ$ is determined by their compatibility with the principal partitions which hug~$\calQ$. This is the intuitive notion of `redundancy' to which was referred in the introduction. 

We make this intuition of `redundancy' precise in our key lemmas (\emph{\ref{lem:condition 1 (NEW)} \& \ref{lem:condition 2 (NEW)}}), which then enable us to perform a~$\uag$-equivariant retraction of~$\spine$ in \S\ref{subsec:retraction process}. The key lemmas require certain graph theoretic conditions (\emph{Definitions \ref{def:condition 1 (NEW)} \& \ref{def:condition 2 (NEW)}}), which means we can only apply our retraction process when~$\Gamma$ satisfies these conditions. There is therefore scope to strengthen this technique by improving the notion of `redundancy' to something finer than `hugging'. While we have some partial results in this direction, we have decided that `hugging', presented in \emph{Definition \ref{def:2-hugged in Pi}}, is the best balance between clarity and strength for this document.

\begin{notation}
    For a vertex~$v \in V(\Gamma)$, write~$\calC(v)$ for the set of connected components of~$\Gamma^\pm \setminus \str{v}^\pm$.
\end{notation}

\subsection{Hugging partitions}\label{subsec:hugging partitions}

We now lay the groundwork for our key definition, that of \emph{hugging partitions} (\emph{\ref{def:2-hugged in Pi}}). 

Let~$\calQ$ be a non-principal partition, based at the vertex~$u$. Let~$m >_\circ u$ be a principal vertex, and let~$C \in \calC(u)$ be the component which contains~$m$. Since~$m$ dominates~$u$,~$\calC(m)$ is the union of:
\begin{enumerate}[(i)]
    \item every element of~$\calC(u) \setminus C$;
    \item \noindent$\{u\}$ and~$\{u^{-1}\}$;
    \item those components of~$\Gamma^\pm \setminus \str{m}^\pm$ which are entirely contained in~$C$.
\end{enumerate}
Hence,~$\calQ$ has the following structure:
\begin{itemize}
    \item One side, say~$Q$, contains~$C$ -- and so contains all elements in~$\calC(m)$ of type (iii), as well as~$\{m, m^{-1}\} \cup \left(\link{m}^\pm \setminus \link{u}^\pm\right)$. The rest of the elements of~$Q$ (apart from the base element~$u$ or~$u^{-1}$) are contained in elements of type (i) -- that is, the rest of~$Q$ is a union of elements of~$\calC(u)$.
    \item The other side,~$\overline{Q}$, consists of the remainder of components of type (i) (as well as the other base element~$u$ or~$u^{-1}$). 
\end{itemize}

Write~$\calC^Q(m)$ for those elements of~$\calC(m)$ contained in~$Q$. Partition~$\calC^Q(m)$ into two sets,~$\calC_1$ and~$\calC_2$, and define two partitions~$\calP_1$,~$\calP_2$ by the following sides:
\begin{itemize}
    \item~$P_1 \coloneqq \{m, \calC_1\}$;
    \item~$P_2 \coloneqq \{\overline{m}, \calC_2\}$.
\end{itemize}

If one of these sets~$\calC_i$ is empty, then this does not define a valid partition~$\calP_i$. However, it will be technically convenient for us to ignore this, as we will be arguing entirely combinatorially, and temporarily allow these `almost'-$\Gamma$-Whitehead partitions. We will consistently indicate any juncture at which it is important to distinguish between valid~$\Gamma$-partitions and `almost'-$\Gamma$-partitions. 

Note that~$P_1 \cap P_2 = \emptyset$, so~$\calP_1$ and~$\calP_2$ are compatible with each other; moreover, we have~$P_1 \subseteq Q$ and~$P_2 \subseteq Q$, so both~$\calP_1$ and~$\calP_2$ are compatible with~$\calQ$.

We say that~$\calQ$ is \emph{2-hugged} by~$\calP_1$ and~$\calP_2$. (If, say,~$\calC_1$ was empty, so we ignored~$\calP_1$, then we may say that~$\calQ$ is \emph{1-hugged} (by~$\calP_2)$.)

\begin{definition}\label{def:2-hugged in Pi}
    Let~$\Pi$ be a set of pairwise compatible partitions, and let~$\calQ \in \Pi$ be non-principal. We say that~$\calQ$ is \emph{2-hugged in~$\Pi$} if there exist~$\calP_1, \calP_2 \in \Pi$ which 2-hug~$\calQ$, as above. 
    
    (We will say that~$\calQ$ is \emph{1-hugged in~$\Pi$} if there exists~$\calP \in \Pi$ which 1-hugs~$\calQ$ as above.)

    We will use the term \emph{hugged in~$\Pi$} to mean either 1- or 2-hugged in~$\Pi$, when it is not important to distinguish between these two cases.
\end{definition}

The motivation for this definition is that if~$\calQ$ is hugged in~$\Pi$, then the side~$Q$ is very `close' to the union~$P_1 \cup P_2$ (and indeed, the intersection~$\overline{P_1} \cap \overline{P_2}$ is precisely equal to~$\overline{Q}$). Under certain conditions on~$\Gamma$, this will allow us to conclude that~$\calQ$ is in some sense `redundant' in~$\Pi$. This notion of redundancy is what drives the retraction process. To accustom ourselves to this definition, and the intuition of `closeness', we provide some examples now.

\begin{example}\label{eg:hugging partitions}
    \begin{enumerate}[(a)]
        \item Consider the following graph:
        
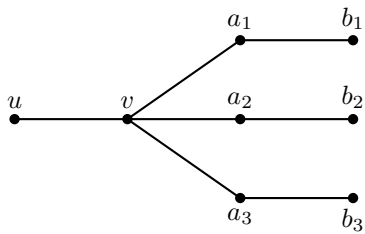
\begin{figure}[H]
            \begin{center}
                \begin{tikzpicture}[scale = 1.5]
                    %
                    %
                    \fill (-1, 0) circle (1.3pt);
                    \fill (0, 0) circle (1.3pt);
                    \fill (1, 0.7) circle (1.3pt);
                    \fill (1, 0) circle (1.3pt);
                    \fill (1, -0.7) circle (1.3pt);
                    \fill (2, 0.7) circle (1.3pt);
                    \fill (2, 0) circle (1.3pt);
                    \fill (2, -0.7) circle (1.3pt);
                    \draw[black, thick] (-1, 0) -- (0, 0);
                    \draw[black, thick] (0, 0) -- (1, 0.7);
                    \draw[black, thick] (0, 0) -- (1, 0);
                    \draw[black, thick] (0, 0) -- (1, -0.7);
                    \draw[black, thick] (1, 0.7) -- (2, 0.7);
                    \draw[black, thick] (1, 0) -- (2, 0);
                    \draw[black, thick] (1, -0.7) -- (2, -0.7);
                    \node at (-1, 0) [above = 0pt, black, thick]{\small$u$};
                    \node at (0, 0) [above = 0pt, black, thick]{\small$v$};
                    \node at (1, 0.7) [above = 0pt, black, thick]{\small$a_1$};
                    \node at (1, 0) [above = 0pt, black, thick]{\small$a_2$};
                    \node at (1, -0.7) [below = 0pt, black, thick]{\small$a_3$};
                    \node at (2, 0.7) [above = 0pt, black, thick]{\small$b_1$};
                    \node at (2, 0) [above = 0pt, black, thick]{\small$b_2$};
                    \node at (2, -0.7) [below = 0pt, black, thick]{\small$b_3$};
                \end{tikzpicture}
            \end{center}
        \end{figure}

        and the partitions

        \begin{figure}[H]
            \begin{center}
                \begin{tikzpicture}
                    %
                    %
                    \node at (1.7, -1.4) [black, thick]{$\calQ = \left( \{u, a_1, a_1^{-1}, b_1, b_1^{-1}, a_2, a_2^{-1}, b_2, b_2^{-1}\} \;\;|\;\; \{u^{-1}, a_3, a_3^{-1}, b_3, b_3^{-1}\} \;\;|\;\; \{v, v^{-1}\} \right)$};
                    \node at (1.7, -2.1) [black, thick]{$\calP = \left( \{a_2, u, a_1, a_1^{-1}, b_1, b_1^{-1}\} \;\;|\;\; \{a_2^{-1}, u^{-1}, a_3, a_3^{-1}, b_3, b_3^{-1}\} \;\;|\;\; \{v, v^{-1}, b_2, b_2^{-1}\} \right)$};
                    %
                    %
                    \draw[black, thick, rounded corners, densely dotted, fill = black!10!white] 
                    (-0.1, -2.9) to (3.8, -2.9) to (3.8, -5.25) to (-0.5, -5.25) to (-0.5, -3.9) to (-1.7, -3.9) to (-1.7, -2.9) to (-0.1, -2.9);
                    \draw[black, thick, rounded corners, densely dotted, fill = black!20!white] 
                    (-0.1, -3.05) to (2.7, -3.05) to (2.7, -3.8) to (1.5, -3.8) to (1.5, -5.1) to (-0.35, -5.1) to (-0.35, -3.8) to (-1.55, -3.8) to (-1.55, -3.05) to (-0.1, -3.05);
                    \fill (-2.5, -3.3) circle (2pt);
                    \fill (-1.3, -3.3) circle (2pt);
                    \fill (-0.1, -3.3) circle (2pt);
                    \fill (1.1, -3.3) circle (2pt);
                    \fill (2.3, -3.3) circle (2pt);
                    \fill (3.5, -3.3) circle (2pt);
                    \fill (4.7, -3.3) circle (2pt);
                    \fill (5.9, -3.3) circle (2pt);
                    \fill (-2.5, -4.8) circle (2pt);
                    \fill (-1.3, -4.8) circle (2pt);
                    \fill (-0.1, -4.8) circle (2pt);
                    \fill (1.1, -4.8) circle (2pt);
                    \fill (2.3, -4.8) circle (2pt);
                    \fill (3.5, -4.8) circle (2pt);
                    \fill (4.7, -4.8) circle (2pt);
                    \fill (5.9, -4.8) circle (2pt);
                    \node at (-2.5, -3.83) [above = 1pt, black, thick]{$v$};
                    \node at (-1.3, -3.83) [above = 1pt, black, thick]{$u$};
                    \node at (-0.1, -3.9) [above = 1pt, black, thick]{$a_1$};
                    \node at (1.1, -3.9) [above = 1pt, black, thick]{$b_1$};
                    \node at (2.3, -3.9) [above = 1pt, black, thick]{$a_2$};
                    \node at (3.5, -3.9) [above = 1pt, black, thick]{$b_2$};
                    \node at (4.7, -3.9) [above = 1pt, black, thick]{$a_3$};
                    \node at (5.9, -3.9) [above = 1pt, black, thick]{$b_3$};
                    \node at (0.5, -4.85) [black, thick]{$P$};
                    \node at (2.9, -5.05) [black, thick]{$Q$};
                \end{tikzpicture}
                \label{fig:example of 1-hugging}
            \end{center}
        \end{figure}
        Here the set~$\{\calP\}$ 1-hugs~$Q$. We have~$a_2 >_\circ u$, and we have partitioned~$\calC^Q(a_2)$ as follows:~$\calC_1 = \{\{u\}, \{a_1, a_1^{-1}, b_1, b_1^{-1}\}\}$ and~$\calC_2 = \emptyset$. 
        \item With the same~$\Gamma$ and non-principal partition~$\calQ$, we present an example of 2-hugging:
        \begin{figure}[H]
            \begin{center}
                \begin{tikzpicture}
                    %
                    %
                    \node at (1.7, -1.4) [black, thick]{$\calP_1 = \left( \{a_1, u\} \;\;|\;\; \{a_1^{-1}, u^{-1}, a_2, a_2^{-1}, b_2, b_2^{-1}, a_3, a_3^{-1}, b_3, b_3^{-1}\} \;\;|\;\; \{v, v^{-1}, b_1, b_1^{-1}\} \right)$};
                    \node at (1.7, -2.1) [black, thick]{$\calP_2 = \left( \{a_1^{-1}, a_2, a_2^{-1}, b_2, b_2^{-1}\} \;\;|\;\; \{a_1, u, u^{-1}, a_3, a_3^{-1}, b_3, b_3^{-1}\} \;\;|\;\; \{v, v^{-1}, b_1, b_1^{-1}\} \right)$};
                    %
                    %
                    \draw[black, thick, rounded corners, densely dotted, fill = black!10!white] 
                    (-0.1, -2.9) to (3.8, -2.9) to (3.8, -5.25) to (-0.5, -5.25) to (-0.5, -3.9) to (-1.7, -3.9) to (-1.7, -2.9) to (-0.1, -2.9);
                    \draw[black, thick, rounded corners, densely dotted, fill = black!20!white] (-0.1, -3.05) to (0.2, -3.05) to (0.2, -3.8) to (-1.55, -3.8) to (-1.55, -3.05) to (-0.1, -3.05);
                    %
                    \draw[black, thick, rounded corners, densely dotted, fill = black!20!white] (2.3, -3.05) to (3.7, -3.05) to (3.7, -5.1) to (2, -5.1) to (2, -4.6) to (0.4, -4.6) to (0.4, -5.1) to (-0.35, -5.1) to (-0.35, -4.4) to (2, -4.4) to (2, -3.05) to (2.3, -3.05);
                    %
                    \fill (-2.5, -3.3) circle (2pt);
                    \fill (-1.3, -3.3) circle (2pt);
                    \fill (-0.1, -3.3) circle (2pt);
                    \fill (1.1, -3.3) circle (2pt);
                    \fill (2.3, -3.3) circle (2pt);
                    \fill (3.5, -3.3) circle (2pt);
                    \fill (4.7, -3.3) circle (2pt);
                    \fill (5.9, -3.3) circle (2pt);
                    \fill (-2.5, -4.8) circle (2pt);
                    \fill (-1.3, -4.8) circle (2pt);
                    \fill (-0.1, -4.8) circle (2pt);
                    \fill (1.1, -4.8) circle (2pt);
                    \fill (2.3, -4.8) circle (2pt);
                    \fill (3.5, -4.8) circle (2pt);
                    \fill (4.7, -4.8) circle (2pt);
                    \fill (5.9, -4.8) circle (2pt);
                    \node at (-2.5, -3.83) [above = 1pt, black, thick]{$v$};
                    \node at (-1.3, -3.83) [above = 1pt, black, thick]{$u$};
                    \node at (-0.1, -3.9) [above = 1pt, black, thick]{$a_1$};
                    \node at (1.1, -3.9) [above = 1pt, black, thick]{$b_1$};
                    \node at (2.3, -3.9) [above = 1pt, black, thick]{$a_2$};
                    \node at (3.5, -3.9) [above = 1pt, black, thick]{$b_2$};
                    \node at (4.7, -3.9) [above = 1pt, black, thick]{$a_3$};
                    \node at (5.9, -3.9) [above = 1pt, black, thick]{$b_3$};
                    \node at (-0.7, -3.6) [black, thick]{$P_1$};
                    \node at (1.4, -4) [black, thick]{$Q$};
                    \node at (2.9, -4.9) [black, thick]{$P_2$};
                \end{tikzpicture}
                \label{fig:example of 2-hugging}
            \end{center}
        \end{figure}
        Here,~$\calP_1$ and~$\calP_2$ are both based at~$a_1 >_\circ u$, and we have partitioned~$\calC^Q(a_1)$ into~$\calC_1 = \{\{u\}\}$ and~$\calC_2 = \{\{a_2, a_2^{-1}, b_2, b_2^{-1}\}\}$. 
    \end{enumerate}
\end{example}

\subsection{Preparatory lemmas}\label{subsec:preparatory lemmas}

In the next subsection \S\ref{subsec:key lemmas}, we will prove two key lemmas (\emph{Lemma \ref{lem:condition 1 (NEW)}} \& \emph{Lemma \ref{lem:condition 2 (NEW)}}) which will allow us to perform an equivariant retraction of~$\spine$ (under certain conditions). In this subsection, we collect some statements we will find useful.

Recall that for a subset of vertices~$W \subseteq V(\Gamma)$,~$M(W)$ denotes the maximum possible size of a compatible set of~$\Gamma$-partitions. 

\begin{lemma}[\cite{MillardVogtmannCubeComplexes21}, \textit{Lemma 5.1}]\label{lem:MV19, 5.1}
    If non-equivalent vertices~$u, v \in V$ have~$d_\Gamma(u, v) \neq 2$, then any partition based at~$u$ is compatible with any partition based at~$v$. In particular, \[M(\{u, v\}) = M(\{u\}) + M(\{v\}).\]
\end{lemma}

\begin{lemma}[\cite{CharneyStambaughVogtmannUntwistedOuterSpace17}, \textit{Lemma 3.4}]\label{lem:distance >= 2 implies containment of sides}
    Let~$\calR_1$,~$\calR_2$ be compatible partitions based at~$v_1$,~$v_2$ respectively. Suppose that~$d_\Gamma(v_1, v_2) \geq 2$. Then one side of~$\calR_1$ is contained in one side of~$\calR_2$. 
\end{lemma}

\begin{remark}\label{rem:proof of CSV17, 3.4}
    The next lemma appears in \cite{CharneyStambaughVogtmannUntwistedOuterSpace17} as \emph{Lemma 3.4}; we have rephrased the proof here for ease of comparison.
\end{remark}

\begin{proof}
    Since~$\calR_1$ and~$\calR_2$ are compatible and not adjacent, we know that there are sides~$R_1$,~$R_2$ of~$\calR_1$,~$\calR_2$ respectively such that~$R_1 \subseteq R_2 \cup \link{v_2}^\pm$.

    If~$v_2$ and~$v_2^{-1}$ are in separate components of~$\Gamma^\pm \setminus \str{v_1}^\pm$, then ~$\link{v_1} \supseteq \link{v_2}$ so~$R_1 \cap \link{v_2}^\pm = \emptyset$, and we're done. So we can suppose that~$\{v_2, v_2^{-1}\}$ is contained in one component~$C \in \calC(v_1)$.

    If~$R_1 \supseteq C$, then we have \[\{v_2, v_2^{-1}\} \subseteq C \subseteq R_1 \subseteq R_2 \cup \link{v_2}^\pm,\] which implies that~$R_2 \cap \overline{R_2} \neq \emptyset$, a contradiction. Hence~$C \subseteq \overline{R_1}$.

    If~$R_1 \cap \link{v_2}^\pm = \emptyset$, then we're done (as~$R_1 \subseteq R_2$), so suppose not; let~$x \in R_1 \cap \link{v_2}^\pm$. Since~$x \in R_1$, we must have~$d_\Gamma(x, v_1) \geq 2$. On the other hand,~$x$ is adjacent to~$v_2$, so~$x$ and~$v_2$ lie in the same element of~$\calC(v_1)$. Hence~$x \in C \subseteq \overline{R_1}$, a contradiction, as we assumed~$x \in R_1$. Hence~$R_1 \cap \link{v_2}^\pm = \emptyset$, and so~$R_1 \subseteq R_2$.
\end{proof}

Note that the statement of the previous lemma is symmetric; we can also conclude~$\overline{R_2} \subseteq \overline{R_1}$.

\begin{lemma}\label{lem:Lemma A, pg 87}
    Let~$\calQ$ be a non-principal partition, based at~$u$, and hugged by the set~$\{\calP_1, \calP_2\}$; suppose that the~$\calP_i$ are based at a vertex~$m$ in the component~$C \in \calC(u)$. Let~$\calR$, based at~$v$, be any partition. Suppose that~$\calR$ is compatible with both~$\calP_1$ and~$\calP_2$ but not with~$\calQ$. Then~$d_\Gamma(u, v) = 2$, and~$v \in C$.
\end{lemma}

\begin{proof}
    Let~$\calP_1$ and~$\calP_2$ be based at~$m$, and let~$Q$ be the side of~$\calQ$ which is hugged by the sides~$P_1$ and~$P_2$ of~$\calP_1$ and~$\calP_2$ respectively.
    
    By \emph{Lemma \ref{lem:MV19, 5.1}}, we know that~$d_\Gamma(u, v) \in \{0, 2\}$. We first dispense with the case~$u = v$.

    \begin{claim*}
        ~$u \neq v$.
    \end{claim*}

    \begin{proof}[Proof of Claim.]
        Suppose that~$\calR$ is based at~$u$. We have~$u \in P_1 \sqcup P_2$; without loss of generality take~$u \in P_1$. Let~$R$ be the side of~$\calR$ which contains~$u$. If~$R \subseteq P_1$, then~$R \subseteq Q$, which contradicts incompatibility of~$\calR$ and~$\calQ$. Hence~$R$ contains either~$P_1$ or~$\overline{P_1}$. We cannot have~$R \supseteq \overline{P_1}$, as~$\overline{P_1} \ni u^{-1}$, and~$u^{-1} \notin R$ (since~$\calR$ is based at~$u$). Hence~$R \supseteq P_1$. However,~$\calR$ cannot split~$m$ (since~$m$ is principal), so since~$m \in R$, we must also have~$m^{-1} \in R$. Hence~$R$ has non-empty intersection with both sides of~$\calP_2$. Since~$u^{-1} \in \overline{R} \cap \overline{P_2}$, we must therefore have~$\overline{R} \cap P_2 = \emptyset$ and so~$P_2 \subseteq R$. Hence~$R$ contains all of~$\calC^Q(u)$, so~$R \supseteq Q$, and once again we have contradicted the incompatibility of~$\calR$ and~$\calQ$. 
    \end{proof}

    This proves the claim;~$u \neq v$. Hence~$d_\Gamma(u, v) = 2$, as required.

    Suppose that~$v \notin C$. Then if~$d_\Gamma(v, m) = 1$, we have~$v \in \link{u}$, which means that~$\calR$ and~$\calQ$ are adjacent (and hence compatible), a contradiction. Hence~$d_\Gamma(v, m) \geq 2$, so one side of~$\calR$ either contains, or is contained in,~$\overline{P_1}$. If the former, then this side also contains~$\overline{Q}$, which contradicts incompatibility of~$\calR$ with~$\calQ$. Hence there is some side~$R$ of~$\calR$ which is contained in~$\overline{P_1}$. 
    
    We may apply the same argument to~$\calP_2$. Now, if~$R \subseteq \overline{P_2}$, then~$R \subseteq \overline{P_1} \cap \overline{P_2} = \overline{Q}$, once again contradicting incompatibility of~$\calR$ and~$\calQ$. Hence we must have~$\overline{R} \subseteq \overline{P_2}$, which implies that~$P_2 \subseteq R$. 
    
    If~$\calR$ does not split~$m$, then~$P_2 \subseteq R$ implies that~$\{m, m^{-1}\} \subseteq R \subseteq \overline{P_1}$, another contradiction. Therefore~$\calR$ must split~$m$, but now since~$m$ is principal, it can only be split by partitions based at vertices equivalent to it, so~$v \sim m$. Since~$d_\Gamma(u, v) = 2$, and~$\link{m} \setminus \link{u} \neq \emptyset$, we must have~$v \in C$, as required.
\end{proof}

\subsection{Key Lemmas}\label{subsec:key lemmas}

In this section, we present the two key lemmas which are the driving force behind the retraction process. Each requires a certain technical condition on~$\Gamma$. 

We will say that a vertex of~$\Gamma$ is \emph{relevant} if it can be the base of a~$\Gamma$-Whitehead automorphism (and hence of a~$\Gamma$-partition). This terminology is simply for brevity, as we will not be concerned with vertices which cannot serve as a base of a~$\Gamma$-partition.

\begin{definition}\label{def:condition 1 (NEW)}
    We say that~$\Gamma$ satisfies \emph{Condition 1} if for every non-principal vertex~$u$, and every principal~$m >_\circ u$, there is no non-principal vertex~$u'$ with~$d_\Gamma(u, u') = 2$ such that~$[u', m] = 1$.
\end{definition}

\begin{remark}
    Note that if two vertices~$v$,~$w$ satisfy~$[v, w] = 1$, then we have~$d_\Gamma(v, w) \in \{0, 1\}$ (so in particular, if~$v \neq w$, then~$[v, w] = 1 \iff d_\Gamma(v, w) = 1$ (and equivalently, we have~$v \in \link{w}$ and~$w \in \link{v})$). Hence the end of the preceding definition could be rephrased as ``there is no non-principal vertex~$u'$ with~$d_\Gamma(u, u') = 2$ and~$d_\Gamma(u', m) = 1$''. However, we will find it useful to interchangeably use both notions throughout the document, so we will use both notations. 
\end{remark}

\begin{example}\label{eg:graph that breaks Condition 1 and the ensuing lemma}
    The following is an example of a graph~$\Gamma$ which does not satisfy Condition 1.

    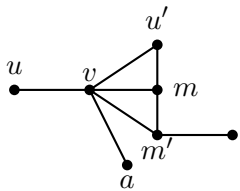
\begin{figure}[H]
        \begin{center}
            \begin{tikzpicture}
            \fill (-1, 0) circle (2pt); 
            \fill (0, 0) circle (2pt); 
            \fill (0.9, 0.6) circle (2pt); 
            \fill (0.9, 0) circle (2pt); 
            \fill (0.9, -0.6) circle (2pt); 
            \fill (0.5, -1) circle (2pt); 
            \fill (1.9, -0.6) circle (2pt);
            \draw[black, thick] (-1, 0) -- (0, 0);
            \draw[black, thick] (0, 0) -- (0.9, 0.6);
            \draw[black, thick] (0, 0) -- (0.9, 0);
            \draw[black, thick] (0, 0) -- (0.9, -0.6);
            \draw[black, thick] (0, 0) -- (0.5, -1); 
            \draw[black, thick] (0.9, 0.6) -- (0.9, 0);
            \draw[black, thick] (0.9, 0) -- (0.9, -0.6);
            \draw[black, thick] (0.9, -0.6) -- (1.9, -0.6);
            \node at (-1, 0) [above = 2pt, black, thick]{$u$};
            \node at (0.9, 0.6) [above = 2pt, black, thick]{$u'$};
            \node at (0.9, 0) [right = 2pt, black, thick]{$m$};
            \node at (0.9, -0.45) [below = 1pt, black, thick]{$m'$};
            \node at (0.5, -0.93) [below = 2pt, black, thick]{$a$};
            \node at (0, 0.2) [black, thick]{$v$};
            \end{tikzpicture}
        \end{center}
        \caption{A configuration violating Condition 1.}
        \label{fig:doesn't satisfy Condition 1}
    \end{figure}

    Here,~$m$ and~$m'$ are principal, with~$m >_\circ u$ and~$m' >_\circ u'$ (so~$u$ and~$u'$ are non-principal); we have~$d_\Gamma(u, u') = 2$ and~$[u', m] = 1$, so this graph does not satisfy Condition 1. We remark that in this example, the unique vertex~$v$ in~$\link{u}$ is principal and commutes with~$u'$ but is not~$>_\circ u$, while~$a$ is non-principal and satisfies~$d_\Gamma(u, a) = 2$, but does not commute with either of~$m$ or~$m'$ (which are both~$>_\circ u$).
\end{example}

\begin{lemma}\label{lem:condition 1 (NEW)}
    Suppose that~$\Gamma$ satisfies Condition 1. Let~$\calQ$ be a non-principal partition, hugged by the set~$\{\calP_1, \calP_2\}$. Let~$\calQ'$ be a non-principal partition, which is compatible with both~$\calP_1$ and~$\calP_2$. Then~$\calQ'$ is compatible with~$\calQ$.
\end{lemma}

\begin{proof}
    Denote the bases of~$\calQ$ and~$\calQ'$ by~$u$ and~$u'$ respectively, the base of~$\calP_1$ and~$\calP_2$ by~$m$, and let~$C \in \calC(u)$ be the component containing~$m$.

    By the contrapositive of \emph{Lemma \ref{lem:Lemma A, pg 87}}, if either~$u' \notin C$ or~$d_\Gamma(u, u') \neq 2$, then~$\calQ'$ is compatible with~$\calQ$, and we're done, so assume that~$u' \in C$ and that~$d_\Gamma(u, u') = 2$. Moreover, since~$\Gamma$ satisfies Condition 1, we may assume that~$u'$ does not commute with~$m$ (so in particular,~$\calQ'$ is not adjacent to~$\calP_1$ or~$\calP_2$).

   ~$\calQ'$ is compatible with~$\calP_1$, so~$P_1$ either contains, or is contained in, a side~$Q'$ of~$\calQ'$. If the former, then~$Q' \subseteq P_1 \subseteq Q$, so~$\calQ$ and~$\calQ'$ are compatible, and we're done. So suppose the latter; then~$Q'$ contains~$\{m, m^{-1}\}$, and so has non-empty intersection with both sides of~$\calP_2$. Therefore,~$Q'$ contains one of the sides of~$\calP_2$. If~$P_2 \subseteq Q'$, then~$Q' \supseteq P_1 \sqcup P_2 \supseteq \{u', u'^{-1}\}$, which is impossible, so we must have~$\overline{Q'} \supseteq \overline{P_2} \supseteq Q$, whence~$\calQ'$ is compatible with~$\calQ$.
\end{proof}

We note that the conclusion of the lemma does not hold for the graph shown in \emph{Figure \ref{fig:doesn't satisfy Condition 1}}. We have~$m$ principal, while~$u$ and~$u'$ are both non-principal (since~$m >_\circ u$ and~$m' >_\circ u'$). The~$u$-partition~$\calQ$ determined by the side~$\{u, a\}$ is 1-hugged by the~$m$-partition~$\calP$ determined by the side~$\{m, u, a\}$. On the other hand,~$u'$ and~$m$ commute, so any~$u'$-partition is compatible with any~$m$-partition, by definition. So in particular, the~$u'$-partition~$\calQ'$ given by the side~$\{u', u\}$ is compatible with~$\calP$, but is not compatible with~$\calQ$.

\begin{definition}\label{def:condition 2 (NEW)}
    Say that~$\Gamma$ satisfies \emph{Condition 2} if, for every non-principal vertex~$u$, every principal vertex~$m >_\circ u$, and every principal vertex~$n \neq m$ with~$d_\Gamma(u, n) = 2$, we have~$[m, n] \neq 1$.
\end{definition}

In other words, in every component of~$\Gamma \setminus \str{u}$ which contains some principal~$m >_\circ u$, there is no other principal vertex~$n$ close to~$u$ and commuting with~$m$. Equivalently, every principal vertex~$n$ with~$d_\Gamma(u, n) = 2$ satisfies~$d_\Gamma(m, n) \geq 2$. Intuitively this amounts to connected components of~$\Gamma \setminus \str{u}$ not having too many vertices close to~$u$, as can be seen in the following example.

\begin{example}
    The following graph violates Condition 2.
    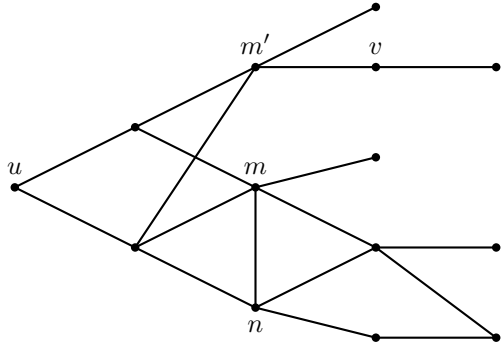
\begin{figure}[H]
        \begin{center}
            \begin{tikzpicture}[scale = 0.8]
                \fill (-0.5, 0) circle (2pt);
                \fill (1.5, 1) circle (2pt);
                \fill (1.5, -1) circle (2pt);
                \fill (3.5, 0) circle (2pt);
                \fill (3.5, 2) circle (2pt);
                \fill (3.5, -2) circle (2pt);
                \fill (5.5, 3) circle (2pt);
                \fill (5.5, 2) circle (2pt);
                \fill (5.5, 0.5) circle (2pt);
                \fill (5.5, -1) circle (2pt);
                \fill (5.5, -2.5) circle (2pt);
                \fill (7.5, -2.5) circle (2pt);
                \fill (7.5, -1) circle (2pt);
                \fill (7.5, 2) circle (2pt);
                \draw[black, thick] (-0.5, 0) -- (1.5, 1);
                \draw[black, thick] (-0.5, 0) -- (1.5, -1);
                \draw[black, thick] (1.5, -1) -- (3.5, 0);
                \draw[black, thick] (1.5, -1) -- (3.5, -2);
                \draw[black, thick] (1.5, -1) -- (3.5, 2);
                \draw[black, thick] (1.5, 1) -- (3.5, 0);
                \draw[black, thick] (1.5, 1) -- (3.5, 2);
                \draw[black, thick] (3.5, 0) -- (3.5, -2);
                \draw[black, thick] (3.5, -2) -- (5.5, -2.5);
                \draw[black, thick] (3.5, 0) -- (5.5, 0.5);
                \draw[black, thick] (3.5, 2) -- (5.5, 2);
                \draw[black, thick] (3.5, 2) -- (5.5, 3);
                \draw[black, thick] (3.5, -2) -- (5.5, -1);
                \draw[black, thick] (3.5, 0) -- (5.5, -1);
                \draw[black, thick] (5.5, -2.5) -- (7.5, -2.5);
                \draw[black, thick] (5.5, -1) -- (7.5, -2.5);
                \draw[black, thick] (5.5, -1) -- (7.5, -1);
                \draw[black, thick] (5.5, 2) -- (7.5, 2);
                \node at (-0.5, 0) [black, thick, above = 1pt]{\small{$u$}};
                \node at (3.5, 2) [black, thick, above = 1pt]{\small{$m'$}};
                \node at (3.5, 0) [black, thick, above = 1pt]{\small{$m$}};
                \node at (3.5, -2) [black, thick, below = 1pt]{\small{$n$}};
                \node at (5.5, 2) [black, thick, above = 1pt]{\small{$v$}};
            \end{tikzpicture}
            \caption{Here we have~$u$ non-principal and relevant with~$u <_\circ m$ principal, and~$n$ principal with~$d_\Gamma(u, n) = 2$ and~$[m, n] = 1$.
            }
            \label{fig:graph violating Condition 2}
        \end{center}
    \end{figure}
\end{example}

\begin{lemma}\label{lem:condition 2 (NEW)}
    Suppose that~$\Gamma$ satisfies Condition 2. Let~$\Pi$ be a compatible set of partitions containing distinct non-principal partitions~$\calQ$ and~$\calQ'$. Suppose that~$\calQ$ is hugged in~$\Pi$ by~$\{\calP_1, \calP_2\}$, and~$\calQ'$ is hugged in~$\Pi$ by~$\{\calP_1', \calP_2'\}$. Suppose that~$\calR$, some principal partition (not necessarily in~$\Pi$), is compatible with~$\{\calP_1, \calP_2, \calP_1', \calP_2'\}$. Then~$\calR$ is in fact compatible with at least one of~$\{\calQ, \calQ'\}$. 
\end{lemma}

\begin{proof}
    Let~$\calQ$ be based at~$u$, and~$\calQ'$ be based at~$u'$. Let~$\calP_1$ and~$\calP_2$ be based at~$m >_\circ u$, and let~$\calP_1'$ and~$\calP_2'$ be based at~$m' >_\circ u'$. Let~$C \in \calC(u)$ be the component containing~$m$. Let~$Q$ be the side of~$\calQ$ hugged by the sides~$P_1$ and~$P_2$ of~$\calP_1$ and~$\calP_2$; take the analogous notation~$Q'$,~$P_1'$,~$P_2'$. 
    
    Let~$\calR$ be based at the principal vertex~$n$. 

    If~$d_\Gamma(u, n) \neq 2$, then by \emph{Lemma \ref{lem:MV19, 5.1}},~$\calR$ is compatible with~$\calQ$, so suppose that~$d_\Gamma(u, n) = 2$.

    \textbf{Case I:}~$n \not\sim m$.

    Since~$\Gamma$ satisfies Condition 2,~$[n, m] \neq 1$, so~$\calR$ and~$\calP_1$ are not adjacent. Nevertheless,~$\calR$ is compatible with~$\calP_1$, so there must be a choice of sides~$P_1^\times, R^\times$ such that~$P_1^\times \cap R^\times = \emptyset$. If~$P_1^\times = \overline{P_1}$, then since~$\overline{Q} \subseteq \overline{P_1}$ we have~$R^\times \cap \overline{Q} = \emptyset$, and we're done. So we may take~$P_1^\times = P_1$; let~$R$ be the side of~$\calR$ with~$R \cap P_1 = \emptyset$.
    
    Similarly, we may assume that some side of~$\calR$ has empty intersection with~$P_2$. We show that~$R \cap P_2 = \emptyset$. Without loss of generality, assume~$R$ is the~$n$-side of~$\calR$. Since~$R \cap P_1 = \emptyset$,~$n \notin P_1$ (and so~$n \in \overline{P_1}$). By \emph{Lemma \ref{lem:Lemma A, pg 87}}, if~$n \notin C$,~$\calR$ is automatically compatible with~$\calQ$, so we may assume that~$n \in C$. Hence~$n \in Q$, so (since~$[n, m] \neq 1$) we have~$n \in P_2$, so~$n \in \overline{P_1} \cap P_2$. Since~$n \not\sim m$ and~$n$ is principal,~$n$ cannot be split by~$\calP_1$, so~$n^{-1} \in \overline{P_1} \cap P_2$ -- in particular, we cannot have~$\overline{R} \cap P_2 = \emptyset$. 

    Hence~$R \cap P_2 = \emptyset$. We already have~$R \cap P_1 = \emptyset$, so by \emph{Lemma \ref{lem:distance >= 2 implies containment of sides}},~$R \subseteq \left( \overline{P_1} \cap \overline{P_2} \right) = \overline{Q}$, which implies that~$\calR$ is compatible with~$\calQ$. 

    Similarly,~$n \not\sim m'$ implies that~$\calR$ is compatible with~$\calQ'$. We are left with the following case.

    \textbf{Case II:}~$n \sim m \sim m'$.

    Since~$\Gamma$ satisfies Condition 2, any distinct pair of these three vertices do not commute (although we could have any pair being equal to each other). Hence~$\link{n} = \link{m} = \link{m'}$.
    
    Note that since~$m >_\circ u$ and~$m >_\circ u'$, we cannot have~$[u, u'] = 1$: otherwise, we would have~$u' \in \link{u} \subseteq \link{m}$, which precludes~$m >_\circ u'$.

    Let~$R$ be the side of~$\calR$ containing~$m$. If either side of~$\calR$ is contained in~$P_1$, then that side is also contained in~$Q$, so we're done. We cannot have~$P_1 \subseteq \overline{R}$ (since~$m \in P_1$), so we may assume that~$P_1 \subseteq R$. Similarly, we may assume that~$P_1' \subseteq R$ and, applying the same arguments again, that~$P_2 \cup P_2' \subseteq \overline{R}$. 
    
    Consider the side~$Q'$ of~$\calQ'$ which contains~$m$. Principality of~$m$ implies that~$\calQ'$ does not split~$m$, so~$\{m, m^{-1}\} \subseteq Q'$, which means that~$Q'$ intersects both sides of~$\calP_1$. Hence, since~$\calQ'$ and~$\calP_1$ are compatible, one side of~$\calP_1$ must have empty intersection with~$\overline{Q'}$. If~$\overline{P_1} \cap \overline{Q'} = \emptyset$, then we have~$\overline{Q'} \subseteq P_1 \subseteq R$, and we're done. So assume that~$P_1 \cap \overline{Q'} = \emptyset$. In this case,~$Q'$ must contain~$P_1$. Applying the same argument to~$P_2$, we see that either~$\overline{P_2} \cap \overline{Q'} = \emptyset$ (in which case~$\overline{Q'} \subseteq P_2 \subseteq \overline{R}$, and we're done) or~$P_2 \subseteq Q'$ -- take this to be the case. But now~$Q' \supseteq \left(P_1 \cup P_2\right)$, which means that 
    \begin{equation}\label{eqn:Q in Q'}
        Q' \supseteq Q \setminus \left( \link{u'}^\pm \cap Q\right);
    \end{equation}
    that is,~$Q'$ contains all of~$Q$, apart from possibly any elements of~$\link{u'}^\pm \subseteq \link{m}^\pm$ which happen to be contained in~$Q$.

    The preceding paragraph used nothing other than the compatibility of~$\calQ'$ with~$\calP_1$ and~$\calP_2$. We may therefore run the same argument using compatibility of~$\calQ$ with~$\calP_1'$ and~$\calP_2'$, concluding that 
    \begin{equation}\label{eqn:Q' in Q}
        Q \supseteq Q' \setminus \left( \link{u}^\pm \cap Q'\right).
    \end{equation}
    Suppose, without loss of generality, that~$u \in Q$ and~$u' \in Q'$. Then by (\ref{eqn:Q' in Q}),~$u' \in Q$, but necessarily~$u'^{-1} \notin Q$ -- as otherwise~$u'^{-1} \in Q'$, which is impossible. Hence~$\calQ$ splits~$u'$, which means that~$u' \leq_\circ u$. Switching~$u$ and~$u'$ in this argument gives that~$u \leq_\circ u'$ -- hence~$\link{u} = \link{u'}$. But now we have~$Q \cap \link{u'}^\pm = Q' \cap \link{u}^\pm = \emptyset$, so~$Q = Q'$. This contradicts the hypothesis that~$\calQ$ and~$\calQ'$ are distinct partitions. This finishes the argument.
\end{proof}

It will be useful to rephrase the statement of this lemma as follows. 
\begin{corollary}\label{cor:restatement of Key Lemma 2}
    Suppose that~$\Gamma$ satisfies Condition 2. Let~$\Pi$ be a compatible set of partitions, and let~$H \subseteq \Pi$ be the set of non-principal partitions in~$\Pi$ which are hugged in~$\Pi$. Suppose that~$\calR$ is a principal partition which is compatible with~$\Pi \setminus H$. Then~$\calR$ is incompatible with at most one element of~$H$. 
    
    \noindent Therefore, if~$\calR$ is incompatible with some~$\calQ \in H$, then it can replace~$\calQ$ -- i.e.,~$\Pi' \coloneqq \left( \Pi \setminus \calQ \right) \cup \calR$ is a compatible set. 
\end{corollary}

\begin{proof}
    The statement in the lemma follows immediately from the statement here. For the converse, notice that if~$\calR$ were incompatible with two elements of~$H$, then we could take them to be~$\calQ$ and~$\calQ'$ in the lemma, which would be a contradiction.
\end{proof}

We point out here that these arguments never used the requirement that~$\Gamma$-partitions are thick (that is, have at least two elements on each side). Therefore, although it was never explicitly mentioned, one may check that these arguments work even if the hugged non-principal partitions concerned are 1-hugged but not 2-hugged; indeed, in some cases, the arguments can be simplified somewhat.

\subsection{The retraction process}\label{subsec:retraction process}

In this and the next subsection we prove \textit{Theorem \ref{customthm:main work} (i)}, which we restate as follows.

\begin{theorem}\label{thm:retraction process (D,(i))}
    If~$\Gamma$ is spiky, then there is a~$\uag$-equivariant deformation retraction of~$\spine$ to a cube complex~$\res$. The cubes of~$\spine$ which survive this retraction can be characterised in terms of hugging partitions. 
\end{theorem}

We begin by setting out the necessary notation and giving a synopsis of the inductive process by which we will achieve this.

\begin{definition}\label{def:spiky}
    If~$\Gamma$ satisfies both Conditions 1 \& 2, then say that~$\Gamma$ is \emph{spiky}.
\end{definition}

Let~$u$ be a non-principal vertex with a (not necessarily principal) vertex~$m >_\circ u$ such that there is a non-principal vertex~$u' \in \str{m}$ with~$d_\Gamma(u, u') = 2$. Then, for any~$m' > m$, we have~$u' \in \str{m}$. In other words, if a triple~$\{u, u', m\}$ violates Condition 1 (dropping the assumption that~$m$ be principal), then the triple~$\{u, u', m'\}$ also violates Condition 1 for all~$m' >_\circ m$. Similarly, in the statement of Condition 2, we may drop the requirement that~$m$ and~$n$ be principal. It is now clear that Conditions 1 and 2 combine to the following characterisation of spikiness:
\begin{center}
    for all non-principal vertices~$u$, if~$m >_\circ u$ and~$v \in \link{m} \setminus \str{u}$, then~$d_\Gamma(u, v) > 2$.
\end{center}

Fix~$\Gamma$ spiky. In particular, both our key lemmas (\textit{\ref{lem:condition 1 (NEW)}} \& \textit{\ref{lem:condition 2 (NEW)}}) apply.

For ease of reference, we restate some important terminology and notation. Recall (see \textit{Figure \ref{fig:cube complex structure of the spine}}) that each cube in~$\spine$ can be given a unique `address'~$c(\Pi_1, \Pi_2; \alpha)$, specified by its two extreme vertices and its marking (and the marking may occasionally be omitted if not important). In particular~$\Pi_1 \subseteq \Pi_2$, and~$\dim(c(\Pi_1, \Pi_2; \alpha)) = \vert \Pi_2 \setminus \Pi_1 \vert$. We write~$V = V(\Gamma)$, and let~$L \subseteq V$ be the set of principal vertices in~$\Gamma$. Finally, recall that for a subset~$W \subseteq V$, we write~$M(W)$ for the size for the largest set of pairwise compatible~$\Gamma$-partitions such that every partition in the set may be based at an element of~$W$. In particular,~$M(V)$ is the dimension of the spine~$\spine$.

\subsection*{Discussion of retraction strategy}

A slogan for our strategy is `remove all hugged partitions at each stage'. The process is phrased as a triple (finite) induction, where the inductive hypothesis is that every cube which was treated earlier in the procedure has been retracted in a particular way. We then use spikiness to complete the inductive step. 

More specifically, we will systematically find free faces of cubes in~$\spine$. If a cube~$C$ has a free face~$F$ (that is, every cube of which~$F$ is a subface is itself a subface of~$C$), then we can perform a deformation retraction by \emph{pushing along~$F$ into~$C$} (see \textit{Figure \ref{fig:pushin (1 face)}}). Since~$F$ is a free face of~$C$, this only alters cubes which are subfaces of~$C$, and in particular replaces~$C$ with a smaller-dimensional `shell' of itself. 

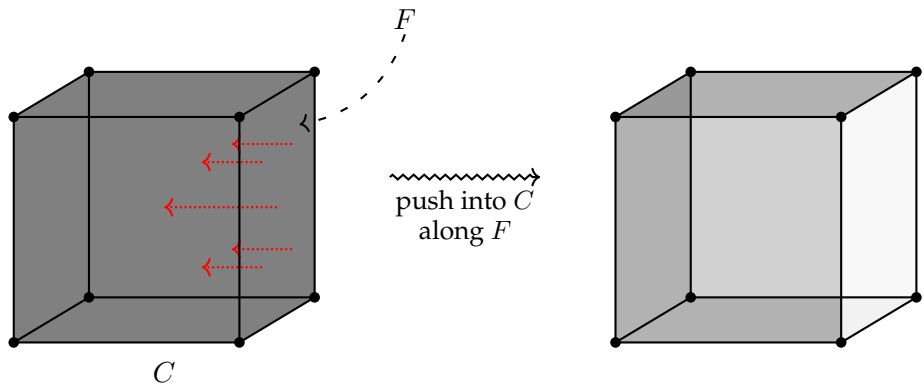
\begin{figure}[H]
    \begin{center}
        \begin{tikzpicture}
            \fill [black!50!white] (-5, 0) -- (-2, 0) -- (-1, 0.6) -- (-1, 3.6) -- (-4, 3.6) -- (-5, 3) -- (-5, 0);
            \fill [black!5!white] (6, 0) -- (6, 0.6) -- (7, 0.6) -- (6, 0);
            \fill [black!18!white] (4, 0.6) -- (6, 0.6) -- (6, 3) -- (4, 3) -- (4, 0.6);
            \fill [black!3!white] (6, 0.6) -- (6, 3) -- (7, 3.6) -- (7, 0.6) -- (6, 0.6);
            \fill [black!30!white] (3, 0) -- (6, 0) -- (6, 0.6) -- (4, 0.6) -- (4, 3) -- (3, 3) -- (3, 0);
            \fill [black!30!white] (4, 3) -- (6, 3) -- (7, 3.6) -- (4, 3.6) -- (4, 3);
            \fill [black!40!white] (3, 3) -- (4, 3) -- (4, 3.6) -- (3, 3);
            \fill (-5, 0) circle (2pt);
            \fill (-2, 0) circle (2pt);
            \fill (-1, 0.6) circle (2pt);
            \fill (-4, 0.6) circle (2pt);
            \fill (-5, 3) circle (2pt);
            \fill (-2, 3) circle (2pt);
            \fill (-1, 3.6) circle (2pt);
            \fill (-4, 3.6) circle (2pt);
            \fill (3, 0) circle (2pt);
            \fill (6, 0) circle (2pt);
            \fill (7, 0.6) circle (2pt);
            \fill (4, 0.6) circle (2pt);
            \fill (3, 3) circle (2pt);
            \fill (6, 3) circle (2pt);
            \fill (7, 3.6) circle (2pt);
            \fill (4, 3.6) circle (2pt);
            \draw [red, thick, densely dotted, ->] (-1.5, 1.8) -- (-3, 1.8);
            \draw [red, thick, densely dotted, ->] (-1.7, 2.4) -- (-2.5, 2.4);
            \draw [red, thick, densely dotted, ->] (-1.3, 2.64) -- (-2.1, 2.64);
            \draw [red, thick, densely dotted, ->] (-1.7, 1) -- (-2.5, 1);
            \draw [red, thick, densely dotted, ->] (-1.3, 1.24) -- (-2.1, 1.24);
            \draw[black, thick] (-5, 0) -- (-2, 0) -- (-1, 0.6) -- (-4, 0.6) -- (-5, 0);
            \draw[black, thick] (-5, 0) -- (-5, 3);
            \draw[black, thick] (-2, 0) -- (-2, 3);
            \draw[black, thick] (-4, 0.6) -- (-4, 3.6);
            \draw[black, thick] (-1, 0.6) -- (-1, 3.6);
            \draw[black, thick] (-5, 3) -- (-2, 3) -- (-1, 3.6) -- (-4, 3.6) -- (-5, 3);
            \draw[black, thick] (3, 0) -- (6, 0) -- (7, 0.6) -- (4, 0.6) -- (3, 0);
            \draw[black, thick] (3, 0) -- (3, 3);
            \draw[black, thick] (6, 0) -- (6, 3);
            \draw[black, thick] (4, 0.6) -- (4, 3.6);
            \draw[black, thick] (7, 0.6) -- (7, 3.6);
            \draw[black, thick] (3, 3) -- (6, 3) -- (7, 3.6) -- (4, 3.6) -- (3, 3);
            \node at (-3, -0.4) [black, thick]{$C$};
            \draw[black, thick, loosely dashed, ->] (0.2, 4.1) to[bend left] (-1.2, 2.9);
            \node at (0.2, 4.3) [black, thick]{$F$};
            \draw [black, thick, decorate, decoration={zigzag, amplitude = 0.9, segment length = 4, post = lineto, post length = 2pt}, ->] (0, 2.2) -- (2, 2.2);
            \node at (1, 1.85) [black, thick]{\small{push along~$F$}};
            \node at (1, 1.45) [black, thick]{\small{into~$C$}};
        \end{tikzpicture}
        \caption{Pushing along~$F$ into~$C$, turning~$C$ from a 3-dimensional cube to a 2-dimensional cube complex looking like an open box or shell. The (red) dotted arrows show how we visualise pushing in~$F$.}
        \label{fig:pushin (1 face)}
    \end{center}
\end{figure}

In fact, we will do slightly more radical deformation retractions than this. We will prove that the structure of~$\spine$ is such that if~$F$ and~$F'$ are both free faces of~$C$, then so is~$F \cap F'$ (observe that this is not true for arbitrary cube complexes: two cubes could `hinge' along a common codimension-2 subface, for instance). We will identify many free faces of~$C$, and prove that we can push along their (non-empty) common intersection into~$C$. By \emph{push along a subface into~$C$}, we mean a deformation retraction in which every cube not containing this subface is fixed, while the interior of every cube containing the subface is pushed onto this fixed subcomplex, as demonstrated in \textit{Figure \ref{fig:pushin (intersection of 2 faces)}}.

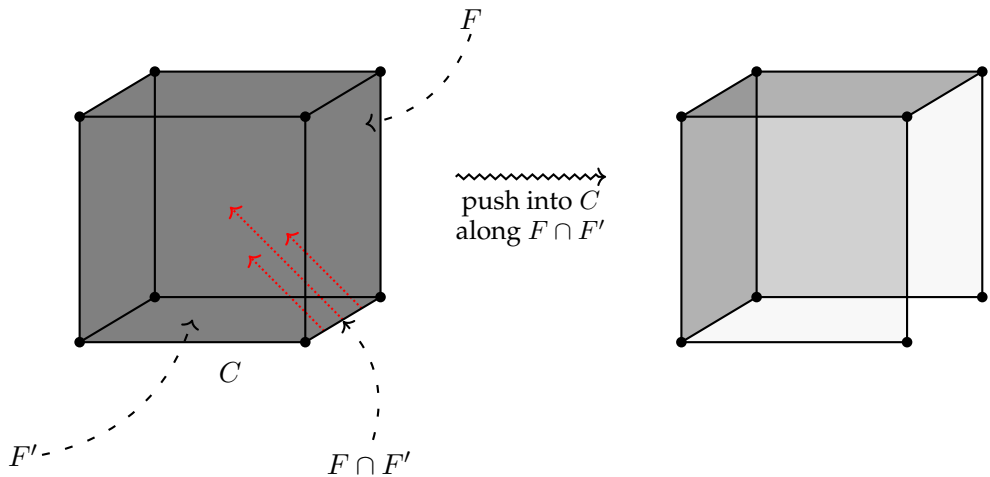
\begin{figure}
    \begin{center}
        \begin{tikzpicture}
            \fill [black!50!white] (-5, 0) -- (-2, 0) -- (-1, 0.6) -- (-1, 3.6) -- (-4, 3.6) -- (-5, 3) -- (-5, 0);
            \fill [black!3!white] (3, 0) -- (6, 0) -- (6, 0.6) -- (4, 0.6) -- (3, 0);
            \fill [black!3!white] (6, 0.6) -- (7, 0.6) -- (7, 3.6) -- (6, 3) -- (6, 0.6);
            \fill [black!18!white] (4, 0.6) -- (6, 0.6) -- (6, 3) -- (4, 3) -- (4, 0.6);
            \fill [black!30!white] (3, 0) -- (4, 0.6) -- (4, 3) -- (3, 3) -- (3, 0);
            \fill [black!30!white] (4, 3) -- (6, 3) -- (7, 3.6) -- (4, 3.6) -- (4, 3);
            \fill [black!40!white] (3, 3) -- (4, 3) -- (4, 3.6) -- (3, 3);
            \fill (-5, 0) circle (2pt);
            \fill (-2, 0) circle (2pt);
            \fill (-1, 0.6) circle (2pt);
            \fill (-4, 0.6) circle (2pt);
            \fill (-5, 3) circle (2pt);
            \fill (-2, 3) circle (2pt);
            \fill (-1, 3.6) circle (2pt);
            \fill (-4, 3.6) circle (2pt);
            \fill (3, 0) circle (2pt);
            \fill (6, 0) circle (2pt);
            \fill (7, 0.6) circle (2pt);
            \fill (4, 0.6) circle (2pt);
            \fill (3, 3) circle (2pt);
            \fill (6, 3) circle (2pt);
            \fill (7, 3.6) circle (2pt);
            \fill (4, 3.6) circle (2pt);
            \draw[red, thick, densely dotted, ->] (-1.5, 0.3) -- (-3, 1.8);
            \draw[red, thick, densely dotted, ->] (-1.75, 0.15) -- (-2.75, 1.15);
            \draw[red, thick, densely dotted, ->] (-1.25, 0.45) -- (-2.25, 1.45);
            \draw[black, thick] (-5, 0) -- (-2, 0) -- (-1, 0.6) -- (-4, 0.6) -- (-5, 0);
            \draw[black, thick] (-5, 0) -- (-5, 3);
            \draw[black, thick] (-2, 0) -- (-2, 3);
            \draw[black, thick] (-4, 0.6) -- (-4, 3.6);
            \draw[black, thick] (-1, 0.6) -- (-1, 3.6);
            \draw[black, thick] (-5, 3) -- (-2, 3) -- (-1, 3.6) -- (-4, 3.6) -- (-5, 3);
            \draw[black, thick] (3, 0) -- (6, 0);
            \draw[black, thick] (3, 0) -- (4, 0.6);
            \draw[black, thick] (4, 0.6) -- (7, 0.6);
            \draw[black, thick] (3, 0) -- (3, 3);
            \draw[black, thick] (6, 0) -- (6, 3);
            \draw[black, thick] (4, 0.6) -- (4, 3.6);
            \draw[black, thick] (7, 0.6) -- (7, 3.6);
            \draw[black, thick] (3, 3) -- (6, 3) -- (7, 3.6) -- (4, 3.6) -- (3, 3);
            \node at (-3, -0.4) [black, thick]{$C$};
            \draw[black, thick, loosely dashed, ->] (0.2, 4.1) to[bend left] (-1.2, 2.9);
            \node at (0.2, 4.3) [black, thick]{$F$};
            \draw[black, thick, loosely dashed, ->] (-5.5, -1.5) to[bend right] (-3.5, 0.3);
            \node at (-5.75, -1.5) [black, thick]{$F'$};
            \draw[black, thick, loosely dashed, ->] (-1.1, -1.3) to[bend right] (-1.5, 0.28);
            \node at (-1.15, -1.6) [black, thick]{$F \cap F'$};
            \draw [black, thick, decorate, decoration={zigzag, amplitude = 0.9, segment length = 4, post = lineto, post length = 2pt}, ->] (0, 2.2) -- (2, 2.2);
            \node at (1, 1.85) [black, thick]{\small{push into~$C$}};
            \node at (1, 1.45) [black, thick]{\small{along~$F \cap F'$}};
        \end{tikzpicture}
        \caption{Here we are pushing along the intersection~$F \cap F'$ into~$C$. The interiors of the faces~$F$,~$F'$,~$F \cap F'$, as well as the interior of~$C$, are all pushed onto the 2-dimensional subcomplex shown on the right.}
        \label{fig:pushin (intersection of 2 faces)}
    \end{center}
\end{figure}

We proceed by finding free faces of~$\spine$, retracting these as described (and doing so equivariantly). Then we look for more free faces in the resultant complex; when we can no longer find free faces, we stop. If we have managed to retract all cubes of dimension greater than some~$k$, then the complex we have obtained has dimension at most~$k$. 

We identify free faces by examining the structure of compatible sets of~$\Gamma$-partitions with respect to the hugged partitions they contain. Intuitively, if a set~$\Pi$ of partitions contains hugged partitions (or if non-principal partitions not in~$\Pi$ would be hugged in~$\Pi$ upon being added to~$\Pi$), then~$\Pi$ contains some redundant information. Partitions which are hugged in~$\Pi$ contribute roughly the same structure to~$\spine$ as the principal partitions which hug them -- and this turns out to result in~$\spine$ having many free faces. Our procedure stops when we have no more hugged partitions, so we have a characterisation of the final complex in terms of hugged partitions (which we state in \textit{Proposition \ref{prop:characterisation of resultant complex}}).

\begin{remark}
    As outlined at the start of \S\ref{sec:retraction process}, it is in this sense that the definition of `hugged' partitions is a notion of redundancy among sets of pairwise compatible~$\Gamma$-partitions. We make no claim that this is an optimal definition. There are certainly more general definitions which capture a similar notion of redundancy which we have decided not to include in the present document, as the added generality did not sufficiently outweigh the longer and less intuitive definitions (and their associated conditions). Moreover, there may well be finer notions of redundancy which allow one to identify more free faces. Then an inductive process very similar to the one outlined below may result in a more complete retraction of~$\spine$, or allow the removal or weakening of Conditions 1 \& 2.
\end{remark}

One of the challenges is to define a procedure which guarantees that after each stage of retractions, one still has free faces. We consider cubes~$c(\Pi_1, \Pi_2)$ in a particular order, looking for free faces in each. The induction will be with respect to the following variables:

\begin{enumerate}[(i)]
    \item~$b \;\text{($\;=b(\Pi_1)$)} \coloneqq \vert \Pi_1 \vert$;
    \item~$t \;\text{($\;=t(\Pi_2)$)} \coloneqq M(V) - \vert \Pi_2 \vert$;
    \item~$p \;\text{($\;=p(\Pi_2)$)} \coloneqq M(L) - \vert \{\text{principal partitions in~$\Pi_2$}\} \vert$.
\end{enumerate}

Here,~$b$ is counting how far away the bottom vertex of the cube~$c(\Pi_1, \Pi_2)$ is from a Salvetti; that is,~$b = \dim\left(c(\emptyset, \Pi_1)\right)$. Similarly,~$t$ is a measure of how far away the top vertex of~$c(\Pi_1, \Pi_2)$ is from being a blowup by a set of partitions of size~$M(V)$ (however, note that there does not necessarily exist a superset~$\Pi \supseteq \Pi_2$ with~$\vert \Pi \vert = M(V)$). Finally,~$p$ is a measure of how far away~$\Pi_2$ is from containing the maximal possible number of principal partitions: if~$p = 0$, then one cannot add any principal partitions to~$\Pi_2$; the maximal value~$p$ can take is~$M(L)$.

The outermost (i.e.~the slowest-moving) variable in this induction is~$b$, followed by~$t$ and then by~$p$. We will therefore start with~$b = 0$, that is, by considering cubes whose bottom vertex is a marked Salvetti. Working at a fixed value of~$b$, we begin with~$t = 0$, that is, with cubes whose top vertex corresponds to a blowup by a set of partitions of maximal possible size,~$M(V)$. Finally, for each fixed pair~$(b, t)$, we start with~$p = 0$ -- that is, among cubes~$c(\Pi_1, \Pi_2)$ with~$\vert \Pi_1 \vert = b$ and~$M(V) - \vert \Pi_2 \vert = t$, we begin by considering those~$\Pi_2$ containing the maximal possible number of principal partitions.

It transpires that the base case~$b = 0$ is the most complex, so for this reason the exposition has been split into first (\S\ref{subsubsec: retractions at salvettis}) a double induction on~$t$ and~$p$ where we hold~$b = 0$, and then (\S\ref{subsubsec: retractions away from salvettis}) an application of this work to a single induction on~$b$.

\subsection*{The retraction in detail}

We now set out in detail how we equivariantly retract the spine~$\spine$ to a subcomplex, for~$\Gamma$ spiky. In \S\ref{sec:realising vcd}, we will prove that the complex obtained at the end of this process is sometimes of smaller dimension than the spine, providing an upper bound on~$\vcduag$ which is tighter than~$\dim(\spine)$.

\subsubsection{\textbf{Case I:}~$b = 0$; retractions at marked Salvettis. \\}\label{subsubsec: retractions at salvettis}

The initial double induction on~$t$ and~$p$, with~$b$ fixed at 0, is concerned with cubes with a vertex at a marked Salvetti complex. 

\textbf{Case Ia:}~$b = 0$,~$t = 0$,~$p = 0$. 

That is, we are considering those cubes where~$\Pi_2$ is a maximally-sized set of pairwise compatible partitions, containing a full complement of~$M(L)$ principal partitions. Note that such a cube may not exist; if so we may ignore this case and begin the induction at the next step.

Fix any such~$\Pi_2$, and let~$H$ be its set of hugged non-principal partitions. We would like to be able to use the face~$c(\emptyset, \Pi_2 \setminus H$) to retract~$c(\emptyset, \Pi_2)$ (and all cubes in between), as illustrated in the following schematic.

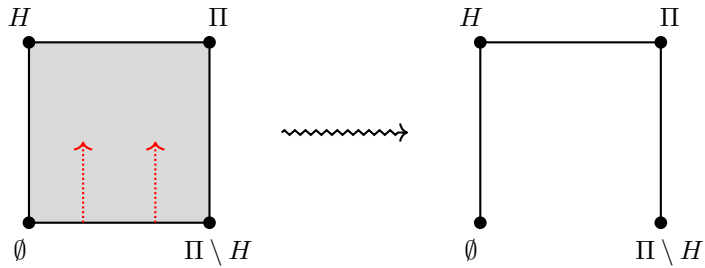
\begin{figure}[H]
    \begin{center}
        \begin{tikzpicture}[scale = 1.2]
            \fill[black!15!white] (0, 0) to (2, 0) to (2, 2) to (0, 2) to (0, 0);
            \fill (0, 0) circle (2pt);
            \fill (2, 0) circle (2pt);
            \fill (0, 2) circle (2pt);
            \fill (2, 2) circle (2pt);
            \fill (5, 0) circle (2pt);
            \fill (7, 0) circle (2pt);
            \fill (5, 2) circle (2pt);
            \fill (7, 2) circle (2pt);
            \draw[black, thick] (0, 0) to (2, 0) to (2, 2) to (0, 2) to (0, 0);
            \draw[black, thick] (5, 0) to (5, 2);
            \draw[black, thick] (5, 2) to (7, 2);
            \draw[black, thick] (7, 0) to (7, 2);
            \draw[red, thick, densely dotted, ->] (0.6, 0) to (0.6, 0.9);
            \draw[red, thick, densely dotted, ->] (1.4, 0) to (1.4, 0.9);
            \node at (-0.1, 0) [below = 2pt, black, thick]{\small{$\emptyset$}};
            \node at (2.1, 0) [below = 2pt, black, thick]{\small{$\Pi \setminus H$}};
            \node at (-0.1, 2) [above = 2pt, black, thick]{\small{$H$}};
            \node at (2.1, 2) [above = 2pt, black, thick]{\small{$\Pi$}};
            \node at (4.9, 0) [below = 2pt, black, thick]{\small{$\emptyset$}};
            \node at (7.1, 0) [below = 2pt, black, thick]{\small{$\Pi \setminus H$}};
            \node at (4.9, 2) [above = 2pt, black, thick]{\small{$H$}};
            \node at (7.1, 2) [above = 2pt, black, thick]{\small{$\Pi$}};
            \draw [black, thick, decorate, decoration={zigzag, amplitude = 0.9, segment length = 4, post = lineto, post length = 2pt}, ->] (2.8, 1) -- (4.2, 1);
        \end{tikzpicture}
        \caption{A schematic illustrating pushing along~$c(\emptyset, \Pi \setminus H)$ into~$c(\emptyset, \Pi)$.}
    \end{center}
    \label{fig:schematic of pushing along non-hugged partitions}
\end{figure}

We examine which partitions could be added to the set~$\Pi_2 \setminus H$. We cannot add any principal partitions to it, since~$\Pi_2 \setminus H$ already contains~$M(L)$ principal partitions. On the other hand, by \textit{Lemma \ref{lem:condition 1 (NEW)}}, any non-principal partition compatible with~$\Pi_2 \setminus H$ would necessarily be compatible with all of~$H$ too, so since~$\Pi_2$ was maximally-sized, any such is an element of~$H$. In other words, the only partitions we may add to~$\Pi_2 \setminus H$ are elements of~$H$, so the only cubes containing~$c(\emptyset, \Pi_2 \setminus H)$ are contained in~$c(\emptyset, \Pi_2)$. Since~$c(\emptyset, \Pi_2 \setminus H; \alpha)$ has a vertex at a Salvetti complex with marking~$\alpha$, any other cube containing it as a subface necessarily also has marking~$\alpha$. We can now perform a deformation retraction by pushing into the cube~$c(\emptyset, \Pi_2)$ along its free subface~$c(\emptyset, \Pi_2 \setminus H)$. Doing this simultaneously at all marked Salvetti complexes is a~$\uag$-equivariant retraction, since~$\uag$ acts only by changing markings. We do the same retractions for all cubes~$c(\emptyset, \Pi_2)$ with~$\Pi_2$ having parameters~$t = 0,\; p = 0$, which completes the~$p = 0$ base case.

\textbf{Case Ib:}~$b = 0$,~$t = 0$,~$p = k > 0$.

Our inductive hypothesis is that all cases~$t = 0,\; p < k$ have been handled in the above way: all hugged partitions were removed, and then the retraction was effected by pushing in along this subface. Now fix~$\Pi_2$ with~$t = 0$ and~$p = k$; let~$H$ be its set of hugged partitions and consider~$\Pi \setminus H$. Once again, we wish to prove that~$c(\emptyset, \Pi_2 \setminus H)$ is a free face of~$c(\emptyset, \Pi_2)$. 

If~$\calP$ is a principal partition which is compatible with all of~$\Pi_2 \setminus H$, then by \emph{Corollary \ref{cor:restatement of Key Lemma 2}}, and since~$\Pi_2$ is maximally-sized, there is exactly one~$\calQ \in H$ with which~$\calP$ is not compatible. Thus we may form the cube~$c(\emptyset, (\Pi_2 \setminus \calQ) \cup \calP)$. This falls under the case~$t = 0,\; p = k-1$, so by the inductive hypothesis, has already been retracted. Moreover, it was retracted via the subface obtained by removing all of its hugged partitions -- which in particular involves removing all of~$H \setminus \calQ$. Hence, the cube~$c(\emptyset, (\Pi_2 \setminus H) \cup \calP)$ has already been removed. 

Instead let~$\calQ'$ be a non-principal partition which is compatible with all of~$\Pi_2 \setminus H$. Then by \textit{Lemma \ref{lem:condition 1 (NEW)}} it is compatible with all of~$H$; since~$t = 0$ (that is,~$\Pi_2$ is maximally-sized), we must have~$\calQ' \in H$. Therefore, for each marking~$\alpha$, the only cubes containing~$c(\emptyset, \Pi_2 \setminus H; \alpha)$ are themselves contained in~$c(\emptyset, \Pi_2; \alpha)$. Hence, for each marking~$\alpha$, we may retract~$c(\emptyset, \Pi_2; \alpha)$ by pushing in along~$c(\emptyset, \Pi_2 \setminus H; \alpha)$; this is once again~$\uag$-equivariant. Do this simultaneously for all~$\Pi_2$ with~$t = 0,\; p = k$; doing this for each possible~$p = k > 0~$ completes the~$t = 0$ step.

\textbf{Case Ic:}~$b = 0$,~$t = l > 0$,~$p = k \geq 0$. 

Let~$\Pi_2$ be a set of pairwise compatible partitions with these values of~$t$ and~$p$, and let~$H$ be its set of hugged partitions. As before, we aim to prove that~$c(\emptyset, \Pi_2 \setminus H)$ is a free face of~$c(\emptyset, \Pi_2)$. 

Let~$\calP$ be a principal partition compatible with all of~$\Pi_2 \setminus H$. Suppose first that~$\calP$ is also compatible with all of~$H$. Then the set~$\Pi_2 \cup \calP$ of pairwise compatible partitions has parameters~$t = l-1,\; p = k$, so by the inductive hypothesis, has already been retracted. This retraction happened by removing all the hugged partitions in~$\Pi_2 \cup \calP$, which in particular includes all of~$H$. Hence the cube~$c(\emptyset, (\Pi_2 \setminus H) \cup \calP)$ is no longer present in our complex. 

Instead suppose that~$\calP$ is not compatible with all of~$H$, so by \textit{Lemma \ref{lem:condition 2 (NEW)}} there is exactly one~$\calQ \in H$ with which~$\calP$ is not compatible. We can therefore replace~$\calQ$ with~$\calP$, and obtain the compatible set~$(\Pi_2 \setminus \calQ) \cup \calP$. This has parameters~$t = l,\; p = k-1$, and so the cube~$c(\emptyset, (\Pi_2 \setminus H) \cup \calP)$ has also already been dealt with previously in the process. The set of hugged partitions in this case contains~$H \setminus \calQ$, so the subface used to retract this cube was contained in~$c(\emptyset, (\Pi_2 \setminus H) \cup \calP)$, so we have already removed~$c(\emptyset, (\Pi_2 \setminus H) \cup \calP)$.

Now let~$\calQ'$ be a non-principal partition compatible with all of~$\Pi_2 \setminus H$. By \textit{Lemma \ref{lem:condition 1 (NEW)}},~$\calQ'$ is necessarily compatible with all of~$H$. Suppose that~$\calQ' \notin H$. Then~$\Pi_2 \cup \calQ'$ is a pairwise compatible set of partitions, with parameters~$t = l-1$ and~$p = k$. By the inductive hypothesis, this was retracted using the subface corresponding to the removal of all the hugged partitions in~$\Pi_2 \cup \calQ'$; note that all of~$H$ is hugged in~$\Pi_2 \cup \calQ'$, and so the cube~$c(\emptyset, (\Pi_2 \setminus H) \cup \calQ')$ has already been removed. (Note that if~$\calQ'$ is also hugged in this set, then the cube~$c(\emptyset, \Pi_2 \cup \calQ')$ was retracted via the subface~$c(\emptyset, \Pi_2 \setminus H)$. This retraction will have taken with it the cube~$c(\emptyset, \Pi_2)$ -- so in this case there is nothing to do.)

We are left in the case~$\calQ' \in H$ -- now the cube~$c(\emptyset, (\Pi_2 \setminus H) \cup \calQ')$ is contained in~$c(\emptyset, \Pi_2)$ and will get retracted when we push in along~$c(\emptyset, \Pi_2 \setminus H)$.

Hence, at this stage of our retraction process, the only cubes which contain~$c(\emptyset, \Pi_2 \setminus H)$ are themselves contained in~$c(\emptyset, \Pi_2)$ -- so we may push into~$c(\emptyset, \Pi_2)$ along~$c(\emptyset, \Pi_2 \setminus H)$. Once again, we do this simultaneously for every marking; this is again~$\uag$-equivariant. We may do this simultaneously for all cubes~$c(\emptyset, \Pi_2)$ where~$\Pi_2$ has parameters~$t = l,\; p = k$, which completes the induction step.

This inductive process completes the~$b = 0$ step.

\subsubsection{\textbf{Case II:}~$b > 0$; retractions away from marked Salvetti complexes. \\}\label{subsubsec: retractions away from salvettis}

Now suppose that~$b = m > 0$. Our inductive hypothesis now has two parts:

\begin{enumerate}[(I)]
    \item For all markings~$\beta$, we have retracted all cubes of the form~$c(\Pi_1', \Pi_2'; \beta)$ where~$\vert \Pi_1' \vert < b$.
    \item For all markings~$\beta$, we have also retracted all cubes of the form~$c(\Pi_1', \Pi_2'; \beta)$ where we have~$\vert \Pi_1' \vert = b$ and where~$c(\emptyset, \Pi_2')$ would have been retracted before~$c(\emptyset, \Pi_2)$ in the~$b = 0$ step (that is,~$t(\Pi_2') < t(\Pi_2)$, or~$t(\Pi_2') = t(\Pi_2)$ and~$p(\Pi_2') < p(\Pi_2)$).

     Also assume that in each case, retraction has been done via the subface~$c(\Pi_1', \Pi_2' \setminus H'; \beta)$, where~$H'$ is the set of non-principal partitions in~$\Pi_2' \setminus \Pi_1'$ which are hugged in~$\Pi_2'$.
\end{enumerate}

Now consider some cube~$c(\Pi_1, \Pi_2; \alpha)$, with parameters~$b = m > 0,\; t = l \geq 0,\; p = k \geq 0$. Let~$H$ be the set of non-principal partitions in~$\Pi_2 \setminus \Pi_1$ which are hugged in~$\Pi_2$. Consider the subface~$c(\Pi_1, \Pi_2 \setminus H; \alpha)$; we wish to prove that this is a free face of~$c(\Pi_1, \Pi_2; \alpha)$. We do this by examining which cubes in our part-retracted complex could contain this subface. First, we restrict our attention to other cubes with marking~$\alpha$.

We could remove some partition~$\calR$ from~$\Pi_1$, yielding a subface~$c(\Pi_1 \setminus \calR, \Pi_2 \setminus H; \alpha)$. Now, if~$\calR$ was hugged in~$\Pi_2$, then the cube~$c(\Pi_1 \setminus \calR, \Pi_2 \setminus (H \cup \calR); \alpha)$ was used to retract~$c(\Pi_1 \setminus \calR, \Pi_2; \alpha)$ in the~$b = m-1$ case (using part (I) of the inductive hypothesis). In particular, this retraction will also have removed~$c(\Pi_1, \Pi_2; \alpha)$, so we may ignore this case. If instead~$\calR$ was not hugged in~$\Pi_2$, then~$c(\Pi_1 \setminus \calR, \Pi_2; \alpha)$ was retracted by pushing in along~$c(\Pi_1 \setminus \calR, \Pi_2 \setminus H; \alpha)$. In either case, we need not consider the cube~$c(\Pi_1 \setminus \calR, \Pi_2 \setminus H; \alpha)$.

Instead, we could add a partition~$\calR$ to~$\Pi_2$, to get the cube~$c(\Pi_1, (\Pi_2 \setminus H) \cup \calR; \alpha)$. We can ignore the case~$\calR \in H$, since pushing in to~$c(\Pi_1, \Pi_2; \alpha)$ along~$c(\Pi_1, \Pi_2 \setminus H; \alpha)$ would remove the face~$c(\Pi_1, (\Pi_2 \setminus H) \cup \calR; \alpha)$ anyway. So now by (II), the cube~$c(\Pi_1, \Pi_2 \cup \calR; \alpha)$ has already been retracted (since it corresponds to the case~$t = k-1$). This retraction used the subface obtained by removing all the non-principal partitions in~$(\Pi_2 \cup \calR) \setminus \Pi_1$ which are hugged in~$\Pi_2 \cup \calR$; in particular, this includes removing all of~$H$. Hence, the cube~$c(\Pi_1, (\Pi_2 \setminus H) \cup \calR; \alpha)$ has already been removed in a previous retraction. (Once again, if~$\calR$ is (non-principal and) hugged in~$\Pi_2$, then we may move on, as the cube~$c(\Pi_1, \Pi_2; \alpha)$ was also removed in this retraction.)

Therefore, the only cubes with marking~$\alpha$ which contain~$c(\Pi_1, \Pi_2 \setminus H; \alpha)$ are themselves contained in~$c(\Pi_1, \Pi_2; \alpha)$. We now turn our attention to cubes with other markings. We have the following general lemma, which we need some terminology to state: we say that a compatible set of~$\Gamma$-partitions~$\Pi$ is \emph{inextendible} if no~$\Gamma$-partitions can be added to~$\Pi$. That is, there are no~$\Gamma$-partitions not in~$\Pi$ which are compatible with all elements of~$\Pi$.

\begin{lemma}\label{lem:meet on (n-1)-face}
    Let~$\Gamma$ be a simplicial graph and~$\Pi$ be an inextendible set of~$\Gamma$-partitions. Let~$\alpha$,~$\beta$ be (distinct, non-equivalent) markings of the~$\Gamma$-complex~$\salb{\Pi}$. Suppose that the cubes~$\mathcal{C}_{\alpha} := c(\emptyset, \Pi; \alpha)$ and~$\mathcal{C}_{\beta} := c(\emptyset, \Pi; \beta)$ meet along some common face. Then this face is contained in a face common to both cubes which contains the vertex corresponding to the maximal blowup,~$\salb{\Pi}$.
\end{lemma}

\begin{proof}
Let~$\Phi$ be a subset of~$\Pi$ such that the blow-up along~$\Phi$ is a vertex contained in the common face of~$\mathcal{C}_{\alpha}$ and~$\mathcal{C}_{\beta}$. If~$\Phi = \Pi$, we are done, so assume not.

Let~$c : \salb{\Pi} \to \salb{\Phi}$ be the collapse along the hyperplane set~$\Pi \setminus \Phi$. Let~$\alpha_{\Phi} : \salb{\Phi} \to \sal$ be the collapse to the Salvetti in the cube~$\mathcal{C}_{\alpha}$ and similarly define~$\beta_{\Phi}$ (note that these are necessarily collapses to different marked Salvettis, since~$\alpha$ and~$\beta$ were non-equivalent). 

We have that~$(\salb{\Phi}, \alpha_{\Phi})$ and~$(\salb{\Phi}, \beta_{\Phi})$ are equivalent marked~$\Gamma$-complexes (recall \S\ref{def:untwisted spine}), so there exists an isomorphism~$h_{\Phi} : \salb{\Phi} \to \salb{\Phi}$ such that~$\alpha_{\Phi} \simeq \beta_{\Phi} \circ h_{\Phi}$.

Since~$(\salb{\Phi}, \alpha_{\Phi})$ is obtained from~$(\salb{\Pi}, \alpha)$ by a hyperplane collapse, we have~$\alpha_{\Phi} = \alpha \circ c^{-1}$. Similarly,~$\beta_{\Phi} = \beta \circ c^{-1}$. Hence,
\begin{eqnarray*}
    \alpha & = & \alpha_{\Phi} \circ c \\
    & \simeq & (\beta_{\Phi} \circ h_{\Phi}) \circ c \\
    & = & \beta \circ c^{-1} \circ h_{\Phi} \circ c.
\end{eqnarray*}

So, defining~$h:\salb{\Pi} \to \salb{\Pi}$ as~$h := c^{-1} \circ h_{\Phi} \circ c$, we see that in fact~$(\salb{\Pi}, \alpha) \sim (\salb{\Pi}, \beta)$, via this isomorphism~$h$. Thus the common face of~$\mathcal{C}_{\alpha}$ and~$\mathcal{C}_{\beta}$ must contain this marked~$\Gamma$-complex.
\end{proof}

If~$c(\Pi_1, \Pi_2 \setminus H; \alpha)$ has already been retracted, we need not consider it, so suppose it has not. Now, suppose that~$c(\Pi_1, \Pi_2 \setminus H; \alpha)$ is a subface of a cube~$C$ with some other marking~$\beta$. Let~$\Pi$ be the blowup corresponding to the top vertex of~$C$; then the cubes~$c(\emptyset, \Pi; \alpha)$ and~$c(\emptyset, \Pi; \beta)$ can be taken to be~$C_\alpha$ and~$C_\beta$ respectively in \textit{Lemma \ref{lem:meet on (n-1)-face}}. Then~$C$ is identified with a face which has marking~$\alpha$. Hence, the only cubes containing~$c(\Pi_1, \Pi_2 \setminus H; \alpha)$ are those contained in~$c(\Pi_1, \Pi_2; \alpha)$. Hence, we can push in to~$c(\Pi_1, \Pi_2; \alpha)$ along~$c(\Pi_1, \Pi_2 \setminus H; \alpha)$; we do this simultaneously for every marking. This is once again a~$\uag$-equivariant retraction. 

Doing this for every cube~$c(\Pi_1, \Pi_2)$ with~$\vert \Pi_1 \vert = b$ completes the induction step and the proof of the first statement of \emph{Theorem \ref{thm:retraction process (D,(i))}}.

\subsection{Resultant complex}\label{subsec:resultant complex}

It is natural to ask about the cube complex obtained upon completion of the retraction process detailed above. Can we characterise which cubes survive?

For~$\Gamma$ spiky, denote the complex resulting from this retraction process by~$\res$.

\begin{proposition}\label{prop:characterisation of resultant complex}
    Let~$\Gamma$ be spiky. A cube~$c(\Pi_1, \Pi_2; \alpha)$ in~$\spine$ is also found in~$\res$ if and only if~$\Pi_2 \setminus \Pi_1$ contains no non-principal partitions which are hugged in~$\Pi_2$, and additionally one cannot add a new non-principal partition to~$\Pi_2$ which would be hugged in~$\Pi_2$. 
\end{proposition}

\begin{proof}
    Suppose that~$\Pi_2 \setminus \Pi_1$ contains non-principal partitions which are hugged in~$\Pi_2$. Then the cube~$c(\Pi_1, \Pi_2; \alpha)$ is retracted during the process by removing these hugged partitions and pushing in along the corresponding subface. Now suppose that~$\Pi_2 \setminus \Pi_1$ does not contain any partitions hugged in~$\Pi_2$, but one could add some partition~$\calQ$ to~$\Pi_2$ such that~$\calQ$ is hugged in~$\Pi_2 \cup \calQ$. Then the cube~$c(\Pi_1, \Pi_2 \cup \calQ; \alpha)$ will have been retracted via~$c(\Pi_1, \Pi_2; \alpha)$. In both cases,~$c(\Pi_1, \Pi_2; \alpha)$ is not found in~$\res$.

    Conversely, the only way for a cube~$c(\Pi_1, \Pi_2)$ to be retracted is for either~$\Pi_2 \setminus \Pi_1$ to contain hugged partitions, or to be the face along which we push in. In the latter case, necessarily we can add a non-principal partition~$\calQ$ to~$\Pi_2$ which becomes hugged in~$\Pi_2 \cup \calQ$.
\end{proof}

This proposition motivates the following definition of a `redundant' cube in~$\spine$. 

\begin{definition}\label{def:redundant cube}
    A cube~$c(\Pi_1, \Pi_2)$ in~$\spine$ is \emph{redundant} if~$\Pi_2 \setminus \Pi_1$ contains no non-principal partitions which are hugged in~$\Pi_2$, and one cannot add a new non-principal partition to~$\Pi_2$ which would be hugged in~$\Pi_2$.
\end{definition}

\begin{remark}
    As remarked before, one could refine the notion of redundancy by refining the definition of `hugging'. 
\end{remark}

Thus the proposition can be stated as follows: a cube in~$\spine$ survives the retraction to~$\res$ if and only if it is not redundant. This is the second statement of \emph{Theorem \ref{thm:retraction process (D,(i))}}.

\begin{remark}
    Note that it is possible for single vertices in~$\spine$ to be redundant. For example, take~$\Pi_1 = \{\calP_1, \calP_2\}$ and~$\Pi_2 = \{\calP_1, \calP_2, \calQ\}$, where~$\calQ$ is hugged by~$\calP_1$ and~$\calP_2$. Then it may happen that at some stage, the (1-dimensional) cube~$c(\Pi_1, \Pi_2)$ (which is redundant since~$\Pi_2 \setminus \Pi_1 = \calQ$ is hugged in~$\Pi_2$) gets removed via its (0-dimensional) subface~$c(\Pi_1, \Pi_1)$ (which is redundant since~$\calQ$ can be added to~$\Pi_1$ to become hugged in~$\Pi_1 \cup \calQ = \Pi_2$). Hence the vertex sets of~$\spine$ and~$\res$ may not coincide.
\end{remark}

Observe (applying \emph{Theorem \ref{thm:vcd upper bound realisation theorem}}) that as the~$\uag$-action on~$\res$ is proper and cocompact, we know that the virtual cohomological dimension of~$\uag$ satisfies \[\dim(\res) \geq \vcduag \geq M(L).\] Hence:

\begin{corollary}\label{cor:there exist M(L) sets of principals and you can't add a non-principal to them}
    Suppose that~$\Gamma$ is spiky. Then there exists a compatible set~$\Pi$ with~$\vert \Pi \vert \geq M(L)$ which contains no hugged partitions, and to which one cannot add any non-principal partitions which would be hugged in~$\Pi$. 
\end{corollary}

In certain cases, the resultant complex~$\res$ is of smaller dimension than~$\spine$. This provides a tighter upper bound on~$\vcduag$, realised geometrically by a proper, cocompact action of~$\uag$.

\section{Realising virtual cohomological dimension}\label{sec:realising vcd}

Millard--Vogtmann \cite{MillardVogtmannCubeComplexes21} prove that if~$\Gamma$ is `barbed' and~$M(V) > M(L)$, then every top-dimensional cube in~$\spine$ has a free face; this allows one to perform an equivariant retraction of~$\spine$ which reduces its dimension by one. In this section, we strengthen this result as follows: if~$\Gamma$ is `barbed' then our resultant complex~$\res$ always has dimension exactly~$M(L)$.


We use the following definition from \cite{MillardVogtmannCubeComplexes21}.

\begin{definition}\label{def:barbed}
    We say that~$\Gamma$ is \emph{barbed} if all non-principal vertices~$u$ satisfy the following condition: for any~$v$ with~$d_\Gamma(u, v) = 2$, we have~$v >_\circ u$.
\end{definition}

Note that if~$\Gamma$ is barbed, then we cannot have non-principal vertices~$u, u'$ with~$d_\Gamma(u, u') = 2$: this would require both~$\link{u}^\pm \subsetneq \link{u'}^\pm$ and~$\link{u'}^\pm \subsetneq \link{u}^\pm$, which is impossible.

\begin{lemma}[\cite{MillardVogtmannCubeComplexes21}, \S7]\label{lem:consequences of barbed}
    Suppose~$\Gamma$ is barbed. Then:
    \begin{enumerate}[(i)]
        \item Every non-principal equivalence class is minimal and has only one element. Moreover, any~$\Gamma$-partition based at a non-principal vertex splits only that vertex.
        \item Suppose that~$\calQ$ is a non-principal partition based at~$u$. Then on each side of~$\calQ$, there is a principal vertex~$v$ with~$v \geq u$.
    \end{enumerate}
\end{lemma}

\begin{definition}\label{def:oversize}
    Let~$\Pi$ be a set of pairwise compatible~$\Gamma$-Whitehead partitions. We say that~$\Pi$ is \emph{oversize} if~$\vert \Pi \vert > M(L)$.
\end{definition}

Note that an oversize set necessarily contains a non-principal partition. 

\begin{lemma}\label{lem:barbed implies oversize contains hugged}
    Let~$\Gamma$ be barbed. Then any oversize set of compatible~$\Gamma$-Whitehead partitions contains a hugged partition.
\end{lemma}

\begin{proof}
    Let~$\Pi$ be an oversize set of compatible~$\Gamma$-partitions. Then~$\Pi$ contains at least one non-principal partition; let~$\calQ$ be a non-principal partition such that one of its sides (say,~$Q$) contains no sides of any other non-principal partitions. 

    \textbf{Step 1:} Either prove that~$\calQ$ is hugged or move to the next iteration, considering a different non-principal partition.
    
    \textbf{Aim:} prove that either~$\calQ$ is hugged in~$\Pi$, or if not, that we can replace~$\calQ$ with some principal partition~$\calP$, thus obtaining a pairwise compatible set~$\Pi' = \left( \Pi \setminus \calQ \right) \cup \calP$ of the same size but with one more principal partition.

    If we succeed in this \emph{Aim}, then we repeat the process. If at each stage the non-principal partition we consider is not hugged, then eventually we will end up with an oversize compatible set consisting only of principal partitions, which is a contradiction. This means that at some point, we have a compatible set~$\widehat{\Pi}$, obtained from~$\Pi$ by step-by-step replacements, which contains a non-principal partition which is hugged in~$\widehat{\Pi}$. We then prove that this non-principal partition must in fact have been hugged in~$\Pi$, which finishes the argument.

    By \textit{Lemma \ref{lem:consequences of barbed} (ii)},~$Q$ contains a principal vertex~$m$ with~$m \geq u$. Since~$d_\Gamma(m, u) > 1$ and~$m \geq u$, we have~$d_\Gamma(m, u) = 2$, and hence as~$\Gamma$ is barbed, in fact~$m >_\circ u$. 

    We start with some candidate principal partition~$\calP$ with which we try to replace~$\calQ$ in~$\Pi$. We will constantly edit this choice of~$\calP$ until we find a principal partition that is compatible with~$\Pi \setminus \calQ$ and so can replace~$\calQ$ (or else conclude that~$\calQ$ is already hugged in~$\Pi$). 
    
    \textbf{First choice of~$\calP$.} As a first pass, define ~$\calP$ by the side~$P \coloneqq Q \setminus \left(\{m^{-1}\} \cup (\link{m} \cap Q)\right)$; this is a partition based at~$m$. If this is in~$\Pi$, then it hugs~$\calQ$ in~$\Pi$, and we're done. So suppose~$\calP \notin \Pi$. If~$\calP$ is compatible with all of~$\Pi \setminus \calQ$, then we may replace~$\calQ$ with~$\calP$, obtaining a set~$\Pi'$ with one more principal partition, in which case we have achieved our aim. 

    So suppose that there is some~$\calR \in \Pi$ which is not compatible with~$\calP$.

    Suppose first that~$\calR$ is based at~$m$. Since~$\calR$ is compatible with~$\calQ$, and since~$Q$ contains both~$m$ and~$m^{-1}$, we know that~$\calQ$ must contain one side~$R$ of~$\calR$. Now, since~$\calR$ is incompatible with~$\calP$, we know that~$R \not\subseteq P$, so~$m^{-1} \in R$. 
    
    \textbf{Second choice of~$\calP$.} Now edit our choice of~$\calP$ to have the side~$P \coloneqq Q \setminus \left(\{m\} \cup (\link{m} \cap Q)\right)$. Note that this ensures that~$P \supseteq R$, so this choice of~$\calP$ is compatible with~$\calR$. Once again, we're done if~$\calP$ is compatible with all of~$\Pi \setminus \calQ$ (which includes the case~$\calP \in \Pi$), so assume this is not the case. Then there is some~$\calR' \in \Pi$ which is not compatible with~$\calP$.

    Suppose first that~$\calR'$ is based at~$m$. As before, since~$\calR'$ is compatible with~$\calQ$, and since~$Q$ contains both~$m$ and~$m^{-1}$, we know that~$\calQ$ must contain one side~$R'$ of~$\calR'$. If~$m^{-1} \in R'$, then necessarily~$R' \subseteq P$, which contradicts our assumption that~$\calR'$ is incompatible with~$\calP$. So~$m \in R'$. Hence we have:
    \begin{itemize}
        \item~$m \in \overline{R} \cap R'$; and
        \item~$m^{-1} \in R \cap \overline{R'}$; and
        \item~$u^{-1} \in \overline{R} \cap \overline{R'}$,
    \end{itemize}
    so since~$\calR$ and~$\calR'$ are compatible, we must have~$R \cap R' = \emptyset$.

    \textbf{Third choice of~$\calP$.} We may assume that~$R$ is maximal among sides of partitions based at~$m$ such that~$m^{-1} \in R \subseteq Q$. We now edit our choice of~$\calP$ once again so that~$\calQ$ is 2-hugged by~$\calR$ and~$\calP$, with the sides~$R$ and~$P$ respectively. Since~$R' \cap R = \emptyset$ and~$R \subseteq Q$, we must have~$R' \subseteq P$, so~$\calP$ is compatible with every partition which is based at~$m$. 

    Now suppose that there is some~$\calR'' \in \Pi$ with which~$\calP$ is incompatible. Let~$\calR''$ be based at~$n$; by what we just said, we must have~$n \neq m$. Since~$\calP$ and~$\calR''$ are incompatible, we must have~$d_\Gamma(m, n) = 2$. Since~$m >_\circ u$, we must also have~$d_\Gamma(n, u) \geq 2$ (as otherwise~$[n, u] = 1$ which implies that~$[m, n] = 1$). 
    
    Since~$\calR''$ is compatible with~$\calQ$, we know that~$Q$ either contains or is contained in a side of~$\calR''$. If the latter, then~$P \subseteq Q$ is also contained in a side of~$\calR''$, which contradicts incompatibility of~$\calR''$ with~$\calP$. Hence~$Q$ contains some side~$R''$ of~$\calR''$.

    We know that~$\calR''$ and~$\calR$ are compatible (but not adjacent, since~$d_\Gamma(n, m) = 2$), so~$R$ either contains or is contained in a side of~$\calR''$. If~$R$ contains a side of~$\calR''$, then that side of~$\calR''$ has empty intersection with~$P$, contradicting the incompatibility of~$\calR''$ and~$\calP$. So~$R$ is contained in a side of~$\calR''$. If~$R \subseteq \overline{R''}$, then~$R'' \subseteq \overline{R}$, which means that~$R'' \subseteq \overline{R} \cap Q$. Now, if~$R'' \cap \overline{P} \neq \emptyset$, this means that~$(\overline{P} \cap \overline{R}) \cap Q \neq \emptyset$. This is a contradiction, since~$\{\calP, \calR\}$ 2-hugs~$\calQ$, so~$\overline{P} \cap \overline{R} = \overline{Q}$. We are forced to conclude that~$R \subseteq R''$.

    Suppose that~$d_\Gamma(n, u) \geq 3$. Then, since~$d_\Gamma(n, m) = 2$, there is a single component~$C \in \calC(n)$ which contains all of~$\{u, u^{-1}, m, m^{-1}\}$. Since the base of~$R$ is~$m$, and since~$R \subseteq R''$, we must therefore have~$C \in R''$. This means that~$\{u, u^{-1}\} \subseteq R'' \subseteq Q$, which is a contradiction. 

    Hence, we have~$d_\Gamma(n, u) = 2$. Since~$\Gamma$ is barbed, we therefore conclude both that~$n$ is principal and that~$n >_\circ u$. Thus we have found a new principal vertex, distinct from~$m$, which dominates~$u$ and lies inside~$Q$. We can now go back to our \emph{First choice}, working with~$n$ instead of~$m$. We will either find that~$\calQ$ is hugged by partitions based at~$n$, or that we can find a partition based at~$n$ which can replace~$\calQ$ in~$\Pi$, or that the entire process will repeat, yielding a third principal vertex~$n' >_\circ u$ with~$n' \in Q$. Note that we cannot have~$n' = m$, as then we will find a partition~$\widetilde{\calR}$, based at~$m$, with one side~$\widetilde{R} \subseteq Q$ such that~$\widetilde{R} \supseteq R''$ (this last using the same argument that gave~$R \subseteq R''$ above). But then~$\widetilde{R} \supseteq R$, which contradicts our choice of~$R$ as maximal among sides of partitions based at~$m$ such that~$m^{-1} \in R \subseteq Q$. 

    This process must terminate at some point, since~$Q$ contains only finitely many vertices. If it terminates by finding that~$\calQ$ is hugged, then we're done, so suppose that is not the case. This means that it terminated because we found some principal partition~$\calP$ which could replace~$\calQ$, to form~$\Pi' \coloneqq \left( \Pi \setminus \calQ\right) \cup \calP$.

    We now go back to the start of \emph{Step 1} with this new~$\Pi'$: pick a non-principal partition with one side containing no sides of any other non-principal partitions, show that it is hugged in~$\Pi'$, or replace it too by a principal partition. This process must terminate, as otherwise we will successfully replace all the non-principal partitions in~$\Pi$ by principal partitions, resulting in an oversize set of pairwise compatible principal partitions -- a contradiction of the definition of `oversized'.

    Therefore, at some point, we have a set~$\widehat{\Pi}$, obtained from~$\Pi$ by step-by-step replacements as above, with a non-principal partition~$\widehat{\calQ} \in \widehat{\Pi}$ which is hugged in~$\widehat{\Pi}$. It only remains to prove that~$\widehat{\calQ}$ was in fact hugged in~$\Pi$. 

    \textbf{Step 2:} Conclude that we can find a partition which was hugged in~$\Pi$.

    Let~$\widehat{\calQ} \in \widehat{\Pi}$ be as above; denote the base of~$\widehat{\calQ}$ by~$\widehat{u}$.

    Note that if at any point we find that~$\widehat{\calQ}$ is hugged by principal partitions which were part of our original set~$\Pi$, then~$\widehat{\calQ}$ was hugged in~$\Pi$, and we're done. 
    
    \begin{claim*}
        ~$\widehat{\calQ}$ is 2-hugged in~$\widehat{\Pi}$. 
    \end{claim*}

    \begin{proof}[Proof of Claim.]
        Suppose that~$\widehat{\calQ}$ is 1-hugged in~$\widehat{\Pi}$ by some principal partition~$\calP$; let~$\widehat{Q}$ be the side of~$\widehat{\calQ}$ which is hugged. If~$\calP \in \Pi$, we're done; if not, then~$\calP$ was added in at some point, in replacement of some non-principal partition~$\calQ$ (which is based at, say,~$u$). Let~$\calP$ be based at some principal vertex~$m$; it had a side~$P \subseteq Q$. We have~$m >_\circ u$ and~$m >_\circ \widehat{u}$, so we cannot have~$d_\Gamma(u, \widehat{u}) = 1$. Since~$\calQ$ was chosen so that~$Q$ was innermost among non-principal sides, we cannot have~$\widehat{Q}$ hugged by the side~$P$ -- this would mean that~$\widehat{Q} \subseteq Q$. Hence~$\widehat{Q}$ must be hugged by the other side~$\overline{P}$ of~$\calP$. Therefore~$\widehat{Q} \supseteq \{m, m^{-1}\}$, which (since~$d_\Gamma(u, \widehat{u}) \neq 1$) means that~$\widehat{Q} \supseteq Q$. But now~$P \cup \overline{P} \subseteq \widehat{Q}$, which implies that~$\widehat{\calQ}$ is not a valid partition: one of its sides is a singleton. We conclude that~$\widehat{\calQ}$ is not 1-hugged in~$\widehat{\Pi}$. 
    \end{proof}
    
    We can now suppose that~$\widehat{\calQ}$ is 2-hugged in~$\widehat{\Pi}$ by~$\{\calP, \calP'\}$.
    
    If~$\{\calP, \calP'\} \subseteq \Pi$, we're done, so suppose not. Without loss of generality, suppose that~$\calP$ was added to our compatible set later in the process than~$\calP'$ was: in particular, the non-principal partition~$\calQ$ which it replaced was compatible with~$\calP'$. Let~$\calQ$ be based at~$u$, and let~$\calP$ be based at~$m$. This means that~$\calP'$ is also based at~$m$; as above, take~$\widehat{u}$ to be the base of~$\widehat{\calQ}$. Let~$Q$ be the side of~$\calQ$ which contains~$\{m, m^{-1}\}$, and let~$P$ be the side of~$\calP$ which is contained in~$Q$. Let~$\widehat{Q}$ be the side of~$\widehat{\calQ}$ which is 2-hugged; since~$\{m, m^{-1}\} \subseteq \widehat{Q}$, we have that~$P$, along with some side~$P'$ of~$\calP'$, are the sides that 2-hug~$\widehat{Q}$.

    Since~$d_\Gamma(u, u') \neq 1$, we know that~$Q$ either contains, or is contained in, one side of~$\widehat{\calQ}$. We chose~$Q$ to be innermost among non-principal sides of a set of partitions that contained~$\widehat{\calQ}$, so we must be in the latter case. Since~$m \in Q \cap \widehat{Q}$, we therefore conclude that~$Q \subseteq \widehat{Q}$.

    By compatibility of~$\calQ$ with~$\calP'$, and since~$Q$ contains~$\{m, m^{-1}\}$, we know that~$Q$ must contain a side of~$\calP'$. If~$P' \subseteq Q$, then~$\smash{\overline{\widehat{Q}} = \overline{P} \cap \overline{P'} \supseteq \overline{Q}}$, which implies that~$\widehat{Q} \subseteq Q$. Hence (again as~$Q$ was chosen to be innermost among non-principal sides of a set of partitions that contained~$\widehat{\calQ}$) we have~$Q = Q'$, a contradiction (as~$\widehat{\calQ} \in \widehat{\Pi}$, a set that does not contain~$\calQ$). 

    On the other hand, if~$\overline{P'} \subseteq Q$, then~$P' \cup \overline{P'} \subseteq \widehat{Q}$, which means that~$\widehat{\calQ}$ is not a valid partition (as one of its sides is a singleton). 

    In both cases, we derive a contradiction. We therefore conclude that~$\calP$ cannot have been added in as a replacement of some non-principal partition~$\calQ$. In other words,~$\calP \in \Pi$. Since we assumed that~$\calP$ was added later in the process than~$\calP'$ was, we conclude that~$\calP' \in \Pi$ too. Hence,~$\widehat{\calQ}$ was in fact 2-hugged in~$\Pi$, and we have found a hugged partition in~$\Pi$. This completes the proof.
\end{proof}

This proves \emph{Theorem \ref{customthm:main work}, (ii)}, which we restate in slightly different language here. Recall (\emph{Definition \ref{def:redundant cube}}) that a cube~$c(\Pi_1, \Pi_2)$ in~$\spine$ is \emph{redundant} if~$\Pi_2 \setminus \Pi_1$ contains no non-principal partitions which are hugged in~$\Pi_2$, and one cannot add a new non-principal partition to~$\Pi_2$ which would be hugged in~$\Pi_2$.

\begin{theorem}\label{thm:barbed implies cubes of dimension higher than principal rank are redundant}
    If~$\Gamma$ is barbed, then every cube in~$\spine$ of dimension greater than the principal rank~$M(L)$ is redundant.
\end{theorem}

We can now deduce \emph{Theorem \ref{customthm:main thm}}:

\begin{corollary}\label{cor:conditions 1 & 2 + barbed implies vcd = M(L)}
    Let~$\Gamma$ be spiky and barbed. Then~$\vcduag = M(L)$, and there exists a~$\uag$-complex realising this virtual cohomological dimension.
\end{corollary}

\begin{remark}
   ~$\Gamma$ being barbed precludes existence of non-principal vertices~$u, u'$ with~$d_\Gamma(u, u') = 2$, so if~$\Gamma$ is barbed then it automatically (vacuously) satisfies Condition 1. So the condition `spiky and barbed' can be replaced with `barbed and satisfies Condition 2'.
\end{remark}

\begin{proof}
    Since~$\Gamma$ is spiky, we may apply the retraction process to~$\spine$, obtaining the new~$\uag$-complex~$\res$. 
    
    A cube~$c(\Pi_1, \Pi_2; \alpha)$ in~$\spine$ will have dimension strictly greater than~$M(L)$ if and only if~$\Pi_2 \setminus \Pi_1$ is an oversize set. By \textit{Lemma \ref{lem:barbed implies oversize contains hugged}}, and since~$\Gamma$ is barbed, we know that~$\Pi_2 \setminus \Pi_1$ contains a non-principal partition which is hugged in~$\Pi_2 \setminus \Pi_1$ -- so in particular, is hugged in~$\Pi_2$. Now by the characterisation of~$\res$ in \textit{Proposition \ref{prop:characterisation of resultant complex}}, all such cubes were removed during the retraction process and so are not present in~$\res$. 

    We have therefore shown that~$\res$ contains no cubes of dimension greater than~$M(L)$; in other words,~$\dim(\res) \leq M(L)$. Since we have a proper cocompact action of~$\uag$ on~$\res$ (which is contractible, as it is a retract of~$\spine$), we have~$\vcduag \leq M(L)$. Combining this with lower bound~$M(L) \leq \vcduag$ given by \textit{Theorem \ref{thm:Millard-Vogtmann M(L) lower bound}} yields the claimed equality.
\end{proof}

\section{Application: Arbitrarily large gaps between~$\vcduag$ and~$\dim(\spine)$}\label{sec:arbitrarily large gaps}

In this section we present a certain prototypical family of examples, which we call \emph{rake} graphs. We explicitly calculate~$M(L)$ and~$M(V)$ for these graphs; we then apply \emph{Corollary \ref{cor:conditions 1 & 2 + barbed implies vcd = M(L)}} to show that~$\vcduag = M(L)$ and hence that the gap between~$\vcduag$ and~$\dim(\spine)$ may be arbitrarily large. 

\subsection{Rake graphs}\label{subsec:rakes}

\begin{definition}
    For~$d \geq 1$, the \emph{$d$-rake}~$T_d$ is the simplicial graph with one vertex~$v$ of valence~$d+1$, to which is attached a leaf~$u$ as well as~$d$ other degree-2 vertices,~$a_1, \dots, a_d$. Each~$a_i$ is adjacent to a leaf~$b_i$. 
\end{definition}

\begin{figure}[H]
    \begin{center}
        \begin{tikzpicture}
            \fill (-1, 0) circle (2pt);
            \fill (0, 0) circle (2pt);
            \fill (0.9, 1) circle (2pt);
            \fill (1.9, 1) circle (2pt);
            \fill (0.9, 0.6) circle (2pt);
            \fill (1.9, 0.6) circle (2pt);
            \fill (0.9, -1) circle (2pt);
            \fill (1.9, -1) circle (2pt);
            \draw[black, thick] (-1, 0) -- (0, 0);
            \draw[black, thick] (0, 0) -- (0.9, 1);
            \draw[black, thick] (0, 0) -- (0.9, -1);
            \draw[black, thick] (0.9, 1) -- (1.9, 1);
            \draw[black, thick] (0.9, -1) -- (1.9, -1);
            \draw[black, thick] (0, 0) -- (0.9, 0.6);
            \draw[black, thick] (0.9, 0.6) -- (1.9, 0.6);
            \node at (-1, 0) [above = 2pt, black, thick]{$u$};
            \node at (0, 0) [above = 2pt, black, thick]{$v$};
            \node at (0.9, 1) [above = 2pt, black, thick]{$a_1$};
            \node at (0.9, 0.6) [below = 2pt, black, thick]{$a_2$};
            \node at (0.9, -1) [below = 2pt, black, thick]{$a_d$};
            \node at (1.9, 1) [above = 2pt, black, thick]{$b_1$};
            \node at (1.9, 0.6) [below = 0pt, black, thick]{$b_2$};
            \node at (1.9, -1) [below = 0pt, black, thick]{$b_d$};
            \fill (1.4, -0.1) circle (0.5pt);
            \fill (1.4, -0.3) circle (0.5pt);
            \fill (1.4, -0.5) circle (0.5pt);
        \end{tikzpicture}
        \caption{The~$d$-rake~$T_d$.}
        \label{fig:the d-rake}
    \end{center}
\end{figure}
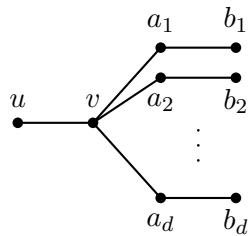

In our calculations of~$M(L)$ and~$M(V)$ below, we will rely on \textit{Lemma \ref{lem:MV19, 5.1}}, which we reproduce here for readability.

\begin{lemma}[\cite{MillardVogtmannCubeComplexes21}, \textit{Lemma 5.1}]\label{lem:M(u)+M(v)=M(u, v)}
    If non-equivalent vertices~$u, v \in V$ have~$d_\Gamma(u, v) \neq 2$, then any partition based at~$u$ is compatible with any partition based at~$v$. In particular, 
    \[M(u, v) = M(u) + M(v).\]
\end{lemma}

\begin{proposition}\label{prop:M(L), M(V) for rake graphs}
    For~$\Gamma = T_d$, we have~$M(L) = 3d-1$ and~$M(V) = 4d-2$.
\end{proposition}

\begin{proof}
    The principal vertices of~$T_d$ are~$L = \{v, a_1, \dots, a_d\}$, and we have~$u <_\circ a_i$ for each~$i$. Observe that the vertices are all pairwise non-equivalent; in particular, each~$T_d$-partition has exactly one possible base vertex.
    
    No~$b_i$ can be a base for a partition, as~$T_d^\pm \setminus \str{b_i}^\pm$ is a single connected component. Hence every vertex that can be a base of a partition commutes with~$v$. Thus by \textit{Lemma \ref{lem:M(u)+M(v)=M(u, v)}}, we have 
    \[M(L) = M(v) + M(a_1, \dots, a_d).\]

    Since~$T_d^\pm \setminus \str{v}^\pm$ consists only of the~$b_i^\pm$ and no edges, we have~$M(v) = 2d-1$. Indeed, the~$v$-sides of a compatible set of partitions all based at~$v$ are totally ordered by inclusion. The smallest possible~$v$-side consists of two vertices:~$v$ and one of~$\{b_1, \overline{b_1}, \dots, b_d, \overline{b_d}\}$; the largest possible consists of~$2d$ vertices:~$v$ along with all but one of~$\{b_1, \overline{b_1}, \dots, b_d, \overline{b_d}\}$. Hence~$M(v) \leq 2d-1$. Equality can be realised, for example by the following set of~$T_d$-partitions (given by their~$v$-sides):
    \begin{equation}\label{eqn:list of (2d-1) v-partitions}
        \{v, b_1\},\; \{v, b_1^\pm\},\; \{v, b_1^\pm, b_2\},\; \dots,\; \{v, b_1^\pm, \dots, b_{d-1}^\pm, b_d\}.
        \tag{$\star$}
    \end{equation}

    We now prove that~$M(a_1, \dots, a_d) = d$. Note that partitions based at~$a_i$ have~$\vert \calC(a_i) \vert = d+1$: we have the components~$\{u\}, \{\overline{u}\}$, and~$S_j \coloneqq \{a_j^\pm, b_j^\pm\}$ for~$j \neq i$. Hence~$M(a_i) \leq d$ for each~$i$. Let~$\calP_i$ be based at~$a_i$ and~$\calP_j$ based at~$a_j$,~$i \neq j$. Then, since~$[a_i, a_j] \neq 1$, we must have a choice of sides~$P_i^\times, P_j^\times$ such that~$P_i^\times \cap P_j^\times = \emptyset$.~$\calP_i$ splits~$a_i$, so~$P_j^\times$ must not contain the connected component~$S_i \in \calC(a_j)$. Hence~$\overline{P_j^\times}$ contains this connected component. Similarly,~$P_i^\times$ does not contain the connected component~$S_j \in \calC(a_i)$. Now, every element of~$\calC(a_i)$ which is contained in~$P_i^\times$ is contained in~$\overline{P_j^\times}$, as otherwise we could not have~$P_i^\times \cap P_j^\times = \emptyset$. Hence~$P_i^\times \subsetneq \overline{P_j^\times}$. Moreover,~$\overline{P_j^\times}$ contains at least one more element of~$\calC(a_j)$ than the number of elements of~$\calC(a_i)$ contained in~$P_i^\times$. Hence, any set of pairwise compatible partitions all based at elements of~$\{a_1, \dots, a_d\}$ form a chain of sides totally ordered by inclusion, and this chain can be of length at most~$d$, so~$M(a_1, \dots, a_d) \leq d$. One can reapply this argument to show that~$M(W) \leq d$, for any subset~$W \subseteq \{a_1, \dots, a_d\}$. Equality can also be realised here. For example, the following is a set of size~$d$ of partitions, one based at~$a_i$ for each~$i = 1, \dots, d$:
    \begin{equation}\label{eqn:list of d a_i-partitions}
        \{a_1, u\},\;\; \{a_2, u, S_1\},\;\; \dots\;\;, \{a_d, u, S_1, \dots, S_{d-1}\}.
        \tag{$\star \star$}
    \end{equation}
     
    Since~$M(L) = M(v) + M(a_1, \dots, a_d)$, we now have~$M(L) = 2d-1 + d = 3d-1$, as required.

    The only non-principal vertex at which partitions can be based is~$u$. Any set of pairwise compatible partitions, all based at~$u$, must have their~$u$-sides totally ordered by inclusion. The only elements of~$\calC(u)$ are~$S_1, \dots, S_d$; since there are~$d$ of these we therefore have~$M(u) = d-1$, using the same reasoning as for~$v$ above.

    Hence~$M(V) \leq M(L) + M(u) = (3d-1) + (d-1) = 4d-2$. The following set of pairwise compatible partitions, all based at~$u$:
    \[\{u, S_1\},\;\; \dots,\;\; \{u, S_1, \dots, S_{d-1}\},\]
    together with the sets given in \eqref{eqn:list of (2d-1) v-partitions} and \eqref{eqn:list of d a_i-partitions} gives a set of pairwise compatible~$T_d$-partitions which realises the equality~$M(V) = 4d-2$.
\end{proof}

\subsection{Arbitrarily large gaps}\label{subsec:arbitrarily large gaps}

Applying the work of Millard--Vogtmann \cite{MillardVogtmannCubeComplexes21} to the calculations from the previous subsection, we have the bounds
\[3d-1 = M(L) \leq \textsc{vcd}(U(A_{T_d})) \leq M(V) = \dim(K_{T_d}) = 4d-2.\]
In fact, the rake graphs are all barbed, so (\cite{MillardVogtmannCubeComplexes21}, \textit{Theorem 7.8}) yields the tighter upper bound of~$4d-3$. In particular, we have~$\textsc{vcd}(U(A_{T_2})) = 5$. 

For~$\Gamma$ a tree, Bux--Charney--Vogtmann \cite{BuxCharneyVogtmann2dimRAAGs09} prove that 
\[\textsc{vcd}(\oag) = e + 2\ell - 3,\]
where~$e$ is the number of edges and~$\ell$ is the number of leaves. As noted in (\cite{CharneyStambaughVogtmannUntwistedOuterSpace17}, \S5.2), since the normal subgroup of twists~$T(\raag) \leq \oag$ is free abelian of rank~$\ell$ and since the intersection~$T(\raag) \cap \uag = \{1\}$, it is natural to hope that~$\vcduag = e + \ell - 3$. For the~$d$-rake, this number is~$(2d+1) + (d+1) - 3 = 3d-1 = M(L)$. 

It is easy to see that in addition to being barbed, the rake graphs are also spiky. We can now apply \emph{Corollary \ref{cor:conditions 1 & 2 + barbed implies vcd = M(L)}} to immediately obtain:

\begin{corollary}\label{cor:arbitrarily large gaps}
   ~$\textsc{vcd}(U(A_{T_d})) = 3d-1$, and there exists a~$U(A_{T_d})$-complex of this dimension. In particular, there exist graphs~$\Gamma$ for which:
    \begin{enumerate}
        \item the difference between~$\dim(\res)$ and~$\dim(\spine)$ is arbitrarily large;
        \item the difference between~$\vcduag$ and~$\dim(\spine)$ is arbitrarily large.
    \end{enumerate}
\end{corollary}

\begin{remark}\label{rem:rake derivatives}
    There is a more general class of \emph{rake-like} graphs to which our results still apply. Let~$\Gamma'$ be any graph with no non-principal vertices. For example,~$\Gamma'$ could be a complete graph, or a cycle, or any graph where vertices are never distance 2 apart. One can replace~$v$ in~$T_d$ with~$\Gamma'$, adding edges so that every vertex of~$\Gamma'$ is adjacent to~$u$ and each of the~$a_i$. This makes a derivative \emph{rake-like} graph~$T_d(\Gamma')$. The set of non-principal vertices of~$T_d(\Gamma')$ is~$\{u, b_1, \dots, b_d\}$, and once again~$u$ is the only relevant non-principal vertex. One may check that~$T_d(\Gamma')$ is barbed and spiky. Hence~$\textsc{vcd}(U(A_{T_d(\Gamma')}))$ is equal to the principal rank of~$T_d(\Gamma')$. As with rake graphs, all rake-like graphs exhibit a growing gap between~$\textsc{vcd}(U(A_{T_d(\Gamma')}))$ and~$\dim(K_{T_d(\Gamma')})$ as~$d$ increases.
\end{remark}

\section{Application: On a sufficient condition for~$M(L) = M(V)$}\label{sec:sufficient condition for M(L) = M(V)}

In this section, we consider the following graph~$\Delta$. 

\begin{figure}[H]
    \begin{center}
        \begin{tikzpicture}
            \fill (-1, 0) circle (2pt); 
            \fill (-1, -1) circle (2pt); 
            \fill (0, 0) circle (2pt); 
            \fill (0, -1) circle (2pt); 
            \fill (0.9, 0.5) circle (2pt); 
            \fill (0.9, -0.5) circle (2pt); 
            \fill (0.9, -1.5) circle (2pt); 
            \fill (1.9, 0.5) circle (2pt); 
            \fill (1.9, -0.5) circle (2pt); 
            \fill (1.9, -1.5) circle (2pt); 
            \draw[black, thick] (-1, 0) -- (0, 0); 
            \draw[black, thick] (-1, -1) -- (0, -1); 
            \draw[black, thick] (0, 0) -- (0.9, 0.5); 
            \draw[black, thick] (0, 0) -- (0.9, -0.5); 
            \draw[black, thick] (0, -1) -- (0.9, -0.5); 
            \draw[black, thick] (0, -1) -- (0.9, -1.5); 
            \draw[black, thick] (0.9, 0.5) -- (1.9, 0.5); 
            \draw[black, thick] (0.9, -0.5) -- (1.9, -0.5); 
            \draw[black, thick] (0.9, -1.5) -- (1.9, -1.5); 
            \node at (-1, 0) [above = 2pt, black, thick]{$u_1$};
            \node at (0, 0) [above = 2pt, black, thick]{$v_1$};
            \node at (-1, -1) [above = 2pt, black, thick]{$u_2$};
            \node at (0, -1) [above = 2pt, black, thick]{$v_2$};
            \node at (0.9, 0.5) [above = 2pt, black, thick]{$a_1$};
            \node at (0.9, -0.5) [below = 2pt, black, thick]{$a_2$};
            \node at (0.9, -1.5) [below = 2pt, black, thick]{$a_3$};
            \node at (1.9, 0.5) [above = 2pt, black, thick]{$b_1$};
            \node at (1.9, -0.5) [below = 0pt, black, thick]{$b_2$};
            \node at (1.9, -1.5) [below = 0pt, black, thick]{$b_3$};
        \end{tikzpicture}
    \end{center}
    \caption{The graph~$\Delta$.}
    \label{fig:delta}
\end{figure}
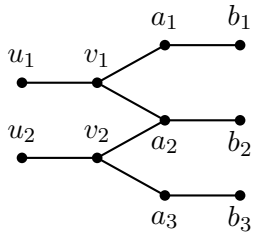

One can check that~$\Delta$ is barbed and satisfies Condition 1, but violates Condition 2, so is not spiky. In particular, one of our key lemmas, \emph{Lemma \ref{lem:condition 2 (NEW)}}, cannot be applied directly to~$\Delta$. However, we will prove using ad-hoc methods that we can still apply the conclusion of \emph{Lemma \ref{lem:condition 2 (NEW)}} to~$\Delta$, and hence the conclusion of our main result, \emph{Corollary \ref{cor:conditions 1 & 2 + barbed implies vcd = M(L)}}. Therefore,~$\Delta$ signifies the potential for further work, as it illustrates the failure of the definition of `hugging' to truly capture the notion of `redundancy' that we use in our retraction. In this paper, we will content ourselves with applying our main theorem to~$\Delta$ in order to investigate a potential sufficient condition for guaranteeing that~$M(L) = M(V)$, as we do in the next subsection.

\subsection{A sufficient condition for~$M(L) = M(V)$}\label{subsec:a suff cond for M(L) = M(V)}

Given a graph~$\Gamma$, Millard--Vogtmann provide a sufficient condition on~$\Gamma$ which ensures~$M(L) = M(V)$:
\begin{theorem}[\cite{MillardVogtmannCubeComplexes21}, \textit{Theorem 5.2}]\label{thm:sufficient condition for M(L) = M(V)}
    Suppose that each non-principal vertex~$u \in V(\Gamma)$ satisfies the following condition:

    \begin{equation}\label{eqn:sufficient condition for M(L) = M(V)}
        \text{all principal maximal~$m >_\circ u$ are in the same component of~$\Gamma \setminus \str{u}$}.
        \tag{$\dagger$}
    \end{equation}

    Then~$M(L) = M(V)$, so in particular~$\vcduag = \dim(\spine)$.
\end{theorem}

One reason for considering rake(-like) graphs is that they are the prototypical configurations which violate \eqref{eqn:sufficient condition for M(L) = M(V)}. Indeed,~$T_2$ is the graph with the fewest vertices which violates \eqref{eqn:sufficient condition for M(L) = M(V)}. 

The~$d$-rake violates~$\eqref{eqn:sufficient condition for M(L) = M(V)}$ in a rather controlled manner, and the gap between~$\textsc{vcd}(U(A_{T_d}))$ and~$\dim(K_{T_d})$ is predictable: there are precisely~$d$ components of~$T_d \setminus \str{u}$ containing a principal maximal~$m >_\circ u$, and the gap between~$\textsc{vcd}(U(A_{T_d})) = M(L)$ and~$\dim(K_{T_d}) = M(V)$ is precisely~$d-1$. One might therefore consider a statement along the lines of:
\begin{center}
    ``if, for every non-principal~$u$, there are at most~$k$ components \\ of~$\Gamma \setminus \str{u}$ containing a principal maximal~$m >_\circ u$, \\ then~$\dim(\spine) - \vcduag = k-1$''.
\end{center}

The graph~$\Delta$ provides a simple counterexample to this speculation.

The principal vertices of~$\Delta$ are~$L = \{a_1, a_2, a_3, v_1, v_2\}$. The relevant non-principal vertices are~$\{u_1, u_2, b_2\}$. We have~$u_1 <_\circ a_1, a_2$ and~$u_2 <_\circ a_2, a_3$ while~$b_2 <_\circ v_1, v_2$ and so~$\Delta$ satisfies the speculative condition above with~$k = 2$.

\begin{lemma}\label{lem:M(L) and M(V) for Delta}
    For~$\Delta$, we have~$M(L) = 11$ and~$M(V) = 14$.
\end{lemma}

\begin{proof}
Note that the vertices are pairwise non-equivalent, so each~$\Delta$-partition has only one possible base vertex.

The elements of~$\calC(v_1)$ are~$\{b_1\}, \{\overline{b_1}\}, \{b_2\}, \{\overline{b_2}\}, C \coloneqq \{u_2^\pm, v_2^\pm, a_3^\pm, b_3^\pm\}$. Since there are five components here, we have~$M(v_1) \leq 4$ (and we can realise equality here). The elements of~$\calC(v_2)$ are similar, and~$M(v_2) = 4$. Let~$\calP_1$,~$\calP_2$ be compatible partitions based at~$v_1$,~$v_2$ respectively. Then we have a choice of sides~$P_1^\times$,~$P_2^\times$ such that~$P_1^\times \cap P_2^\times = \emptyset$. Observe that~$P_1^\times$ cannot contain~$C$, since~$C \supset \{v_2^\pm\}$. Similarly,~$P_2^\times$ must not contain the element~$\{u_1^\pm, v_1^\pm, a_1^\pm, b_1^\pm\} \in \calC(v_2)$. If we have~$b_2 \in P_1^\times$, then~$P_2^\times$ must not contain~$b_2$; the same holds for~$\overline{b_2}$. Now, since in any compatible set, the partitions all based at~$v_1$ must have their~$v_1$-sides totally ordered by inclusion (and the same for~$v_2$), we have~$M(v_1, v_2) \leq 6$. Indeed, if we have four partitions based at~$v_1$, then we may only add in a further two partitions based at~$v_2$, both with one side only having components from~$\{\{b_3\}, \{\overline{b_3}\}, \{b_3^\pm\}\}$ (for example,~$\{v_2, b_3\}$ and~$\{v_2, b_3^\pm\}$). The same holds if we are to have four partitions based at~$v_2$. If we have three partitions based at~$v_1$, then we can add in at most three partitions based at~$v_2$: at least one of those based at~$v_1$ necessarily contains at least one of~$\{b_2, \overline{b_2}\}$ on the same side as at least one of~$\{b_3, \overline{b_3}\}$. Hence~$M(v_1, v_2) = 6$.

The elements of~$\calC(a_1)$ are~$\{u_1\}$,~$\{\overline{u_1}\}$, and~$C_{2, 3} \coloneqq \{a_2^\pm, b_2^\pm, v_2^\pm, u_2^\pm, a_3^\pm, b_3^\pm\}$. The elements of~$\calC(a_3)$ are similar, while those of~$\calC(a_2)$ are~$\{u_1\}$,~$\{\overline{u_2}\}$,~$\{u_2\}$,~$\{\overline{u_2}\}$,~$\{a_1^\pm, b_1^\pm\}$, and~$\{a_3^\pm, b_3^\pm\}$. Hence~$M(a_2) = 5$, while~$M(a_1) = M(a_3) = 2$. By \textit{Lemma \ref{lem:M(u)+M(v)=M(u, v)}}, since~$d_\Delta(a_1, a_3) \neq 2$, we can calculate~$M(a_1, a_3) = M(a_1) + M(a_3) = 2+2 = 4$. 

Consider any compatible set~$\Pi$ of partitions based in the set~$\{a_1, a_3\}$. We consider how many partitions based at~$a_2$ it is possible to add to~$\Pi$ to obtain a set of size~$M(a_1, a_2, a_3)$. If~$\Pi$ is empty, we can add at most five partitions based at~$a_2$, since~$M(a_2) = 5$. Assume~$\Pi \neq \emptyset$; without loss of generality let~$\calP_1 \in \Pi$ be based at~$a_1$. We know that this has one side containing one or both of the components~$\{u_1\}$,~$\{\overline{u_1}\}$. But these are both elements of~$\calC(a_2)$, so if there is a partition based at~$a_1$ with both of these components on the same side, then we can only add at most three partitions based at~$a_2$. Similarly, if there is no partition based at~$a_1$ with both of these components on the same side, then we may add at most four partitions based at~$a_2$. Now a symmetric argument with any partitions based at~$a_3$ further reduces the options for partitions based at~$a_2$ which can be added to~$\Pi$ to obtain a larger compatible set. If there is a partition based at~$a_3$ with both of the components~$\{u_3\}$,~$\{\overline{u_3}\}$ on the same side, then the number of partitions based at~$a_2$ we can add is reduced by another two. If there is only one partition based at~$a_3$ and the components~$\{u_3\}$ and~$\{\overline{u_3}\}$ are on different sides, then the number of partitions based at~$a_2$ which we can add is further reduced by only one. Notice that if~$\Pi$ contains two partitions based at~$a_1$ then one of them necessarily has~$\{u_1\}$ and~$\{\overline{u_1}\}$ on the same side; the same holds for~$a_3$. Hence~$M(a_1, a_2, a_3) \leq 5$ (and since~$M(a_2) = 5$, in fact we have equality).

Hence~$M(L) \leq M(v_1, v_2) + M(a_1, a_2, a_3) = 6 + 5 = 11$. We can attain this upper bound, as witnessed by the following compatible set of~$\Delta$-partitions, defined by the sides:
\begin{itemize}
    \item based at~$a_1$:~$\{a_1, u_1\}$;
    \item based at~$a_2$:~$\{a_2, u_1, a_1^\pm, b_1^\pm\}$,~$\{a_2, u_1^\pm, a_1^\pm, b_1^\pm\}$,~$\{a_2, u_1^\pm, a_1^\pm, b_1^\pm, u_2, a_3^\pm, b_3^\pm\}$;
    \item based at~$a_3$:~$\{a_3, u_2\}$;
    \item based at~$v_1$:~$\{v_1, b_1\}$,~$\{v_1, b_1^\pm\}$,~$\{v_1, b_1^\pm, b_2\}$;
    \item based at~$v_2$:~$\{v_2, b_3\}$,~$\{v_2, b_3^\pm\}$,~$\{v_2, b_3^\pm, \overline{b_2}\}$.
\end{itemize}

Now, the elements of~$\calC(u_1)$ are~$\{a_1^\pm, b_1^\pm\}$ and~$\{u_2^\pm, v_2^\pm, a_2^\pm, b_2^\pm, a_3^\pm, b_3^\pm\}$. Hence~$M(u_1) = 1$; we have~$M(u_2) = 1$ similarly. By \textit{Lemma \ref{lem:M(u)+M(v)=M(u, v)}},~$M(u_1, u_2) = M(u_1) + M(u_2) = 1 + 1 = 2$, since~$d_\Delta(u_1, u_2) \neq 2$. The components of~$\Delta^\pm \setminus \str{b_2}^\pm$ are~$\{u_1^\pm, v_1^\pm, a_1^\pm, b_1^\pm\}$ and~$\{u_2^\pm, v_2^\pm, a_3^\pm, b_3^\pm\}$; hence~$M(b_2) = 1$. Since~$d_\Delta(u_1, b_2) = d_\Delta(u_2, b_2) > 2$, we know that any partition based at~$b_2$ is compatible with any set of partitions based in~$\{u_1, u_2\}$. Hence~$M(u_1, u_2, b_2) = 3$, by \emph{Lemma \ref{lem:M(u)+M(v)=M(u, v)}}.

Therefore, we have~$M(V) \leq M(L) + M(u_1, u_2, b_2) = 11 + 3 = 14$. We have equality since we can add the following three~$\Delta$-partitions to the above list realising~$M(L)$:
\begin{itemize}
    \item based at~$u_1$:~$\{u_1, a_1^\pm, b_1^\pm\}$;
    \item based at~$u_2$:~$\{u_2, a_3^\pm, b_3^\pm\}$;
    \item based at~$b_2$:~$\{b_2, u_1^\pm, v_1^\pm, a_1^\pm, b_1^\pm\}$.
\end{itemize}
This completes the proof.
\end{proof}

\begin{corollary}\label{cor:counterexample to suff condition hoped-theorem}
    We have~$\dim(K_\Delta) - \textsc{vcd}(U(A_\Delta)) = 3$.
\end{corollary}

\begin{proof}
    We would like to apply our work to show that~$\textsc{vcd}(U(A_\Delta)) = M(L)$. However,~$\Delta$ violates Condition 2, so we cannot use \emph{Lemma \ref{lem:condition 2 (NEW)}}. It turns out that~$\Delta$ actually still satisfies the conclusion of \emph{Lemma \ref{lem:condition 2 (NEW)}}; we now prove this by hand. Since our work only relies on graphs satisfying Condition 2 so that \emph{Lemma \ref{lem:condition 2 (NEW)}} is applicable, this is sufficient to apply \emph{Corollary \ref{cor:conditions 1 & 2 + barbed implies vcd = M(L)}}.

    Let~$\Pi$ be a set of compatible~$\Delta$-partitions. Suppose that~$\Pi$ contains non-principal partitions~$\calQ$ and~$\calQ'$ which are hugged in~$\Pi$. We would like to show that if~$\calR$ is some principal partition which is compatible with all the principal partitions in~$\Pi$ (so in particular, with all partitions involved in hugging~$\calQ$ and~$\calQ'$), then~$\calR$ is compatible with at least one of~$\{\calQ, \calQ'\}$. 

    Note that there is no principal vertex in~$\Delta$ which is distance two from both~$u_1$ and~$b_2$. Hence, by \emph{Lemma \ref{lem:M(u)+M(v)=M(u, v)}}, there is no principal~$\Delta$-partition which is incompatible with a partition based at~$u_1$ and a partition based at~$b_2$ simultaneously. The same argument holds for the pair~$u_2$ and~$b_2$. Therefore (since~$M(u_1) = M(u_2) = M(b_2) = 1$), the only case left to consider is when~$\calQ$ is based at~$u_1$ and~$\calQ'$ is based at~$u_2$. 

    Without loss of generality, we may assume that~$\calQ$ has a side~$Q = \{u_1, a_1^\pm, b_1^\pm\}$ and~$\calQ'$ has a side~$Q' = \{u_2, a_3^\pm, b_3^\pm\}$. Suppose that~$\calR$ is incompatible with both~$\calQ$ and~$\calQ'$. This forces the base of~$\calR$ to be~$a_2$, as this is the unique vertex of~$\Delta$ which is distance two away from both~$u_1$ and~$u_2$. 

    Since~$\calR$ is incompatible with~$\calQ$, we know that no side of~$\calR$ may contain~$Q$. In particular, the elements~$\{u_1\}, \{a_1^\pm, b_1^\pm\}$ of~$\calC(a_2)$ must be on different sides of~$\calR$. This means that~$\calR$ is incompatible with the two partitions having a side of the form~$\{a_1^\times, u_1\}$. Hence~$Q$ is not the side of~$\calQ$ which is hugged. The same argument applies to~$Q'$: the elements~$\{u_2\}, \{a_3^\pm, b_3^\pm\}$ of~$\calC(a_2)$ must be on different sides of~$\calR$. 

    Therefore~$\overline{Q}$ and~$\overline{Q'}$ are both hugged. The only vertex which dominates~$u_1$ and is contained in~$\overline{Q}$ is~$a_2$ -- so the partition(s) which hug~$\overline{Q}$ are based at~$a_2$. In a hugging set for~$\overline{Q}$, there must be a partition~$\calP$ which has side~$P \subseteq \overline{Q}$ with~$u_2 \in P$. Since~$\{a_2^\pm\} \subseteq \overline{Q'}$, both sides of~$\calP$ must intersect~$\overline{Q'}$. Since~$u_2 \in Q' \cap P$, we are forced to have~$\overline{P} \cap Q' = \emptyset$. Hence~$P \supseteq \{u_2, a_3^\pm, b_3^\pm\}$. 

    We already said that the elements~$\{u_2\}, \{a_3^\pm, b_3^\pm\}$ of~$\calC(a_2)$ must be on different sides of~$\calR$. By the previous paragraph, both sides of~$\calR$ thus have non-empty intersection with~$P$. But the elements~$\{u_1\}, \{a_1^\pm, b_1^\pm\}$ of~$\calC(a_2)$ are also on different sides of~$\calR$ -- so both sides of~$\calR$ intersect~$\overline{P}$. Hence~$\calP$ is incompatible with~$\calR$, which is a contradiction. 

    Hence, even though~$\Delta$ violates Condition 2, it still satisfies the conclusion of \emph{Lemma \ref{lem:condition 2 (NEW)}}. Since~$\Delta$ is barbed, we can apply \emph{Corollary \ref{cor:conditions 1 & 2 + barbed implies vcd = M(L)}} to~$U(A_\Delta)$-equivariantly retract~$K_\Delta$ to a complex~$\widehat{K_\Delta}$ of dimension~$M(L) = 11$, concluding that~$\textsc{vcd}(U(A_\Delta)) = 11$. Hence \[\dim(K_\Delta) - \textsc{vcd}(U(A_\Delta)) = 14 - 11 = 3,\] as required. 
\end{proof}

\footnotesize
\bibliographystyle{alpha}
\bibliography{Bibliography_arXiv_v2}\medskip

\end{document}